\documentclass[11pt,a4paper]{scrartcl}

\makeatletter
\DeclareOldFontCommand{\rm}{\normalfont\rmfamily}{\mathrm}
\DeclareOldFontCommand{\sf}{\normalfont\sffamily}{\mathsf}
\DeclareOldFontCommand{\tt}{\normalfont\ttfamily}{\mathtt}
\DeclareOldFontCommand{\bf}{\normalfont\bfseries}{\mathbf}
\DeclareOldFontCommand{\it}{\normalfont\itshape}{\mathit}
\DeclareOldFontCommand{\sl}{\normalfont\slshape}{\@nomath\sl}
\DeclareOldFontCommand{\sc}{\normalfont\scshape}{\@nomath\sc}
\makeatother

\usepackage{a4wide}
\usepackage{latexsym}
\usepackage[USenglish]{babel}
\usepackage{appendix}
\usepackage[utf8]{inputenc}
\usepackage[T1]{fontenc}
\usepackage{amssymb}
\usepackage{hyperref}
\usepackage[amsthm,thmmarks]{ntheorem}
\usepackage{fixmath}

\usepackage{mathrsfs}
\usepackage{graphicx}
\usepackage{amsmath}
\usepackage{mathdots}
\usepackage{amsfonts}
\usepackage{amsxtra}
\usepackage{url}
\usepackage{mathptmx}
\usepackage{helvet}
\usepackage{stmaryrd}
\usepackage{dsfont}
\usepackage{setspace}
\usepackage{enumerate}
\usepackage{verbatim}
\usepackage{rotating}
\usepackage[authoryear,round]{natbib}
\usepackage{color}
\usepackage{cleveref}
\usepackage{cancel}

\usepackage{tikz}
\usetikzlibrary{shapes.geometric, arrows, shadows}
\usetikzlibrary{shapes.multipart}
\usetikzlibrary{mindmap,backgrounds,trees,arrows,shapes,positioning,trees,fit}
\usepackage[customcolors,shade]{hf-tikz}
\tikzstyle{startstop} = [rectangle, rounded corners, minimum width=3cm, minimum height=1cm,text centered, text width=5cm, draw=black, fill=red!10]
\tikzstyle{startstop3} = [rectangle, rounded corners, minimum width=3cm, minimum height=1cm,text centered, text width=5cm, draw=black, fill=red!10]
\tikzstyle{startstop2} = [rectangle, rounded corners, minimum width=3cm, minimum height=1cm,text centered, text width=5cm, draw=black, fill=blue!30]
\tikzstyle{startstop7} = [rectangle, rounded corners, minimum width=5cm, minimum height=1cm,text centered, text width=5.2cm, draw=black, fill=blue!30]
\tikzstyle{io} = [trapezium, trapezium left angle=70, trapezium right angle=110, minimum width=3cm, minimum height=1cm, text centered, draw=black, fill=blue!30]
\tikzstyle{process} = [rectangle, minimum width=3cm, minimum height=1cm, text centered, draw=black, fill=orange!30]
\tikzstyle{estimator} = [circle, minimum width=1.5cm, minimum height=1.5cm, text centered, draw=black, fill=green!10]
\tikzstyle{arrow} = [thick,->,>=stealth]
\tikzstyle{arrow2} = [thick,-]
\tikzstyle{myboxnone} = [draw=none, fill=none, minimum width=\textwidth]
\usepackage{epstopdf}

\def\dif{\mathrm{d}}

\def\BB{\mathcal{B}}

\def\DD{\mathcal{D}}

\def\GG{\mathcal{G}}
\def\HH{\mathcal{H}}
\def\II{\mathcal{I}}
\def\JJ{\mathcal{J}}

\def\LL{\mathcal{L}}
\def\MM{\mathcal{M}}
\def\NN{\mathcal{N}}

\def\PP{\mathcal{P}}

\def\Ind{\mathds{1}}

\def\Pas{\P\text{-a.s.}}
\def\stoch{\stackrel{\P}{\to}}

\renewcommand{\P}{\ensuremath{\mathbb{P}}}
\DeclareMathOperator*{\argmin}{arg\,min}
\newcommand{\vt}{\ensuremath{\vartheta}}

\newcommand{\E}{\ensuremath{\mathbb{E}}}

\usepackage[tablewithin=section]{caption}

\theoremstyle{plain}
\newtheorem{theorem}{Theorem}[section]
\newtheorem{lemma}[theorem]{Lemma}
\newtheorem{proposition}[theorem]{Proposition}
\newtheorem{definition}[theorem]{Definition}
\newtheorem{corollary}[theorem]{Corollary}
\newtheorem{assumptionletter}{{\textbf{Assumption}}}

\theoremstyle{definition}
\newtheorem{example}[theorem]{Example}
\newtheorem{remark}[theorem]{Remark}

\def\C{\mathbb{C}}
\def\dif{\mathrm{d}}

\def\BB{\mathcal{B}}

\def\DD{\mathcal{D}}

\def\GG{\mathcal{G}}
\def\HH{\mathcal{H}}
\def\II{\mathcal{I}}
\def\JJ{\mathcal{J}}

\def\LL{\mathcal{L}}
\def\MM{\mathcal{M}}
\def\NN{\mathcal{N}}

\def\PP{\mathcal{P}}

\def\QQQ{\mathscr{Q}}

\def\Ind{\mathds{1}}

\def\Pas{\P\text{-a.s.}}

\newcommand{\wh}{\widehat}
\newcommand{\ov}{\overline}
\newcommand{\wt}{\widetilde}

\newcommand{\R}{\ensuremath{\mathbb{R}}}
\newcommand{\Q}{\ensuremath{\mathbb{Q}}}
\newcommand{\N}{\ensuremath{\mathbb{N}}}
\newcommand{\Z}{\ensuremath{\mathbb{Z}}}
\renewcommand{\P}{\ensuremath{\mathbb{P}}}
\renewcommand{\Q}{\ensuremath{\mathbb{Q}}}

\DeclareMathOperator{\e}{e}
\DeclareMathOperator{\Var}{Var}
\DeclareMathOperator{\Cov}{Cov}

\numberwithin{equation}{section}
\allowdisplaybreaks

\renewcommand{\labelenumi}{(\roman{enumi})}

\title{Robust estimation of stationary\\[2mm] continuous-time ARMA models\\[2mm] via indirect inference}

\author{Vicky Fasen-Hartmann \setcounter{footnote}{1}\thanks{Institute of Stochastics, Englerstra{\ss}e 2,
D-76131 Karlsruhe, Germany.} \label{fnote}
 \and  Sebastian Kimmig \thanks{Württembergische Versicherung AG, Gutenbergstraße 30
D-70176 Stuttgart,  Germany. }}

\date{}

\begin{document}
%
\maketitle
%
\begin{abstract}
In this paper we present a robust estimator for the parameters of a stationary continuous-time ARMA$(p,q)$ (CARMA$(p,q)$)
process sampled equidistantly which is not necessarily Gaussian. Therefore, an indirect estimation procedure is used.
It is an indirect estimation because
we first estimate the parameters of the auxiliary AR$(r)$ representation ($r\geq 2p-1$) of the sampled CARMA process
using a generalized M- (GM-)estimator. Since the map which maps the parameters of the auxiliary AR$(r)$ representation to
the parameters of the CARMA process is not given explicitly, a separate simulation part is necessary where the parameters
of the AR$(r)$ representation are estimated from
simulated CARMA processes. Then, the parameter which takes the minimum distance
between the estimated AR parameters and the simulated AR parameters gives an estimator for the CARMA parameters.
First, we show that under some standard assumptions
the GM-estimator for the AR$(r)$ parameters is consistent and asymptotically normally distributed.
Next, we prove that the indirect estimator is consistent and
asymptotically normally distributed as well using in the simulation part the asymptotically normally distributed LS-estimator.
The indirect estimator
satisfies several important robustness properties such as weak resistance, $\pi_{d_n}$-robustness and it has a bounded influence functional.
The practical applicability of our method is demonstrated through a simulation study with replacement outliers
and compared to the non-robust quasi-maximum-likelihood estimation method.
\end{abstract}


\noindent
\begin{tabbing}
\\ \> Secondary:   60G10, 60G51
\end{tabbing}

\vspace{0.2cm}\noindent\emph{Keywords:} AR process, CARMA process, indirect estimator, influence functional, GM-estimator, LS-estimator, outlier, resistance, robustness

\section{Introduction}
The paper presents a robust estimator for the parameters of a discretely observed stationary continuous-time ARMA (CARMA) process. A  {\em weak ARMA$(p,q)$ process} in discrete-time is a weakly stationary solution of the stochastic difference equation
\begin{eqnarray} \label{ARMA}
    \phi(B)X_m=\theta(B)Z_m,\quad m\in\Z,
\end{eqnarray}
where $B$ denotes the backward shift operator (i.e. $BX_m=X_{m-1}$),
\begin{eqnarray*}
    \phi(z)=1-\phi_1z-\ldots-\phi_pz^p \quad \mbox{ and } \quad \theta(z)=1+\theta_1z+\ldots+\theta_qz^q
\end{eqnarray*}
are the autoregressive and the moving average polynomials, respectively, with $\phi_1,\ldots,\phi_p$, $\theta_1,\ldots,\theta_q\in\R$, $\phi_p,\theta_q\not=0$
and $(Z_m)_{m\in\Z}$ a \emph{weak white noise}, i.e., $(Z_m)_{m\in\Z}$ is an uncorrelated sequence with constant mean and constant variance. If $(Z_m)_{m\in\Z}$ is even an independent and identically distributed (i.i.d.) sequence then we call $(X_m)_{m\in\Z}$ a  {\em strong ARMA process}. A natural continuous-time analog of this difference equation with i.i.d. noise $(Z_m)_{m\in\Z}$ is
 the formal $p$-th order stochastic differential equation
\begin{eqnarray} \label{CARMA}
 a({ D})Y_t=c({ D}) DL_t,\quad t\in\R,
\end{eqnarray}
where ${ D}$ denotes the differential operator with respect to $t$, \[
 a(z)= z^p+a_1z^{p-1}+\ldots+a_p\qquad\text{and}\qquad c(z)=c_0z^q+c_1z^{q-1}+\ldots+c_q
\]
are the autoregressive and the moving average polynomials, respectively, with
$p>q$, and \linebreak $a_1,\,\ldots,\,a_p,\,c_0,\,\ldots,\,c_q\in\R,\,a_p,\,c_0\neq 0$. The process $(L_t)_{t\in\R}$ is a Lévy process, i.e.,
 a stochastic process with $L_0=0$ almost surely, independent and stationary increments  and almost surely c\`{a}dl\`{a}g
sample paths.
However, this is not the formal definition of a {\em CARMA$(p,q)$ process } because a Lévy process is not differentiable. The idea is more that the
differential operator on the autoregressive side act like an integration operator on the moving average side. The precise definition of a CARMA process
 is given
later. \label{pageref 2}
A rigorous foundation for CARMA$(p,0)$ processes is provided in \citet{Bergstrom:1983,Bergstrom:1984} and for CARMA$(p,q)$ processes in \citet{Brockwell2001}. A Lévy driven CARMA process can be defined  via a controller canonical state space representation.
Necessary and sufficient conditions for the existence of strictly stationary CARMA processes are given in \citet{brockwelllindner}.
From \cite{brockwelllindner} (see as well \citet{Thornton:Chambers:2017}) it is also well  known that a discretely sampled stationary CARMA process $(Y_{mh})_{m\in\Z}$ ($h>0$ fixed) admits a weak ARMA
representation, but unfortunately this is in general for Lévy driven models not a strong ARMA representation.
 For an overview and a comprehensive list of references on CARMA processes we refer to \citet{Brockwell:2014} and \citet{ChambersMcCrorieThornton}.

\label{cCARMA}In many situations it is more appropriate to specify a model in continuous time rather than in
discrete time. In recent years the interest in these models has increased with the availability
of high-frequency data in finance and turbulence but  as well by irregularly spaced data, missing observations
or situations when estimation and inference at various frequencies is to be carried out.
It is not surprising that stationary CARMA processes are applied in many areas  as, e.g., signal processing and control (cf. \cite{GarnierWang2008,LarssonMossbergSoederstroem2006}),
high-frequency financial econometrics (cf. \cite{Todorov2009})  and financial mathematics (cf. \cite{Benth:Klu:Muller:Vos,BenthKoekebakkerZakamouline2010}).
\label{litZad} The first attempts for maximum-likelihood estimation of Gaussian
stationary and non-stationary MCAR$(p)$ models are going back \cite{Harvey:Stock:1985,Harvey:Stock:1988,HarveyStock1989} and
were further explored in the well-known paper of \cite{Zadrozny1988}. \cite{Zadrozny1988} investigates continuous-time Brownian motion driven ARMAX models
and allows stocks and flows at different frequencies, and higher order integration.
There exist a few papers dealing with the asymptotic properties of parameter estimators of discretely sampled stationary CARMA models as
\cite{schlemmstelzer,brockwelldavisyang} and  for non-stationary CARMA models \citet{fasenscholz}. The papers have in common that they use a
 quasi maximum likelihood estimator (QMLE). However, it is well known  that QMLE are  sensitive
 to outliers and irregularities in the data. Hence, we are looking for an alternative robust approach.

In statistics the most fundamental question when considering {\em robustness} of an estimator is how the estimator behaves when the data
does not satisfy the model assumptions
(cf.  \cite{huberronchetti,maronna,Olive}). In the case of small deviations
from the model assumptions a robust estimator should give estimations not too far away from the estimations of the
original model. The most common and best understood robustness property is {\em distributional robustness} where the shape of the true
underlying distribution deviates slightly from the assumed model. The amount of measures for robustness is huge, e.g., qualitative robustness,
quantitative robustness, optimal robustness, efficiency  robustness and the breakdown point, to mention only a few. In contrast to the case of i.i.d. random variables,
in the  case of time series, there  exist several types of possible contamination of the data which makes it more difficult to characterize robustness.
In particular, for AR processes it is well-known that the GM-estimator (cf. \cite{boentefraimanyohai,kunsch,martin}) and the RA-estimator (cf. \cite{ben}) satisfy different robustness properties  in contrast
to $M$- or $LS$-estimators which are sensitive to the presence of additive outliers (cf. \cite{denbymartin}). However, for general ARMA models the GM-estimator and the RA-estimator are again sensitive to outliers and hence, non-robust (cf. \cite{bustos1986}).
\citet{mulerpenayohai} develop a robust estimation procedure for ARMA models by calculating the residuals of the ARMA models with the help
of BIP-ARMA models. For their result it is essential to have a strong ARMA model. Unfortunately the results can not easily be extended to weak ARMA models which we have in our context.

In this paper we use the indirect inference method originally proposed by \citet{smith} for nonlinear dynamic economic models.  That paper was extended by \citet{GallantTauchen96} and \citet{gourierouxmonfort2} (see also the overview in \citet{gourierouxmonfort}) for models with intractable likelihood functions  and moments. If the likelihood function and moments are intractable   maximum likelihood estimation  and generalized methods of moments are infeasible.
The authors applied the indirect inference method to macroeconomics, microeconomics, finance and auction models; see as well \citet{Monfort1996,PhillipsYu2009} for applications to
continuous-time models, \citet{GourierouxRenaultTouzi,Kypriacou} for applications to time series models and \citet{Monfardini1998} for applications to stochastic
volatility models. 
In addition indirect inference is used for bias reduction in finite samples as, e.g., in \citet{GourierouxRenaultTouzi,GouPhillipsYug,Yu2011SimulationBasedEM,Kypriacou,Cklu:Haug:Thiago}. \label{pbias} 
Our motivation for the indirect inference method is robust estimation (cf. \citet{delunagenton,delunagenton2,Kypriacou}).
An alternative approach for bias correction is given in \citet{WangPhillipsYu} for univariate and multivariate diffusion models.
For estimators of the mean revision parameter based on the Euler
approximation and  the trapezoidal approximation for discretization the authors calculate the bias  and relate it to the estimation bias and discretization  bias.
\label{pagerefWang}

 The core idea of the indirect estimation method is to avoid estimating the parameters of interest directly and instead fit an auxiliary model to the data, estimate the parameters
 of this auxiliary model and then  use this estimates with simulated data to construct an estimator for the original parameter of interest (see \citet{delunagenton} for a schematic overview over the indirect estimation method). \citet{delunagenton,delunagenton2} recognized that it is possible to construct robust estimators via this approach, even for model classes where direct robust estimation is difficult. The reason is
that it is sufficient if the parameters of the auxiliary model are estimated by a robust estimation method. Therefore, \citet{delunagenton} present an indirect estimation procedure for
strong ARMA processes (without detailed assumptions and rigorous proofs). They fit an AR$(r)$ process to the ARMA model and estimate the parameters of the AR$(r)$ process with a GM-estimator.
We present a similar approach in our paper for the estimation of the CARMA parameters. Since the discretely sampled stationary CARMA process admits a
weak ARMA representation instead of a strong ARMA representation several proofs have to be added and identifiability issues have to be taken into account.

The paper is structured as follows. In \Cref{sec:Preliminaries}, we first present our parametric family of stationary CARMA processes
and our model assumptions. Furthermore, we motivate that for any $r\geq 2p-1$ any stationary CARMA process has an AR$(r)$ representation.
Then, in \Cref{sec:indirect estimation}, we introduce the indirect estimation procedure and give sufficient criteria for indirect estimators
to be consistent and asymptotically normally distributed independent of the model; we have to assume at least consistent and asymptotically normally distributed estimators  in the estimation part and
in the simulation part of the indirect estimation method.
Since the auxiliary AR$(r)$ parameters of the sampled CARMA process are estimated by a GM-estimator we give an introduction into  GM-estimators
 in \Cref{sect:GM} and derive consistency and asymptotic normality of this estimator in our setup.
 Moreover, we see that the GM-estimator is still asymptotically
 normally distributed for CARMA processes with outliers as additive outliers and replacement outliers. Our conclusions extend the results of \citet{bustos}.
Finally, in \Cref{sec:indirect estimator CARMA}, we are able to show that the indirect estimator for the parameters of the discretely observed stationary CARMA process is consistent and asymptotically normally distributed  using in the estimation part a GM-estimator
and in the simulation part a LS-estimator. Several robustness properties of this estimator are derived as well as
qualitative robustness and a bounded influence functional.
After all, the simulation part, in \Cref{sect:Simulation}, shows the practical applicability of our indirect estimator
and its robustness properties. We compare our estimator with the non-robust QMLE.
Conclusions are given in \Cref{sec:Conclusion}.
The paper ends with the proofs of the results in \Cref{sec:Proofs}.
\vspace*{-0.2cm}


\subsubsection*{Notation} \vspace*{-0.2cm}
We use as norms the Euclidean norm $ \lVert\cdot \rVert$ in $\R^d$ and its operator $\lVert\cdot \rVert$  in $\R^{m\times d}$ which is submultiplicative.
 For a matrix $A\in\R^{m\times d}$ we denote by $A^T$ its transpose.
For a matrix function $f(\vt)$ in $\R^{m\times d}$ with $\vt\in\R^s$ the gradient with respect to the parameter vector $\vt$
 is $\nabla_{\vartheta} f(\vt)=\frac{\partial\text{vec}(f(\vt))}{\partial\vartheta^T}\in\R^{dm\times s}$ and similarly
 $\nabla^2_{\vartheta} f(\vt)=\frac{\partial\text{vec}(\nabla_\vt f(\vt))}{\partial\vartheta^T}\in\R^{dms\times s}$.
Finally,
we write $\stackrel{\DD}{\longrightarrow}$ for weak convergence and $\stoch$ for convergence in probability. In general $C$ denotes a constant which may change from line to line.

\section{Preliminaries} \label{sec:Preliminaries} \vspace*{-0.2cm} 

\subsection{The CARMA model} \vspace*{-0.2cm} 

In this paper we consider  a parametric family of stationary CARMA processes. Let $\Theta\subseteq \R^{N(\Theta)}$ ($N(\Theta)\in\N$) be a parameter space, $p\in\N$ be fixed and for any
$\vt\in\Theta$ let  $a_1(\vt),\,\ldots,\,a_p(\vt),$ \linebreak $\,c_0(\vt),\,\ldots,\,c_{p-1}(\vt)\in\R,\,a_p(\vt)\neq 0$ and   $c_j(\vt)\not=0$ for some $j\in\{0,\ldots,p-1\}$.
Furthermore, define
\begin{eqnarray*} \label{A}
 A_\vt&:=&
 \begin{pmatrix}
  0      & 1            & 0            & \ldots & 0            \\
  0      & 0            & 1            & \ddots & \vdots       \\
  \vdots & \vdots       & \ddots       & \ddots & 0            \\
  0      & 0            & \ldots       & 0      & 1            \\
  -a_p(\vt)   & -a_{p-1}(\vt)     & \ldots       & \ldots & -a_1(\vt)
 \end{pmatrix}\in \R^{p\times p},
 \end{eqnarray*}
 \begin{eqnarray*}
 q(\vt)&=&\sup\{j\in\{0,\ldots,p-1\}:\,c_l(\vt)=0\,\,\forall\, l>j\} \quad \text{ with }\quad \sup\emptyset:=p-1,\\
 c_\vt&:=&(c_{q(\vt)}(\vt),\,c_{q(\vt)-1}(\vt),\,\ldots,\,c_{0}(\vt),0,\ldots,0)^T\in\R^p.
\end{eqnarray*}
The CARMA process $(Y_t(\vt))_{t\in\R}$ is then defined  via the controller canonical state space representation:
Let $(X_t(\vt))_{t\in\R}$ be a strictly stationary solution to the stochastic differential equation
\begin{subequations} \label{Equation 1.1} \label{2.1}
 \begin{equation} \label{Equation 1.1a}
  {\rm d}X_t(\vt) = A_\vt X_t(\vt)\,{\rm d}t + e_p\,{\rm d}L_t,\quad t\in\R,
 \end{equation}
 where $e_p$ denotes the $p$-th unit vector in $\R^p$. Then the process
 \begin{equation} \label{Equation 1.1b}
  Y_t(\vt) := c_\vt^T X_t(\vt),\quad t\in\R,
 \end{equation}
\end{subequations}
 is said to be a \textit{(stationary) CARMA process} of order $(p,\,q(\vt))$.  Rewriting \eqref{Equation 1.1} line by line $(Y_t(\vt))_{t\in\R}$ can be interpreted
 as solution of the differential equation \eqref{CARMA}; see \citet{Brockwell2001,MarquardtStelzer2007}. \label{pageref 3}
 This means that
 in our parametric family of CARMA processes the order of the autoregressive polynomial is fixed to $p$ but the order of the moving average polynomial
 $q(\vt)$ may change. In addition, we investigate only stationary CARMA processes.

 Furthermore, we have the discrete-time observations  $Y_h,\ldots,Y_{nh}$ of the CARMA process $(Y_t)_{t\in\R}=(Y_t(\vt_0))_{t\in\R}$
 with fixed grid distance $h>0$. Hence, the true model parameter is $\vt_0$.
 The aim of this paper is to receive from the observations $Y_h,\ldots,Y_{nh}$ an estimator for $\vt_0$.
 Throughout the paper we will assume that the following \autoref{as_D} holds.

\setcounter{assumptionletter}{0}
\begin{assumptionletter}\strut
         \label{as_D}
         \renewcommand{\theenumi}{(A.\arabic{enumi})}
         \renewcommand{\labelenumi}{\theenumi}
         \begin{enumerate}
                 \item \label{as_D1} The parameter space $\Theta$ is a compact subset of $\R^{N(\Theta)}$.
                 \item \label{as_D8} The true parameter $\vt_0$ is an element of the interior of $\Theta$.
                 \item \label{as_D2}  $\E [L_1] = 0$, $0<\E L_1 ^2=\sigma_L^2 < \infty$ and there exists a $\delta > 0$ such
                    that $\E| L_1 |^{4+\delta} < \infty$.
                 \item \label{as_D3} The eigenvalues of $A_\vt$ have strictly negative real parts.
                 \item \label{as_D5} For all $\vt \in \Theta$ the zeros of $c_\vt(z)=c_0(\vt)z^{q(\vt)}+c_1(\vt)z^{q(\vt)-1}+\ldots +c_{q(\vt)}$ are different from the eigenvalues of $A_\vt$.
                 \item \label{as_D6} For any $\vt,\vt'\in\Theta$ we have $(c_\vt,A_\vt)\not=(c_{\vt'},A_{\vt'})$. 
                 \item \label{as_D7} For all $\vt \in \Theta$ the spectrum of $A_\vt$ is a subset of $\{ z \in \C: -\frac{\pi}{h} < \text{Im}(z) < \frac{\pi}{h} \}$ where $\text{Im}(z)$ denotes the imaginary part of $z$.
                 \item \label{as_D9} The maps $\vt\mapsto A_\vt$ and $\vt\mapsto c_\vt$ are three times continuous differentiable.
         \end{enumerate}
\end{assumptionletter}

\begin{remark} \label{Remark 2.1}
$\mbox{}$
\begin{enumerate}
    \item \ref{as_D1} and \ref{as_D8} are standard assumptions in point estimation theory.
    \item \ref{as_D3} guarantees that there exists a stationary solution of the state process \eqref{Equation 1.1a} and hence, a stationary CARMA process $(Y_t(\vt))_{t\in\R}$ (see \citet{MarquardtStelzer2007}). For this reason we can and will assume throughout the paper that $(Y_t(\vt))_{t\in\R}$ is stationary.
        The assumption of a stationary CARMA process $(Y_t(\vt))_{t\in\R}$ is essential for the indirect estimation approach of this paper.
    \item A consequence of \ref{as_D3},  \ref{as_D9}, the compactness of $\Theta$  and the fact that the eigenvalues of a
        matrix are continuous functions of its entries (cf. \citet[Fact 10.11.2]{bernstein}) is $\sup_{\vt\in\Theta}\max\{|\lambda|:\lambda \text{ is eigenvalue of }\e^{A_\vt}\}<1$ and hence, $\sup_{\vt\in\Theta}\|\e^{A_\vt u}\|\leq C\e^{-\rho u}$ for some $C,\rho>0$.
    \item Due to \ref{as_D5} the state space representation \eqref{Equation 1.1} of the CARMA process is minimal (cf. \citet[Proposition 12.9.3]{bernstein} and
            \citet[Theorem 2.3.3]{hannandeistler}).
    \item A consequence of \ref{as_D5} and \ref{as_D6} is that the family of stationary CARMA processes $(Y_t(\vt))_{t\in\R}$ is identifiable from their spectral densities and in combination with
        \ref{as_D7} that the same is true for the discrete-time process  $(Y_{mh}(\vt))_{m\in\Z}$ (cf. \citet[Theorem 3.13]{schlemmstelzer}).
    \item The CARMA process has to be sampled sufficiently finely to ensure that \ref{as_D7} holds  so that the parameters can be identified from the discrete data.
\end{enumerate}
\end{remark}

In the following we denote the autocovariance function of the stationary CARMA process $(Y_t(\vt))_{t\in\R}$ as $(\gamma_\vt(t))_{t\in\R}$ which has by
\citet[Proposition 3.1]{schlemmstelzer}  the form
\begin{eqnarray} \label{ACF}
    \gamma_{\vt}(t) =\Cov(Y_{s+t}(\vartheta),Y_s(\vartheta))= c_\vt^T \e^{A_\vt t} \Sigma_\vt c_\vt, \quad s\in \R,\,t\geq 0,
\end{eqnarray}
with $\Sigma_\vt =  \sigma_{L}^2\int_0^\infty \e^{A_\vt u}e_p  e_p^T \e^{A_\vt u} \dif u$. Due to \autoref{as_D} the autocovariance function is three times
continuous differentiable as well.

\subsection{The AR$(r)$ representation of a stationary CARMA process}\label{sec:IndEst}

First, we define the auxiliary AR($r$) representation of the sampled CARMA process $(Y_{mh}(\vt))_{m \in \Z}$.
\begin{proposition}\label{auxiliaryprop}
For every $\vt \in \Theta$ and every $r \geq 2p-1$, there exists  a unique
$$\pi(\vt) := (\pi_{1}(\vt), \ldots, \pi_{r}(\vt), \sigma(\vt))\in\R^{r}\times\left[0,\infty\right)$$
such that
\begin{equation}\label{chap8:AuxRep}
 U_{m}(\vt) := Y_{mh}(\vt) - \sum_{k=1}^r \pi_{k}(\vt) Y_{(m-k)h}(\vt)
\end{equation}
is stationary with $\E[ U_{1}(\vt)] = 0$, $\Var( U_{1}(\vt)) = \sigma^2(\vt)$ and
\begin{equation}\label{chap8:DefAuxParam2}
\E \left[ U_{m}(\vt)  Y_{(m-k)h}(\vt) \right] = 0 \quad \text{ for } k=1, \ldots, r.
\end{equation}
We call $\pi(\vt) $ the \emph{auxiliary parameter} of the AR$(r)$ representation of $(Y_{mh}(\vt))_{m \in \Z}$.
\end{proposition}

\begin{remark}\label{piremark}
$U_{m}(\vt)$ can be interpreted as the error of the best linear predictor of $Y_{\vt}(mh)$ in terms of \linebreak $Y_{(m-1)h}(\vt), \ldots, Y_{(m-r)h}(\vt)$.
Per construction, however, the sequence  $(U_{m}(\vt))_{m \in \Z}$ is not an uncorrelated sequence, $U_{m}(\vt)$ is only uncorrelated with $Y_{(m-1)h}(\vt), \ldots, Y_{(m-r)h}(\vt)$.
\end{remark}

\begin{definition}
Let $\Pi \subseteq \R^{r+1}$ be the parameter space containing all possible parameter vectors of stationary  AR($r$) processes. The map $\pi:\Theta\to\Pi$ with
$\vt \mapsto \pi(\vt)$ and $\pi(\vt)$ as given in \autoref{auxiliaryprop} is called the \emph{link function} or binding function.
\end{definition}

\begin{lemma}\label{defbindfct}
Let  $r \geq 2p-1$. Then, $\pi(\vt)$ is injective and three times continuously differentiable.
\end{lemma}

Finally, due to \Cref{defbindfct} we suppose throughout the paper:

\begin{assumptionletter} \label{Assumption B}
Let $r \geq 2p-1$.
\end{assumptionletter}

\section{Indirect estimation}\label{sec:indirect estimation}

For fixed $r$, denote by $\widehat{\pi}_n$ an estimator of $\pi(\vt_0)$ that is calculated from the observations \linebreak $\mathcal{Y}^n=(Y_{h},\ldots,Y_{nh})$.
If we were able to analytically invert the link function $\pi$  and calculate $\pi^{-1}(\widehat{\pi}_n)$, then $\pi^{-1}(\widehat{\pi}_n)$ would be an estimator for $\vt_0=\pi^{-1}(\pi(\vt_0))$. However, this is not possible in general since no analytic representation of $\pi^{-1}$ exists.
To overcome this problem, we  perform a second estimation, which is based on simulations, and constitutes the other building block of indirect estimation.
We fix a number $s \in \N$ and simulate a sample path of length $s n$ of a L\'{e}vy process $(L_t^S)_{t \in \R}$ with $\E L_1^S=0$ and $\E(L_1^S)^2=\sigma_L^2$. Then, for a fixed parameter $\vt \in \Theta$ we generate a sample path of the associated CARMA process $(Y_t^S(\vt))_{t\in\R}$ using the simulated path $(L^S_t)_{t\in\R}$. This gives us a vector of ``pseudo--observations'' $\mathcal{Y}_S^{sn}(\vt)=(Y_{h}^S(\vt), \ldots, Y_{snh}^S(\vt))$ of length $sn$. From this observation $\mathcal{Y}_S^{sn}(\vt)$ we estimate again $\pi(\vt)$ by an estimator
$ \wh{\pi}^{\text{S}}_{sn}(\vt)$.
 The idea is now to choose that value of $\vt$ as estimator for $\vt_0$ which minimizes a suitable distance between $\wh{\pi}_n$ and $\wh \pi^{\text{S}}_n(\vt)$. The formal definition is as follows.

\begin{definition}\label{DefIndEst}
Let $\wh{\pi}_n$ be an estimator for $\pi(\vt_0)$ calculated from the data $\mathcal{Y}^n$, let $\wh{\pi}^{\text{S}}_{sn}(\vt)$ be an estimator for $\pi(\vt)$ calculated from the pseudo--observations $\mathcal{Y}_S^{sn}(\vt)=(Y_{h}^S(\vt), \ldots, Y_\vt^S(snh))$ and let $\Omega\in\R^{N(\Theta)\times N(\Theta)}$ be a symmetric positive definite weighting matrix. The function $ \LL_{\text{Ind}} : \Theta \to \left[0, \infty\right)$ is defined as
\begin{equation*}\label{chap8:DefLLInd}
\LL_{\text{Ind}} ( \vt, \mathcal{Y}^n ) :=  [\widehat{\pi}_n - \wh{\pi}^{\text{S}}_{sn}(\vt) ]^T \Omega [\widehat{\pi}_n - \wh{\pi}^{\text{S}}_{sn}(\vt) ].
\end{equation*}
Then, the indirect estimator for $\vt_0$ is
\begin{equation*}\label{chap8:DefCheckVT}
\wh{\vt}_n^{\text{Ind}} = \argmin_{\vt \in \Theta} \LL_{\text{Ind}} ( \vt, \mathcal{Y}^n ).
\end{equation*}
\end{definition}

We are able to present general conditions under which this indirect estimator is consistent and asymptotically normally distributed.


\begin{theorem}\label{chap8:IndEstThm}  $\mbox{}$
\begin{itemize}
\item[(a)] Suppose that the following assumptions are satisfied:
 \renewcommand{\theenumi}{(C.\arabic{enumi})}
         \renewcommand{\labelenumi}{\theenumi}
\begin{enumerate}
\item \label{chap8:UnifConv1}
$\wh \pi_n \stoch \pi(\vt_0)$ as $n\to\infty$.
\item \label{chap8:UnifConv}
$\sup_{\vt \in \Theta} \| \wh{\pi}_n^{\text{S}}(\vt) - \pi(\vt) \| \stoch 0
$  as $n\to\infty$.
\end{enumerate}
Define the map
\begin{equation}
\begin{split}\label{mathscrind}
\mathscr{Q}_{\text{Ind}}: \Theta \to \left[0, \infty\right)  \quad \text{ as } \quad
\vt \mapsto [\pi(\vt)-\pi(\vt_0)]^T \Omega [\pi(\vt)-\pi(\vt_0)].
\end{split}
\end{equation}
Then
\begin{equation*}\label{chap8:StrongCons}
    \sup_{\vt \in \Theta} | \LL_{\text{Ind}}(\vt, \mathcal{Y}^n) - \mathscr{Q}_{\text{Ind}}(\vt)|\stoch 0 \quad \mbox{ and }\quad \wh{\vt}_n^{\text{Ind}} \stoch  \vt_0.
\end{equation*}
If we replace in \ref{chap8:UnifConv1} and \ref{chap8:UnifConv} convergence in probability by almost sure convergence then we can replace in the statement
convergence in probability by almost sure convergence as well.
\item[(b)] Assume additionally to \ref{chap8:UnifConv1} and \ref{chap8:UnifConv}:
\begin{enumerate}
\setcounter{enumi}{2}
\item \label{ThmIndEstAssumption1}
$
\sqrt{n} ( \wh{\pi}_n^{\text{S}}(\vt) - \pi(\vt)) \stackrel{\DD}{\longrightarrow} \NN(0, \Xi_{\text{S}}(\vt))$ as $ n \to \infty$
for any  $\vt \in \Theta$.
\item \label{ThmIndEstAssumption2}
$
\sqrt{n} (\wh \pi_n - \pi(\vt_0) ) \stackrel{\DD}{\longrightarrow} \NN( 0, \Xi_{\text{D}}(\vt_0))$ as $ n \to \infty$.
\item \label{ThmIndEstAssumption3} For any sequence $(\overline \vt_n)_{n \in \N}$ with $\overline \vt_n \stoch \vt_0$ as $n\to\infty$ the asymptotic behaviors
\begin{eqnarray*}
\nabla_\vt \wh \pi_n^{\text{S}}(\overline \vt_n)  &\stoch & \nabla_\vt \pi(\vt_0),\\
\nabla_\vt^2 \wh \pi_n^{\text{S}}(\overline \vt_n) &=&O_P(1),
\end{eqnarray*}
hold as $n\to\infty$ and $\nabla_\vt \pi(\vt_0)$ has full column rank $N(\Theta)$.
\end{enumerate}
Then, as $n\to\infty$,
\begin{equation*}\label{chap8:AsNormInd}
\sqrt{n} (\wh{\vt}_n^{\text{Ind}} - \vt_0) \stackrel{\DD}{\longrightarrow} \NN(0, \Xi_{\text{Ind}}(\vt_0)),
\end{equation*}
where
$$\Xi_{\text{Ind}}(\vt_0) = \JJ_{\text{Ind}}( \vt_0)^{-1} \II_{\text{Ind}}( \vt_0)\JJ_{\text{Ind}}( \vt_0)^{-1}$$
with
\begin{eqnarray*}
\JJ_{\text{Ind}}( \vt_0) &=& [\nabla_\vt \pi(\vt_0)]^T \Omega [\nabla_\vt \pi(\vt_0)] \quad \text{ and }\quad\\
\II_{\text{Ind}}(\vt_0) &=& [\nabla_\vt \pi(\vt_0)]^T \Omega \left[ \Xi_{\text{D}}(\vt_0) + \frac1s \Xi_{S}(\vt_0) \right] \Omega [\nabla_\vt \pi(\vt_0)].
\end{eqnarray*}
\end{itemize}
\end{theorem}

\citet{gourierouxmonfort2} develop for a dynamic model as well the consistency and the asymptotic normality of the indirect estimator but under different assumptions
mainly based on $\LL_{\text{Ind}} ( \vt, \mathcal{Y}^n )$ (see as well \citet{smith}). These results are again summarized in \citet{gourierouxmonfort}.
In the context of indirect estimation of ARMA models, \citet[p.22]{delunagenton} mention  the asymptotic normality of their indirect estimator
but without stating any regularity conditions and only referring to \citet[Proposition 4.2]{gourierouxmonfort}.

\begin{remark}\label{IndEstThm:Remark} \strut
\begin{itemize}
\item[(a)] The asymptotic covariance matrix can  be written as
$$\Xi_{\text{Ind}}(\vt_0) = \HH( \vt_0)  \left( \Xi_{\text{D}}(\vt_0) + \frac1s \Xi_{S}(\vt_0) \right) \HH(\vt_0)^T,$$
where
$\HH(\vt_0) = [ \nabla_\vt \pi(\vt_0)^T \Omega \nabla_\vt \pi(\vt_0)]^{-1} [\nabla_\vt \pi(\vt_0)]^T  \Omega.$
This is the analog  form of \linebreak \citet[Eq. (4.4)]{delunagenton}.
\item[(b)] Note that the asymptotic results hold for any $r \geq 2p-1$. But increasing the auxiliary AR order does not necessarily yield better results.
 On the other hand, increasing $s$ increases the efficiency. For $s \to \infty$ we receive $\Xi_{\text{Ind}}(\vt_0) \to \HH( \vt_0)  \Xi_{\text{D}}(\vt_0) \HH(\vt_0)^T$.
 The best efficiency is received for $\Omega = [\Xi_{\text{D}}(\vt_0)]^{-1}$ in which case
$\Xi_{\text{Ind}}(\vt_0)\stackrel{s\to\infty}{\to} \left[ \nabla_\vt \pi(\vt_0)^T  \Xi_{\text{D}}(\vt_0)^{-1} \nabla_\vt \pi(\vt_0)\right]^{-1}.$
\end{itemize}
\end{remark}

\begin{remark} \label{Remark 2.4}
A fundamental assumption for \Cref{auxiliaryprop} is  \ref{as_D3} resulting in the existence of stationary CARMA processes. In particular, in the case of integrated CARMA
processes $(Y_t(\vt))_{t\in\R}$, where $A_\vt$  has eigenvalue $0$, the result of  \Cref{auxiliaryprop} does not hold in general.
For this reason the indirect estimation approach of this paper can not be extended to integrated CARMA processes which are non-stationary.
Even for integrated CARMA processes it is well known that estimators for  the parameter determining the integration have a $n$
convergence instead of a $\sqrt{n}$ convergence (cf. \cite{Chambers:McCrorie:2007,fasenscholz,ChambersMcCrorieThornton}).
\end{remark}

\begin{remark} \label{Remark 2.5}
The discretely observed stationary CARMA$(p,q(\vt))$ process $(Y_{mh}(\vt))_{m \in \Z}$ admits a representation as a stationary ARMA$(p,p-1)$ process with weak white noise  of the form
\begin{equation}\label{chap8:ARMAP}
    \phi(B) Y_{mh}(\vt) = \theta(B) \epsilon_{m}(\vt),
\end{equation}
where  $\phi(z) = \prod_{i=1}^{p} (1 - \e^{h \lambda_i}z)$ (the $\lambda_i$ being the eigenvalues of $A_\vt$), $\theta(z)$ is a monic, Schur--stable polynomial
and $(\varepsilon_m(\vt))_{m\in\Z}$  is a weak white noise (see \citet[Lemma 2.1]{brockwelllindner}), i.e.  $(Y_{mh}(\vt))_{m \in \Z}$ is a weak ARMA$(p,p-1)$ process. Such an exact discrete-time ARMA representation for multivariate CARMA processes was generalized in \cite[Theorem 1]{Thornton:Chambers:2017} to possible
non-stationary multivariate CARMA processes. Thus, it is as well possible to do an indirect estimation procedure by estimating the parameters of the discrete-time ARMA$(p,p-1)$
representation, e.g., using maximum-likelihood, instead of estimating the parameters of the auxiliary AR$(r)$ model. Then the map $\pi$ is replaced by the
map $\pi_1$ which maps the parameters of the CARMA process to the coefficients of the weak ARMA($p, p-1$) representation of its sampled version  \eqref{chap8:ARMAP}.
Using $\pi_1(\vt)$ instead of $\pi(\vt)$ in \Cref{chap8:IndEstThm},
\Cref{chap8:IndEstThm} can be adapted under the same assumptions giving asymptotic normality of the indirect estimator based on the discrete-time ARMA representation
of the CARMA process. In particular,
it is as well possible to derive an estimation procedure for non-stationary CARMA processes. However, until now there does not exist robust estimators for the parameters of
weak ARMA processes such that this approach does not give robust estimators for the parameters of the stationary CARMA process, which is the topic of this paper.
\end{remark}

\section{Estimating the auxiliary AR$(r)$ parameters of a CARMA process with outliers}\label{sect:GM}

In order to apply the indirect estimator to a discretely sampled stationary CARMA process we need strongly consistent and asymptotically normally distributed estimators for the parameters
of the auxiliary  AR$(r)$ representation. In this section we will study  generalized M- (GM-) estimators.
The GM-estimator will be applied to a stationary CARMA process afflicted by outliers because we want to study some robustness properties of our estimator as well.
Outliers can be thought as typical observations that do not arise because of the model structure but due to some external influence, e.g., measurement errors. Therefore, a whole sample of observations which contains outliers does not come from the true model anymore but it is still close to it as long as the total number of outliers is not overwhelmingly large.

\begin{definition}\label{definition outliers}
Let $g: [0,1] \to [0,1]$ be a function that satisfies $g(\gamma) - \gamma = o(\gamma)$ for $\gamma \to 0$. Let $(V_m)_{m \in \Z}$ be a stochastic process taking only the values $0$ and $1$ with
\begin{equation*}\label{chap8:gammaProb}
\P( V_m = 1) = g(\gamma)
\end{equation*}
 and let $(Z_m)_{m \in \Z}$ be a real-valued stochastic process.
 The {\em disturbed process} $(Y_{mh}^\gamma(\vt))_{m \in \Z}$ is defined as
\begin{equation}\label{obs:conta}
  Y_{mh}^{\gamma}(\vt) = (1- V_m) Y_{mh}(\vt) + V_m Z_m.
\end{equation}
\end{definition}
The disturbed process  $(Y_{mh}^\gamma(\vt))_{m \in \Z}$ is in general not a sampled CARMA process anymore.
\begin{remark}\label{OutRemark} $\mbox{}$
\begin{itemize}
\item[(a)] The interpretation of this model is that at each point $m \in \Z$ an outlier is observed with probability $g(\gamma)$ while the true value $Y_{mh}(\vt)$ is observed with probability $1-g(\gamma)$. The model has the advantage that one can obtain both additive and replacement outliers by choosing the processes $(Z_m)_{m \in \Z}$ and $(V_m)_{m \in \Z}$ adequately. Specifically, to model {\em replacement outliers}, one assumes that $(Z_m)_{m \in \Z}$, $(V_m)_{m \in \Z}$ and $(Y_{mh}(\vt))_{m \in \Z}$ are jointly independent. Then, if the realization of $V_m$ is equal to $1$, the value $Y_{mh}(\vt)$ will be replaced by the realization of $Z_m$ justifying the use of the name replacement outliers.
    On the other hand, modeling {\em additive outliers} can be achieved by taking $Z_m = Y_{mh}(\vt) + W_m$ for some process $(W_m)_{m \in \Z}$ and assuming that $(Y_{mh}(\vt))_{m \in \Z}$ is independent from $(V_m)_{m \in \Z}$. Then we have $Y_{mh}^\gamma(\vt) = Y_{mh}(\vt) + V_m W_m$
such that the realization of $W_m$ is added to the realization of $Y_{mh}(\vt)$ if $V_m$ is $1$.
\item[(b)] Another advantage of this general outlier model is that one can easily model the temporal structure of outliers. On the one hand,  if $(V_m)_{m \in \Z}$ is chosen as an i.i.d. sequence with $\P(V_m = 1)=\gamma$, then outliers typically appear {\em isolated}, i.e., between two outliers there is usually a period of time where no outliers are present. On the other hand, one can also model {\em patchy outliers} by letting $(B_m)_{m \in \Z}$ be an i.i.d. process of Bernoulli random variables with success probability $\epsilon$ and setting $V_m = \max( B_{m-l}, \ldots, B_m)$ for a fixed $l \in \N$. Then as $\epsilon \to 0$,
$$\P( V_m = 1 ) = 1 - (1-\epsilon)^l = l \epsilon + o(\epsilon),$$
 which results in $\gamma = l \epsilon$. For $\epsilon$ sufficiently small, outliers then appear in a  block of size $l$.
\end{itemize}
\end{remark}

Recall the following notion:

\begin{definition}
A  stationary stochastic process $Y=(Y_t)_{t\in I}$ with $I=\R$ or $I=\Z$ is
called strongly (or $\alpha$-) mixing if \label{strong:mixing}
\begin{eqnarray*}
    \alpha_l:=\sup\left\{|\P(A\cap B)-\P(A)\P(B)|:\, A\in\mathcal{F}_{-\infty}^0,\,B\in\mathcal{F}_l^{\infty}\right\}\stackrel{l\to\infty}{\to}0
\end{eqnarray*}
where $\mathcal{F}_{-\infty}^0=\sigma(Y_t:\,t\leq 0)$ and $\mathcal{F}_{l}^\infty=\sigma(Y_t:\,t\geq l)$. If $\alpha_l\leq C\alpha^l$
for some constants $C>0$ and $0<\alpha<1$ we call $Y=(Y_t)_{t\in I}$
exponentially strongly  mixing.
\end{definition}


\setcounter{assumptionletter}{3}
\begin{assumptionletter}\strut
         \label{as_G}
         \renewcommand{\theenumi}{(D.\arabic{enumi})}
         \renewcommand{\labelenumi}{\theenumi}
         \begin{enumerate}
\item \label{as_G1}
The processes
$(V_m)_{m \in \Z}$ and $(Z_m)_{m \in \Z}$ are strictly stationary with
$\E | V_1 | < \infty$ and $\E | Z_1| < \infty$.
\item \label{as_G2} Either we have the replacement model where the processes $(Y_{mh}(\vt))_{m\in \Z}$, $(V_m)_{m \in \Z}$ and $(Z_m)_{m \in \Z}$ are jointly independent, and
$(V_m)_{m \in \Z}$ and $(Z_m)_{m \in \Z}$ are exponentially strongly mixing, i.e., \linebreak
$\alpha_V(m) \leq C \rho^m$ and $\alpha_Z(m) \leq C \rho^m$ for some $C > 0$, $\rho \in (0,1)$ and every $m \in \N$. Or we have the additive model with $Z_m=Y_{mh}(\vt)+W_m$ where
the processes $(Y_{mh}(\vt))_{m\in \Z}$, $(V_m)_{m \in \Z}$ and $(W_m)_{m \in \Z}$ are jointly independent, and
$(V_m)_{m \in \Z}$ and $(W_m)_{m \in \Z}$ are exponentially strongly mixing.
\item \label{as_G4} For all $a \in \R, \pi \in \R^{r}$ with $|a| + \| \pi \| > 0$:
$$\P ( aY_{(r+1)h}^\gamma(\vt) + \pi_1Y_{rh}^\gamma(\vt) + \ldots + \pi_rY_h^\gamma(\vt) = 0) = 0.$$
\end{enumerate}
\end{assumptionletter}
We largely follow the ideas of \citet{bustos} for the GM-estimation of AR$(r)$ parameters, however our model and our assumptions are slightly different.
\autoref{as_G} corresponds to \citet[Assumption (M2),(M4),(M5)]{bustos}.  The main difference is that the sampled stationary CARMA process $(Y_{mh})_{m \in \Z}$ is in  \citet{bustos}
 an infinite-order moving average process whose noise is  $\Phi$--mixing which is in general not satisfied for a sampled stationary CARMA process. However, we already know from \citet[Proposition~3.34]{MarquardtStelzer2007} that a stationary CARMA process is exponentially strongly mixing which is weaker than $\Phi$-mixing. Therefore, we assume that
   $(V_m)_{m \in \Z}$, $(Z_m)_{m \in \Z}$ and $(W_m)_{m \in \Z}$  are exponentially strongly mixing instead of $\Phi$-mixing as in \cite{bustos}.

  In the following we define GM-estimators. Let two functions $\phi: \R^{r} \times \R \to \R$ and $\chi: \R \to \R$ be given.
   Moreover, assume that we have observations $\mathcal{Y}^{n,\gamma}(\vt)=(Y_h^\gamma(\vt), Y_{2h}^\gamma(\vt), \ldots, Y_{nh}^\gamma(\vt))$ from the disturbed process  in \eqref{obs:conta}. The parameter
 $$\pi^{\text{GM}}(\vt^\gamma) = (\pi_{1}^{\text{GM}}(\vt^\gamma), \ldots, \pi_{r}^{\text{GM}}(\vt^\gamma), \sigma^{\text{GM}}(\vt^\gamma))$$
  is defined as the solution of the equations
  \begin{subequations} \label{pseudo}
\begin{eqnarray}
\E \left[ \phi \left( \begin{pmatrix}Y_h^\gamma(\vt) \\ \vdots \\ Y_{rh}^\gamma(\vt) \end{pmatrix},  \frac{Y_{(r+1)h}^\gamma(\vt) - \pi_{1} Y_{rh}^\gamma(\vt) -  \ldots - \pi_{r}Y_h^\gamma(\vt)}{\sigma} \right) \begin{pmatrix}Y_h^\gamma(\vt) \\ \vdots \\ Y_{rh}^\gamma(\vt) \end{pmatrix} \right] &=& 0,  \label{pseudo:pi} \\
\E \left[ \chi \left( \left( \frac{Y_{(r+1)h}^\gamma(\vt) - \pi_{1} Y_{rh}^\gamma(\vt) -  \ldots - \pi_{r}Y_h^\gamma(\vt)}{\sigma} \right)^2 \right) \right] &= &0 \label{pseudo:sigma}
\end{eqnarray}
\end{subequations}
for $(\pi_1,\ldots,\pi_r,\sigma)\in\R^r\times\left(0,\infty\right)$.
 The idea is again that these are the parameters of the auxiliary AR representation of $(Y_{mh}^\gamma(\vt))_{m\in\Z}$. Note that $\pi^{\text{GM}}(\vt^\gamma)$  depends on the processes $(V_m)_{m \in \Z}$ and $(Z_m)_{m \in \Z}$ as well. We choose not to indicate this in the notation to make the exposition more readable. For the uncontaminated
 process  $(Y_{mh}(\vt))_{m\in\Z}$ we also write $\pi^{\text{GM}}(\vt)$ instead of $\pi^{\text{GM}}(\vt^0)$.
Now, the GM-estimator $\widehat{\pi}_n^{\text{GM}}(\vt^\gamma)=(\widehat{\pi}_{n,1}^{\text{GM}}(\vt^\gamma),\ldots,\widehat{\pi}_{n,r}^{\text{GM}}(\vt^\gamma),\widehat{\sigma}_n^{\text{GM}}(\vt^\gamma))$ based on $\phi$ and $\chi$ is defined to satisfy
\begin{subequations} \label{pseudopseudo}
\begin{eqnarray}
\frac1{n-r} \sum\limits_{k=1}^{n-r} \phi  {\scriptstyle\left( \begin{pmatrix} Y_{kh}^\gamma(\vt) \\ \vdots \\ Y_{(k+r-1)h}^\gamma(\vt) \end{pmatrix}, \frac{Y_{(k+r)h}^\gamma(\vt) - \widehat{\pi}_{n,1}^{\text{GM}}(\vt^\gamma) Y_{(k+r-1)h}^\gamma(\vt) - \ldots - \widehat{\pi}_{n,r}^{\text{GM}}(\vt^\gamma)Y_{kh}^\gamma(\vt)}{ \widehat{\sigma}_n^{\text{GM}}(\vt^\gamma) } \right) \begin{pmatrix} Y_{kh}^\gamma(\vt) \\ \vdots \\ Y_{(k+r-1)h}^\gamma(\vt) \end{pmatrix}} &=& 0, \quad\quad\label{Phieq} \\
\frac1{n-r} \sum\limits_{k=1}^{n-r} \chi \left( \left( \frac{Y_{(k+r)h}^\gamma(\vt^\gamma) - \widehat{\pi}_{n,1}^{\text{GM}}(\vt^\gamma) Y_{(k+r-1)h}^\gamma(\vt) - \ldots - \widehat{\pi}_{n,r}^{\text{GM}}(\vt^\gamma)Y_{kh}^\gamma(\vt)}{ \widehat{\sigma}_n^{\text{GM}}(\vt^\gamma) } \right)^2 \right) &=& 0. \quad\quad\label{Chieq}
\end{eqnarray}
\end{subequations}
 Throughout the paper we assume that there exists
a solution of \eqref{pseudopseudo} although this is not always the case in practice.
\begin{example}\label{GMexample} \strut
\begin{itemize}
\item[(a)] There are two main classes of GM-estimators, the so--called Mallows estimators and the Hampel--Krasker--Welsch estimators. More information on them can be found in \cite{bustos}; \linebreak \citet{denbymartin}; \citet{martin,martinyohai}. In the literature, this kind of estimators sometimes appear under the name BIF (for bounded influence) estimators. The class of {\em Mallows estimators} are defined as 
$\phi(y,u) = w(y) \psi(u),$
where $w$ is a strictly positive weight function and $\psi$ is a suitably chosen robustifying function. The {\em Hampel--Krasker--Welsch estimators} are of the form
$$\phi(y,u) = \frac{\psi(w(y) u)}{w(y)},$$
where $w$ is  a weight function and $\psi$  is again a suitably chosen bounded function.
\item[(b)]  Typical choices for $\psi$ are the {\em Huber $\psi_k$--functions} (cf. \citet[Eq. (2.28)]{maronna}). Those functions are defined as
$\psi_k(u) = \text{sign}(u) \min\{ |u|, k\}$
for a constant $k > 0$. A possibility for $w$ is, e.g.,  $w(y) = \psi_k(|y|)/|y|$ for a Huber function $\psi_k$.
Another choice for $\psi$ is the so-called {\em Tukey bisquare} (or biweight) function which is given by
$$\psi(u) = u \left(1-\frac{u^2}{k^2} \right)^2 \Ind_{\{ |u| \leq k\}},$$
where $k$ is  a tuning constant.
\item[(c)] For the function $\chi$, a possibility is
$\chi(x^2) = \psi^2(x) - \E_{Z}[ \psi^2(Z) ]$
with the same $\psi$ function as in the definition of $\phi$. The random variable $Z$ is suitably distributed.
\end{itemize}
\end{example}
In order to develop an asymptotic theory and to obtain a robust estimator it is necessary to impose assumptions on $\phi$ and $\chi$ which we will do next analogous to \citet[(E1) - (E6)]{bustos}:
\begin{assumptionletter}\strut
         \label{as_H}
         \renewcommand{\theenumi}{(E.\arabic{enumi})}
         \renewcommand{\labelenumi}{\theenumi}
Suppose $\phi: \R^{r} \times \R \to \R$ and $\chi: \R \to \R$ satisfy the following assumptions:
                  \begin{enumerate}
\item \label{as_H1} For each $y \in \R^r$, the map $u \mapsto \phi(y,u)$ is odd, uniformly continuous and  $\phi(y, u) \geq 0$ for $u \geq 0$.
\item \label{as_H2} $(y,u) \mapsto \phi(y,u)y$ is bounded and there exists a $c > 0$ such that
$$| \phi(y,u)y - \phi(z,u)z| \leq c \|y-z \| \quad  \text{ for all } u \in \R.$$
\item \label{as_H3} The map $u \mapsto \frac{\phi(y,u)}{u}$ is non-increasing for $y \in \R^r$  and there exists a $u_0 \in \R$ such that $\frac{\phi(y, u_0)}{u_0} > 0$.
\item \label{as_H4} $\phi(y,u)$ is differentiable with respect to $u$ and the map $u \mapsto \frac{\partial \phi(y,u)}{\partial u}$ is continuous, while $(y,u) \mapsto \frac{\partial \phi(y,u)}{\partial u} y$ is bounded.
\item \label{as_H5}
${\displaystyle \E \left[ \sup_{u \in \R} \left\{ u \left( \frac{\partial}{\partial u} \phi \left( \begin{pmatrix}Y_h^\gamma(\vt) \\ \vdots \\ Y_{rh}^\gamma(\vt) \end{pmatrix}, u \right) \right)  \left\| \begin{pmatrix}Y_h^\gamma(\vt) \\ \vdots \\ Y_{rh}^\gamma(\vt) \end{pmatrix} \right\| \right\} \right] < \infty.}$
\item \label{as_H6} $\chi$ is bounded and increasing on $\{ x: -a \leq \chi(x) < b\}$ where $b = \sup_{x \in \R} \chi(x)$ and $a = -\chi(0)$. Furthermore, $\chi$ is differentiable and  $x \mapsto x \chi'(x^2)$ is continuous and bounded. Lastly, $\chi(u^2_0) > 0$.
         \end{enumerate}
\end{assumptionletter}

In the remaining of this section  we always assume that \autoref{as_G} and \ref{as_H} are satisfied.

\begin{remark}\label{existGMremark}
As pointed out in \citet[p. 497]{bustos} one can deduce from \citet[Theorem 2.1]{maronnayohai} that there exists a solution  $\pi^{\text{GM}}(\vt^\gamma) \in \R^r \times (0,\infty)$
of equation  \eqref{pseudo} if \linebreak \autoref{as_H} holds. Moreover, there exists
 a compact set $K \subset \R^r \times (0, \infty)$ with $\pi^{\text{GM}}(\vt^\gamma) \in K$ and for any $\pi \in K^c$ equation \eqref{pseudo}  does not hold (see \citet[p. 500]{bustos}).
\end{remark}

 In general it is not easy to verify that $\pi^{\text{GM}}(\vt^\gamma)$ is unique.
 Additionally, one would like to have that $\pi^{\text{GM}}(\vt^0) =\pi^{\text{GM}}(\vt)= \pi(\vt)$ are the parameters of the auxiliary AR$(r)$ model in the case that the GM-estimator is applied to realizations of an uncontaminated sampled stationary CARMA process $(Y_{mh}(\vt))_{m \in \Z}$. The following proposition gives a sufficient condition.
\begin{proposition}\label{uniquenessprop}
 Suppose that  $U_{r+1}(\vt)$ as defined in  equation \eqref{chap8:AuxRep} satisfies
 \begin{eqnarray} \label{sym}
(U_{r+1}(\vt), Y_{rh}(\vt), \ldots, Y_{h}(\vt)) \stackrel{\DD}{=} (-U_{ r+1}(\vt), Y_{rh}(\vt), \ldots, Y_{h}(\vt)).
\end{eqnarray}
Assume further that the function $u \mapsto \phi(y,u)$ is nondecreasing and strictly increasing for $|u| \leq u_0$, where $u_0$ satisfies Assumptions \ref{as_H3} and \ref{as_H6}, and the function $\chi$ is chosen in such a way that
\begin{eqnarray} \label{4.5}
    \E \left[ \chi \left( \left( \frac{U_{1}(\vt)}{\sigma(\vt)} \right) ^2 \right) \right] = 0.
\end{eqnarray}
 Finally, assume that $\gamma=0$ so that $(Y_{mh}^\gamma(\vt))_{m \in \Z} = (Y_{mh}(\vt))_{m \in \Z}$.
Then the auxiliary parameter $\pi(\vt)$ as defined in \autoref{auxiliaryprop} is the unique solution of \eqref{pseudo},
i.e., $\pi^{\text{GM}}(\vt^0)=\pi(\vt)$.
\end{proposition}

\begin{remark} \strut
\begin{enumerate}[(a)]
\item Assumption \eqref{sym} holds if the distribution of $U_{r+1}(\vt)$ is symmetric and $U_{r+1}(\vt)$ is independent of  $(Y_{rh}(\vt), \ldots, Y_{h}(\vt))$. This again is satisfied if  $(L_t)_{t\in\R}$  is a Brownian motion. 
\item The monotonicity assumption on $\phi$ is valid, e.g., for both the Mallows and Hampel--Krasker--Welsch estimators  when the function $\psi$ is chosen as a Huber $\psi_k$--function  with $u_0 = k$.
\item The assumption on $\chi$ is fulfilled, e.g., if $\chi$ is chosen as in \Cref{GMexample}(c) with \linebreak $Z \stackrel{\DD}{=} U_{1}(\vt)/\sqrt{\Var(U_{1}(\vt))}$. In the case that the driving L\'{e}vy process is a Brownian motion this means that $Z \sim \NN(0,1)$.
\end{enumerate}
\end{remark}

\begin{theorem}\label{chap8:ThmCons}
Suppose  that there exists a unique solution $\pi^{\text{GM}}(\vt^\gamma)$   of \eqref{pseudo}. Then \linebreak $\widehat{\pi}_n^{\text{GM}}(\vt^\gamma) \stackrel{n\to\infty}{\to} \pi^{\text{GM}}(\vt^\gamma)$  $\Pas$
\end{theorem}
The proof goes in the same vein as the proof of \citet[Theorem 2.1]{bustos} and is therefore omitted.

Next, we would like to deduce the asymptotic normality of the GM-estimator.
Let the set $K$ be given as in \Cref{existGMremark} and for $\pi = (\pi_1, \ldots, \pi_r, \sigma) \in K$ define
\begin{eqnarray} \label{QGM}
\mathscr{Q}_{\text{GM}}(\pi,\vt^\gamma) = \begin{pmatrix} \E \left[ \phi \left( \begin{pmatrix}Y_h^\gamma(\vt) \\ \vdots \\ Y_{rh}^\gamma(\vt) \end{pmatrix}, \frac{Y_{(r+1)h}^\gamma(\vt) - \pi_1 Y_{rh}^\gamma(\vt) - \ldots - \pi_{r} Y_{h}^\gamma(\vt)}{ \sigma } \right) \begin{pmatrix}Y_h^\gamma(\vt) \\ \vdots \\ Y_{rh}^\gamma(\vt) \end{pmatrix} \right] \\ \E \left[ \chi \left( \left( \frac{Y_{(r+1)h}^\gamma(\vt) - \pi_1 Y_{rh}^\gamma(\vt) - \ldots - \pi_{r} Y_{h}^\gamma(\vt)}{ \sigma } \right)^2 \right) \right] \end{pmatrix}.
\end{eqnarray}
For the proof of the asymptotic normality of the GM estimator we use a Taylor expansion of \linebreak $\mathscr{Q}_{\text{GM}}(\pi,\vt^\gamma)$ at $\pi^{\text{GM}}(\vt^\gamma)$.
With the knowledge of the asymptotic behavior  $\mathscr{Q}_{\text{GM}} ( \widehat{\pi}_n^{\text{GM}}(\vt^\gamma),\vt^\gamma)$ and \linebreak $\nabla_\pi \mathscr{Q}_{\text{GM}} ( \widehat{\pi}_n^{\text{GM}}(\vt^\gamma),\vt^\gamma)$ it is then straightforward to derive the asymptotic behavior
of the GM-estimator $\widehat{\pi}_n^{\text{GM}}(\vt^\gamma)$.

We need the following auxiliary result which is the analog of \citet[Lemma 3.1]{bustos} under our different model assumptions.
\begin{lemma}\label{AsNorm}
Define the map
$\Psi:  \R^{r+1}\times \R^r \times (0,\infty) \to \R^{r+1}$ as
\begin{align*}
&\Psi( y, \pi)
= \begin{pmatrix} \phi \left( \begin{pmatrix} y_1 \\ \vdots \\ y_r \end{pmatrix}, \frac{y_{r+1} - \pi_{1} y_r - \ldots - \pi_{r} y_1}{ \sigma } \right) \begin{pmatrix} y_1 \\ \vdots \\ y_r \end{pmatrix} \\ \chi \left( \left( \frac{y_{r+1} - \pi_{1} y_r - \ldots - \pi_{r} y_1}{ \sigma } \right)^2 \right) \end{pmatrix}.
\end{align*}
Furthermore, define the stochastic process $\Psi(\vt^\gamma)=(\Psi_k(\vt^\gamma))_{k \in \N}$ as
$\Psi_k(\vt^\gamma) \linebreak = \Psi(Y_{kh}^\gamma(\vt), \ldots, Y_{(k+r+1)h}^\gamma(\vt), \pi^{\text{GM}}(\vt^\gamma))$.
Then
$$\frac{1}{\sqrt{n-r}} \sum_{k=1}^{n-r} \Psi_k(\vt^\gamma) \stackrel{\DD}{\to} \NN( 0, \II_{\text{GM}}(\vt^\gamma)),$$
where the $(i,j)$-th component of $\II_{\text{GM}}(\vt^\gamma)$ is
\begin{equation}\label{IIGM}
[\II_{\text{GM}}(\vt^\gamma)]_{ij} = \E \left[ \Psi_{1,i}(\vt^\gamma) \Psi_{1,j}(\vt^\gamma) \right] + 2 \sum_{k=1}^{\infty} \E \left[ \Psi_{1,i}(\vt^\gamma)
\Psi_{1+k,j}(\vt^\gamma) \right]
\end{equation}
and $\Psi_{k,i}(\vt^\gamma)$ denotes the $i$--th component of $\Psi_k(\vt^\gamma)$, $i=1, \ldots, r+1$. Especially, each $[\II_{\text{GM}}(\vt^\gamma)]_{ij}$ is
finite for $i, j \in \{ 1, \ldots, r+1 \}$.
\end{lemma}

First, we derive the asymptotic behavior of the gradient $ \nabla_\pi \mathscr{Q}_{\text{GM}}(\pi_n,\vt^\gamma)$.

\begin{lemma} \label{Lemma 4.8}
Let $\mathscr{Q}_{\text{GM}}(\pi,\vt^\gamma)$ be defined as in \eqref{QGM}. Then
the gradient $\nabla_\pi \mathscr{Q}_{\text{GM}}(\pi,\vt^\gamma)$ exists. Moreover, for any sequence $(\pi_n)_{n\in\N}$
with $\pi_n \stackrel{\P}{\to}\pi^{\text{GM}}(\vt^\gamma)$ as $n\to\infty$
we have as $n\to\infty$,
\begin{eqnarray*}
    \nabla_\pi \mathscr{Q}_{\text{GM}}(\pi_n,\vt^\gamma) \stackrel{\P}{\to} \nabla_\pi \mathscr{Q}_{\text{GM}}(\pi^{\text{GM}}(\vt^\gamma),\vt^\gamma).
\end{eqnarray*}
\end{lemma}

Next, we deduce the asymptotic normality of  $\mathscr{Q}_{\text{GM}}( \widehat{\pi}_n^{\text{GM}}(\vt^\gamma),\vt^\gamma)$.

\begin{lemma} \label{Lemma 4.7}
Let $\mathscr{Q}_{\text{GM}}(\pi,\vt^\gamma)$ be defined as in \eqref{QGM} and suppose that $\nabla_\pi \mathscr{Q}_{\text{GM}}(\pi,\vt^\gamma)$ is non-singular.
Furthermore, let $\II_{\text{GM}}(\vt^\gamma)$ be given as in \eqref{IIGM} and
 suppose that $\widehat{\pi}_n^{\text{GM}}(\vt^\gamma) \stackrel{\P}{\to} \pi^{GM}(\vt^\gamma)$ as $n\to\infty$.
Then, as $n\to\infty$,
$$\sqrt{n-r} \mathscr{Q}_{\text{GM}} ( \widehat{\pi}_n^{\text{GM}}(\vt^\gamma),\vt^\gamma) \stackrel{\DD}{\longrightarrow} \NN(0, \II_{\text{GM}}(\vt^\gamma)).$$
\end{lemma}

The following analog version of \citet[Theorem 2.2]{bustos} holds in our setting which gives the asymptotic normality of the GM-estimator.

\begin{theorem}\label{chap8:ThmNorm}
Let $\mathscr{Q}_{\text{GM}}(\pi,\vt^\gamma)$ be defined as in \eqref{QGM} and suppose that $\JJ_{\text{GM}}(\vt^\gamma) :=\nabla_\pi \mathscr{Q}_{\text{GM}}(\pi,\vt^\gamma)$ is non-singular.
Furthermore, let $\II_{\text{GM}}(\vt^\gamma)$ be given as in \eqref{IIGM} and
 suppose that \linebreak $\widehat{\pi}_n^{\text{GM}}(\vt^\gamma) \stackrel{\P}{\to} \pi^{GM}(\vt^\gamma)$ as $n\to\infty$.
Then, as $n\to\infty$,
$$ \sqrt{n-r} ( \widehat{\pi}_n^{\text{GM}}(\vt^\gamma) - \pi^{\text{GM}}(\vt^\gamma)) \stackrel{\DD}{\longrightarrow} \NN(0, \Xi_{\text{GM}}(\vt^\gamma)),$$
where
\begin{equation}\label{XiGM}
\Xi_{\text{GM}}(\vt^\gamma) := [\JJ_{\text{GM}}(\vt^\gamma)]^{-1} \II_{\text{GM}}(\vt^\gamma) [\JJ_{\text{GM}}(\vt^\gamma)]^{-1}.
\end{equation}
\end{theorem}

\section{The indirect estimator for the CARMA parameters} \label{sec:indirect estimator CARMA}

\subsection{Asymptotic normality}

In \Cref{sec:indirect estimation} we already introduced the indirect estimator and presented in \Cref{chap8:IndEstThm} sufficient criteria
for the indirect estimator to be consistent and asymptotically normally distributed. In the following we want to show that
these assumptions are satisfied in the setting of discretely sampled CARMA processes when we use as estimator $\widehat\pi_n^S(\vt)$ in the simulation part the least-squares- (LS-) estimator $\widehat \pi_n^{\text{LS}}(\vt)$
and for $\widehat\pi_n$ the GM-estimator $\widehat\pi_n^{\text{GM}}(\vt_0)$.

\begin{definition}\label{def:pseudoLS}
Based on the sample $\mathcal{Y}^{sn}_S(\vt) = (Y_h^S(\vt), \ldots, Y_{snh}^S(\vt))$  the LS-estimator
$\wh{\pi}_{sn}^{\text{LS}}(\vt)  = (\wh{\pi}_{sn, 1}^{\text{LS}}(\vt), \ldots, \wh{\pi}_{sn,r}^{\text{LS}}(\vt), \wh \sigma_{sn}^{\text{LS}}(\vt))$ of $\pi(\vt)$ minimizes
\begin{align}
\LL_{\text{LS}}(\pi, \mathcal{Y}_S^{sn}(\vt) ) &:= \frac{1}{sn-r} \sum_{k=1}^{sn-r} \left( Y_{(k+r)h}^S(\vt) - \pi_1Y_{(k+r-1)h}^S(\vt) - \ldots - \pi_rY_{kh}^S(\vt) \right)^2
  \label{chap8:DefPiRing}
\intertext{in $\Pi':=\pi(\Theta) $ and $\wh \sigma_{sn}^{\text{LS}}(\vt)$ is defined as  }
\wh \sigma_{\text{LS},sn}^{2}(\vt) = &\frac{1}{sn-r} \sum_{k=1}^{sn-r} \left( Y_{(k+r)h}^S(\vt) - \wh{\pi}_{sn,1}^{\text{LS}}(\vt)Y_{(k+r-1)h}^S(\vt)
- \ldots - \wh{\pi}_{sn, r}^{\text{LS}}(\vt)Y_{kh}^S(\vt) \right)^2. \notag
\end{align}
\end{definition}

\begin{remark}\label{def:pseudoMLE}
The quasi ML-function for the auxiliary AR$(r)$ parameters of the discretely sampled CARMA process is defined as
\begin{eqnarray*}
\LL_{\text{QMLE}}(\pi, \mathcal{Y}_S^{sn}(\vt))= \frac{1}{sn-r} \sum_{k=1}^{sn-r} \left( \log(\sigma^2)
+  \frac{ (Y_{(k+r)h}^S(\vt) - \pi_1Y_{(k+r-1)h}^S(\vt) - \ldots - \pi_rY_{kh}^S(\vt))^2}{\sigma^2} \right)
\end{eqnarray*}
and the quasi ML-estimator as
$
\wh{\pi}^{\text{QMLE}}_{sn}(\vt) = \argmin_{\pi \in \Pi'} \LL_{\text{QMLE}}(\pi, \mathcal{Y}^{sn}_S(\vt)).
$
It is well known that for the estimation of AR$(r)$ parameters the ML-estimator and the LS-estimator are equivalent (this can be seen
by straightforward calculations taking the derivatives of the ML-function $\LL_{\text{QMLE}}$ which are proportional to the derivatives of $\LL_{\text{LS}}$).
\end{remark}

\begin{theorem} \label{Theorem:indirect estimator}
Let \autoref{as_D}, \ref{Assumption B}, \ref{as_G} and \ref{as_H} hold. Suppose that the unique solution $\pi^{\text{GM}}(\vt_0)$   of \eqref{pseudo} for $(Y_{mh})_{m\in\Z}$ is $\pi(\vt_0)$, that $\nabla_\vt\pi(\vt_0)$ has full column rank $N(\Theta)$ and that $\JJ_{\text{GM}}(\vt_0)$ is non-singular. Further, assume that $\E|L_1^S|^{2N^*}$ for some $N^*\in\N$ with $2N^*>\max(N(\Theta),4+\delta)$.
If $\widehat\pi_n^S(\vt)=\widehat \pi_n^{\text{LS}}(\vt)$ and  $\widehat\pi_n=\widehat\pi_n^{\text{GM}}(\vt_0)$ then the indirect estimator
$\wh{\vt}_n^{\text{Ind}}$ is weakly consistent and
\begin{equation*}
\sqrt{n} (\wh{\vt}_n^{\text{Ind}} - \vt_0) \stackrel{\DD}{\longrightarrow} \NN(0, \Xi_{\text{Ind}}(\vt_0)),
\end{equation*}
where
$$\Xi_{\text{Ind}}(\vt_0) = \JJ_{\text{Ind}}( \vt_0)^{-1} \II_{\text{Ind}}( \vt_0)\JJ_{\text{Ind}}( \vt_0)^{-1}$$
with
\begin{eqnarray*}
\JJ_{\text{Ind}}( \vt_0) &=& [\nabla_\vt \pi(\vt_0)]^T \Omega [\nabla_\vt \pi(\vt_0)] \quad \text{ and }\quad\\
\II_{\text{Ind}}(\vt_0) &=& [\nabla_\vt \pi(\vt_0)]^T \Omega \left[ \Xi_{\text{GM}}(\vt_0) + \frac1s \Xi_{LS}(\vt_0) \right] \Omega [\nabla_\vt \pi(\vt_0)],
\end{eqnarray*}
where the matrix $\Xi_{\text{LS}}(\vt)$ is defined as in \eqref{XiGM} with $\phi(y,u) = u$ and $\chi(x) = x -1$.
\end{theorem}
We have already proven that \ref{chap8:UnifConv1} and \ref {ThmIndEstAssumption2} of \Cref{chap8:IndEstThm}  are satisfied.
To show the remaining conditions on the LS-estimator $\widehat\pi_n^{\text{LS}}(\vt)$ we require several auxiliary results. The remaining of this section
is devoted to that.

Sufficient conditions for \ref{chap8:UnifConv} and \ref{ThmIndEstAssumption3}  are the weak uniform convergence of the LS-estimator and its derivatives.
Since the LS-estimator is defined via the sample autocovariance function we first derive the uniform weak convergence of the sample
autocovariance function and its derivatives.

\begin{proposition} \label{Konv Kov}
For $j,l\in\{0,\ldots,r\}$ define
\begin{eqnarray*}
    \wh\gamma_{\vt,n}(l,j)=\frac{1}{n-r}\sum_{k=1}^{n-r}Y_{(k+l)h}(\vt)Y_{(k+j)h}(\vt).
\end{eqnarray*}
Then for $i,u\in\{1,\ldots,N(\Theta)\}$ the following statements hold.
\begin{itemize}
    \item[(a)] $\sup_{\vt\in\Theta}|\wh\gamma_{\vt,n}(l,j)-\gamma_{\vt}(l-j)|\stoch 0$.
    \item[(b)] $\sup_{\vt\in\Theta}\left|\frac{\partial}{\partial\vt_i}\wh\gamma_{\vt,n}(l,j)-\frac{\partial}{\partial\vt_i}\gamma_{\vt}(l-j)\right|\stoch 0$.
    \item[(c)] $\sup_{\vt\in\Theta}\left|\frac{\partial^2}{\partial\vt_i \partial\vt_u}\wh\gamma_{\vt,n}(l,j)-\frac{\partial}{\partial\vt_i \partial\vt_u}\gamma_{\vt}(l-j)\right|\stoch 0$.
\end{itemize}
\end{proposition}
Then the proof of \ref{chap8:UnifConv} follows from \Cref{proposition 5.2}.

\begin{proposition} \label{proposition 5.2} $\mbox{}$
\begin{itemize}
    \item[(a)] $\sup_{\vt\in\Theta}|\wh{\pi}_n^{\text{LS}}(\vt)-\pi(\vt)|\stackrel{\P}{\to}0$.
    \item[(b)] $\sup_{\vt\in\Theta}|\nabla_\vt\wh{\pi}_n^{\text{LS}}(\vt)-\nabla_\vt\pi(\vt)|\stackrel{\P}{\to}0$.
    \item[(c)] $\sup_{\vt\in\Theta}|\nabla^2_\vt\wh{\pi}_n^{\text{LS}}(\vt)-\nabla^2_\vt\pi(\vt)|\stackrel{\P}{\to}0$.
\end{itemize}
\end{proposition}
A direct consequence from this is the next corollary.

\begin{corollary} \label{Corollary 5.4}
Let $\ov\vt_n$ be a sequence in $\Theta$ with $\ov\vt_n\stoch \vt_0$. Then the following statements hold:
\begin{itemize}
    \item[(a)] $\wh{\pi}_n^{\text{LS}}(\ov\vt_n)\stoch \pi(\vt_0)$.
    \item[(b)] $\nabla_\vt\wh {\pi}_n^{\text{LS}}(\ov\vt_n)\stoch \nabla_\vt\pi(\vt_0)$.
    \item[(c)] $\nabla_\vt^2\wh{\pi}_n^{\text{LS}}(\ov\vt_n)\stoch \nabla_\vt^2\pi(\vt_0)$.
\end{itemize}
\end{corollary}
This corollary already gives \ref{ThmIndEstAssumption3}.

Finally, \ref{ThmIndEstAssumption1} is a consequence of \Cref{chap8:LemmaNormRing} which gives the asymptotic normality of the the LS-estimator.
In principle this follows from \autoref{chap8:ThmNorm}
 by interpreting the least squares estimator as a particular GM-estimator with $\phi(y,u) = u$ and $\chi(x) = x-1$.

\begin{proposition}\label{chap8:LemmaNormRing}
For any $\vt\in\Theta$ the LS-estimator $\wh{\pi}_n^{\text{LS}}(\vt)$ is strongly consistent and as $n\to\infty$,
$$\sqrt{n} (\wh{\pi}_n^{\text{LS}}(\vt) - \pi(\vt)) \stackrel{\DD}{\longrightarrow} \NN \left(0, \Xi_{\text{LS}}(\vt) \right). $$
\end{proposition}

\subsection{Robustness properties} \label{sec:robust}

   Roughly speaking an estimator is robust when small deviations from the nominal model have not much effect on the estimator. This property is known as qualitative robustness or resistance of the estimator and was originally introduced in \citet{hampel} for i.i.d. sequences.  The same article also gives a slight extension to the case of data that are generated by permutation--invariant distributions, introducing the term $\pi$--robustness (\citet[p.1893]{hampel}). Of course, time series do not satisfy the assumption of permutation invariance in general. Therefore, there have been various attempts to generalize the concept of qualitative robustness to the time series setting.
\citet[Theorem 3.1]{boentefraimanyohai} prove that their $\pi_{d_n}$--robustness for time series  is   equivalent to Hampel's $\pi$--robustness for i.i.d. random variables and therefore, extends Hampel's $\pi$--robustness.  They go ahead and define the term resistance as well. The concept of resistance has the intuitive appeal of making a statement about changes in the values of the estimator when comparing two deterministic samples. In contrast, $\pi_{d_n}$--robustness is only a statement concerning the distribution of the estimator, which is in general not easily tractable.  The indirect estimator is weakly resistant and $\pi_{d_n}$--robust.
The explicit definitions and the derivation of these properties for our indirect estimator are given in Section~\ref{sec:qualitative:1} of the Supporting Information.

 Intuitively speaking,
 the influence functional measures the change in the asymptotic bias of an estimator caused by an infinitesimal amount of contamination in the data. This measure of robustness was originally introduced as influence curve by \citet{hampel2} for i.i.d. processes.  It was later generalized to the time series context by \citet{kunsch} who explicitly studies the estimation of autoregressive processes. However, in the paper of K\"{u}nsch only estimators which depend on a finite--dimensional marginal distribution of the data--generating process and a very specific form of contaminations are considered. To remedy this, a further generalization was then made by \citet{martinyohai} who consider the influence functional and explicitly allow for the estimators to depend on the measure of the process which makes more sense in the time series setup  (cf. \citet[Section 4]{martinyohai}).
 In the sense of \citet[Section 4]{martinyohai} the indirect estimator has a bounded influence functional; see Section~\ref{InfFunc:1} in the Supporting Information.

The breakdown point is (for a sample of data with fixed length
$n$) the maximum percentage of outliers which can be contained in the data without
''ruining'' the estimator. In this sense, it measures how much the observed data can
deviate from the nominal model before catastrophic effects in the estimation procedure
happen. However, the formal definition depends on the model and the estimator.  \citet{maronnayohai91}
and \citet{maronnabustosyohai} deal explicitly with the breakdown point of GM-estimators in regression models and
\citet{martinyohai2} and \citet{martin} study it in the time series context. A very general definition
of the breakdown point is given in \citet[Definition 1 and Definition 2]{gentonlucas}.
Heuristically speaking, the fundamental idea of that definition is that the breakdown point is the smallest amount of outlier contamination with the property that the performance of the estimator does not get worse anymore if the contamination is increased further.
As already mentioned in
\citet[p. 239]{martin} (the proof is given in the unpublished paper of \citet{martinjong}), and later in \citet[p. 377]{delunagenton} and \citet[p. 89]{gentonlucas}, the breakdown point of the
GM-estimator applied to estimate the parameters of an AR$(r)$ process is $1/(r+1)$. Hence, the breakdown point of our indirect estimator is
as well $1/(r+1)$ since  the other building block of the indirect estimator, the estimator $\wh{\pi}_n^{\text{S}}(\vt)$ is applied to a simulated outlier--free sample.

\section{Simulation study }\label{sect:Simulation}

 We simulate  CARMA processes on the interval $[0, 1000]$ and choose a sampling distance of $h=1$, resulting in $n=1000$ observations of the discrete--time process. The simulated processes are driven either by a standard Brownian motion or by a univariate NIG (normal inverse Gaussian) Lévy process.
 The increments of a  NIG-L\'evy process $L(t)-L(t-1)$ have the density \label{pageNIG}
\begin{align*}
f_{NIG}(x;\mu,\alpha,\beta,\delta)=\frac{\alpha\delta}{\pi}\exp(\delta\sqrt{\alpha^2 -\beta^2}+\beta x)\frac{K_1(\alpha\sqrt{\delta^2+x^2})}{\sqrt{\delta^2+x^2}} , \quad x\in\R,
\end{align*}
$\mu\in\R$ is a location parameter, $\alpha\geq 0$ is a shape parameter, $\beta\in\R$ is a symmetry parameter and $K_1$ is the modified Bessel
function of the third kind with index $1$. 
The variance of the process is then
$\sigma_L^2=\delta\alpha^2/(\alpha^2-\beta^2)^{\frac{3}{2}}.$
For the NIG Lévy process we use the parameters $\alpha=3$, $\beta=1$,  $\delta= 2.5145$ and $\mu = -0.8890$. These parameters result in a zero--mean L\'{e}vy process with variance approximately $1$ which allows for comparison of the results to the standard Brownian motion case.
For the outlier model we choose additive outliers where the process $(V_m)_{m \in \Z}$ is a sequence of i.i.d. Bernoulli random variables with $\P( V_1 = 1) = \gamma$. The process $(Z_m)_{m \in \Z}$ is  $Z_m = \xi$ for $m \in \Z$ where  $\xi$ and $\gamma$ take different values in different simulations. 

The indirect estimator is defined as in \Cref{sec:indirect estimator CARMA}. We take $\wh \pi_n$ as GM-estimator $\widehat\pi_n^{\text{GM}}(\vt_0)$ using the \texttt{R} software which provides the pre--built function \texttt{arGM} in the package \texttt{robKalman} for applying GM-estimators to AR processes. This function uses a Mallows estimator as in \Cref{GMexample}(a). The weight function $w(y)$ is the Tukey bisquare function from \Cref{GMexample}(b) applied to $\| y \|$, for the function $\psi(u)$ the user can choose between the Huber $\psi_k$--function and the bisquare function. The function is implemented as an iterative least squares procedure as described by \citet[p. 231ff.]{martin}. 
    We do $6$ iterations using the Huber function and then $50$ iterations with the bisquare function, which is the maximum number of iterations where the algorithm stops earlier if convergence is achieved. In our experiments we use $k=4$ for the tuning constant of the $\psi_k$--function.  In general, we set $s=75$ to obtain the simulation--based observations $\mathcal{Y}^{sn}_S(\vt)=(Y_{h}^S(\vt), \ldots, Y_{snh}^S(\vt))$ in the simulation part of the indirect procedure. The type of L\'{e}vy process used for the simulation part is of the same type as the Lévy process driving the CARMA process. For the estimator $\wh \pi_n^{\text{S}}(\vt)$ we apply the least squares estimator and as weighting matrix $\Omega$ we take
the identity matrix for convenience reasons. In some  experiments we first estimated the asymptotic covariance matrix of the GM-estimator by the empirical covariance matrix of a suitable number of independent realizations of $\wh \pi_n$. Setting $\Omega$ as the inverse of that estimate did not significantly affect the procedure positively or negatively so that the use of the convenient identity matrix seems justified.
In each experiment, we calculate the indirect estimator and, for comparison purposes, the QMLE as defined in \citet{schlemmstelzer}. For the indirect estimator as well for the
QMLE we use  50 independent samples and report on the average estimated value, the bias and the empirical variance of the parameter estimates.


\begin{table}[h] \small
\centering
\begin{tabular}{|l||c|c|c||c|c|c||c|c|c|}
\hline
   \; & \multicolumn{3}{c||}{$\xi=0$, $\gamma=0$ (uncontaminated)} & \multicolumn{3}{c||}{$\xi=10$, $\gamma=0.1$}  & \multicolumn{3}{c|}{$\xi=5$, $\gamma=0.15$}\\
 \hline
  \; & Mean & Bias & Var & Mean & Bias & Var & Mean & Bias & Var\\
  \hline
    $r=1$ & 2.1187 & -0.1187 & 0.1008 & -2.0027 & -0.0027 & 0.1004 & -2.0711 & -0.0711 & 0.0905\\
  $r=2$ & -2.1238 & -0.1238 & 0.0956 &-2.1121 & -0.1121 & 0.1681 & -1.9555 & 0.0445 & 0.1494\\
  $r=3$ & -2.1214 & -0.1214 & 0.0937 &-2.4828 & -0.4828 & 0.4811 & -2.4838 & -0.4838 & 0.3367\\
  \hline
\end{tabular}
\caption{Indirect estimation of a CARMA$(1,0)$ process with parameter $\vt_0=-2$ driven by a Brownian motion with $n=1000$.} \label{Table 2A}
\end{table}

First, CARMA(1,0) processes with parameter $\vt_0\in (-\infty,0)$ are studied where $A_{ \vt_0}=\vt_0$ and $c_{\vt_0}=1$. These processes are of particular interest because their discretely sampled version admit  an AR(1) representation. For this reason, one would expect the indirect procedure to work very well  as the auxiliary representation is actually exact. Initially, in Table~\ref{Table 2A}, we estimate contaminated and uncontaminated CARMA(1,0) processes with $\vt_0=-2$ driven by a Brownian motion
using in  the indirect estimation method an auxiliary AR$(r)$ process with   $r=1,2,3$.
For uncontaminated CARMA(1,0) processes the parameter  $r=1$ gives the lowest absolute bias where the variance is the highest. However, the bias and the variances are very similar. By contrast with contaminated CARMA(1,0) processes,
 if we increase $r$ the   variances increase. That is not surprising because $r=1$ reflects
the true model and including more parameters than necessary results in more estimation errors. As well the bias is quite low for $r=1$.

Next, in Table~\ref{Table 1A}, we compare the indirect estimator with $r=1$ and the QMLE  for a Brownian motion driven CARMA$(1,0)$ process with either
$\vt_0=-2$ or $\vt_0=-0.2$ . For $\vt_0=-0.2$ the CARMA$(1,0)$ process is not so far away from a non-stationary process.  In both cases we see that the QMLE and the indirect estimator work quite well for uncontaminated CARMA(1,0) processes (top of Table~\ref{Table 1A}). The QMLE has a lower variance in both cases where for
$\vt_0=-2$ the absolute bias of the indirect estimator  and for $\vt_0=-0.2$ the bias of the QMLE is lower. But still for both estimation procedures the values are
comparable.
\begin{table}[h] \small
\centering
\begin{tabular}{|l||c|c|c||c|c|c|c|}
 \hline
 \; & \multicolumn{6}{c|}{$\xi=0$, $\gamma=0$ (uncontaminated)} \\
 \hline
   \; & \multicolumn{3}{c||}{QMLE} & \multicolumn{3}{c|}{Indirect}  \\
 \hline
  \; & Mean & Bias & Var & Mean & Bias & Var \\
  \hline
    $\vt_0=-2$ & -2.1424 & -0.1424 & 0.0913 &  -2.1187 & -0.1187 & 0.1008 \\
  \hline
    $\vt_0=-0.2$ & -0.2031 & -0.0031 & 0.0003 & -0.2100 & -0.0100 & 0.0009\\
  \hline
 \hline
 \; & \multicolumn{6}{c|}{$\xi=5$, $\gamma=0.1$} \\
 \hline
   \; & \multicolumn{3}{c||}{QMLE} & \multicolumn{3}{c|}{Indirect}  \\
 \hline
  \; & Mean & Bias & Var & Mean & Bias & Var \\
  \hline
     $\vt_0=-2$ & -2.4017 & -0.4017 & 0.1487 & -2.0027 & -0.0027 & 0.1004  \\
    \hline
    $\vt_0=-0.2$ & -2.4513 & -2.2513 & 0.0093  & -0.1981 & 0.0019 & 0.0010\\
  \hline
 \hline
 \; & \multicolumn{6}{c|}{$\xi=10$, $\gamma=0.1$} \\
 \hline
   \; & \multicolumn{3}{c||}{QMLE} & \multicolumn{3}{c|}{Indirect}  \\
 \hline
  \; & Mean & Bias & Var & Mean & Bias & Var \\
  \hline
    $\vt_0=-2$ &   -4.7942 & -2.7942 & 0.0315 & -1.8070 & 0.1930 & 0.0655 \\
    \hline
    $\vt_0=-0.2$ & -4.9139 & -4.7139 & 0.0440 & -0.1981 & 0.0019 & 0.0010\\
  \hline
 \hline
 \; & \multicolumn{6}{c|}{$\xi=5$, $\gamma=0.15$} \\
 \hline
   \; & \multicolumn{3}{c||}{QMLE} & \multicolumn{3}{c|}{Indirect}  \\
 \hline
  \; & Mean & Bias & Var & Mean & Bias & Var \\
  \hline
    $\vt_0=-2$ & -2.1207 & -0.1207 & 0.2592  & -2.0711 & -0.0711 & 0.0905\\
    \hline
    $\vt_0=-0.2$ & -2.9511 & -2.7511 & 0.0102 & -0.1772 & 0.0228 & 0.0008  \\
  \hline
  \end{tabular}
\caption{Estimation results for CARMA$(1,0)$ processes with parameter $\vt_0$ driven by a Brownian motion with $n=1000$ and $r=1$.} \label{Table 1A}
\end{table}
If we  allow additionally outliers in the CARMA(1,0) model, already in the case  $\xi=5$ and $\gamma=0.1$, the indirect estimator performs vastly better than the QMLE giving a much less biased estimate and lower variance. 
For $\xi=10$ and $\gamma=0.1$ the QMLE  is far away from the true values where the indirect estimator still gives good results.
Increasing $\gamma$ to $0.15$ but keeping $\xi = 5$ shows that both estimators perform worse than in the situation with $\gamma=0.1$, which is to be expected. But once again, the indirect estimator gives excellent results. However, the QMLE runs much faster than the indirect estimator. 

In a further study we investigate CARMA(3,1) processes. This especially means that the sampled process is not a weak AR process anymore.  The true parameter is
$\vt_0 = \begin{pmatrix} \vt_1 & \vt_2 & \vt_3 & \vt_4 & \vt_5 \end{pmatrix}$ such that
\begin{eqnarray*}
A_{\vt_0}=\begin{pmatrix}
  0      & 1            & 0            \\
  0      & 0            & 1         \\
  \vt_1   & \vt_2     & \vt_3
    \end{pmatrix}\in \R^{3\times 3}\quad
    \text{ and } \quad c_{\vt_0}=(\vt_4,\, \vt_5,0).
\end{eqnarray*}
For the CARMA(3,1) model  we choose $r=5$, which is also the minimum order of the auxiliary AR representation to satisfy \autoref{Assumption B}.
We also tried different values of $r$ but they didn't give better results (see Table~\ref{Table 7A} in the Supporting Information). In contrast, for
contaminated CARMA(3,1) processes  it seems that the  absolute bias and variance are  the lowest for $r=5$.

\begin{table}[h] \small
\centering
\begin{tabular}{|l||c|c|c||c|c|c|c|}
\hline
 \; & \multicolumn{6}{c|}{$n=200$} \\
 \hline
   \; & \multicolumn{3}{c||}{QMLE} & \multicolumn{3}{c|}{Indirect}  \\
 \hline
  \; & Mean & Bias & Var & Mean & Bias & Var \\
  \hline
   $\vt_1=-1$ & -1.02574 & -0.02574 & 0.01814 & -1.32232 & -0.32232 & 0.94063\\
   \hline
    $\vt_2=-2$ &-1.99600 & 0.00400 & 0.01105 & -2.28763 & -0.28763 & 1.29387\\
    \hline
    $\vt_3=-2$ &-1.98396 & 0.01604 & 0.02499  & -2.11439 & -0.11439 & 0.42097 \\
    \hline
    $\vt_4=0$ & -0.00688 & -0.00688 & 0.01309 & 0.00287 & 0.00287 & 0.02619\\
    \hline
    $\vt_5=1$ &  0.99773 & -0.00227 & 0.01599 & 0.88711 & -0.11289 & 0.08930\\
    \hline
    \hline
 \; & \multicolumn{6}{c|}{$n=1000$} \\
 \hline
   \; & \multicolumn{3}{c||}{QMLE} & \multicolumn{3}{c|}{Indirect}  \\
 \hline
  \; & Mean & Bias & Var & Mean & Bias & Var \\
  \hline
   $\vt_1=-1$ & -1.01492 & -0.01492 & 0.00129& -1.02333 & -0.02333 & 0.00515\\
   \hline
    $\vt_2=-2$ & -1.99192 & 0.00808 & 0.00310& -1.98112 & 0.01888 & 0.00905 \\
    \hline
    $\vt_3=-2$ & -1.99376 & 0.00624 & 0.00411& -2.00188 & -0.00188 & 0.01286\\
    \hline
    $\vt_4=0$ & 0.01404 & 0.01404 & 0.00109 & 0.01253 & 0.01253 & 0.00436\\
    \hline
    $\vt_5=1$ & 1.01311 & 0.01311 & 0.00044  & 1.00248 & 0.00248 & 0.00356\\
    \hline
 \hline
 \; & \multicolumn{6}{c|}{$n=5000$} \\
 \hline
   \; & \multicolumn{3}{c||}{QMLE} & \multicolumn{3}{c|}{Indirect}  \\
 \hline
  \; & Mean & Bias & Var & Mean & Bias & Var \\
  \hline
   $\vt_1=-1$ & -1.00647 & -0.00647 & 0.00016& -1.00534 & -0.00534 & 0.00046\\
   \hline
    $\vt_2=-2$ & -1.99953 & 0.00047 & 0.00005& -1.99576 & 0.00424 & 0.00034  \\
    \hline
    $\vt_3=-2$ & -1.99398 & 0.00602 & 0.00025& -1.99930 & 0.00070 & 0.00103 \\
    \hline
    $\vt_4=0$ &  0.00843 & 0.00843 & 0.00004 & 0.00089 & 0.00089 & 0.00035\\
    \hline
    $\vt_5=1$ & 1.00849 & 0.00849 & 0.00003 & 0.99864 & -0.00136 & 0.00027 \\
    \hline
 \hline
\end{tabular}
\caption{Estimation results for an  uncontamined CARMA$(3,1)$ process with parameter $\vartheta_0=(\vt_1,\vt_2,\vt_3,\vt_4,\vt_5)$ driven by a Brownian motion with $r=5$.}\label{Table 4A}
\end{table}

\begin{table}[h] \small
\centering
\begin{tabular}{|l||c|c|c||c|c|c|c|}
\hline
 \; & \multicolumn{6}{c|}{$n=200$} \\
 \hline
   \; & \multicolumn{3}{c||}{QMLE} & \multicolumn{3}{c|}{Indirect}  \\
 \hline
  \; & Mean & Bias & Var & Mean & Bias & Var \\
  \hline
   $\vt_1=-1$ & -1.01143 &-0.01143 & 0.00735 & -1.21354 & -0.21354 & 0.48999\\
   \hline
    $\vt_2=-2$ & -2.03219 & -0.03219 & 0.01168 & -2.10578 & -0.10578 & 0.67909\\
    \hline
    $\vt_3=-2$ & -1.95505 &  0.04495 & 0.01746 & -2.02967 & -0.02967 & 0.20588 \\
    \hline
    $\vt_4=0$ & 0.01161 &  0.01161 & 0.00766 &  0.02754 &  0.02754 & 0.03439\\
    \hline
    $\vt_5=1$ & 1.01040 &  0.01040 & 0.00462 &  0.89459 & -0.10541 & 0.17048\\
    \hline
     \hline
 \; & \multicolumn{6}{c|}{$n=1000$} \\
 \hline
   \; & \multicolumn{3}{c||}{QMLE} & \multicolumn{3}{c|}{Indirect}  \\
 \hline
  \; & Mean & Bias & Var & Mean & Bias & Var \\
  \hline
   $\vt_1=-1$ & -1.01502 & -0.01502 & 0.00123 & -1.03184 & -0.03184 & 0.00510\\
   \hline
    $\vt_2=-2$ & -2.00152 & -0.00152 & 0.00125 & -1.98345 &  0.01655 & 0.00682\\
    \hline
    $\vt_3=-2$ & -1.98346 &  0.01654 & 0.00232 & -1.99289 &  0.00711& 0.00944\\
    \hline
    $\vt_4=0$ &  0.00235 &  0.00235 & 0.00052 &  -0.01551 & -0.01551 & 0.00544\\
    \hline
    $\vt_5=1$ & 1.00318 &  0.00318 & 0.00045 & 0.99433 & -0.00567 & 0.00201\\
    \hline
 \hline
 \; & \multicolumn{6}{c|}{$n=5000$} \\
 \hline
   \; & \multicolumn{3}{c||}{QMLE} & \multicolumn{3}{c|}{Indirect}  \\
 \hline
  \; & Mean & Bias & Var & Mean & Bias & Var \\
  \hline
   $\vt_1=-1$ & -1.00640 & -0.00640 & 0.00005 & -1.00584 & -0.00584 & 0.00049\\
   \hline
    $\vt_2=-2$ & -1.99918 &  0.00082 & 0.00006 & -1.99993  & 0.00007 & 0.00020  \\
    \hline
    $\vt_3=-2$ & -1.99441  & 0.00559 & 0.00010 & -1.99437 &  0.00563 & 0.00074\\
    \hline
    $\vt_4=0$ &  0.00831 &  0.00831 & 0.00004 & -0.00309 & -0.00309 & 0.00021\\
    \hline
    $\vt_5=1$ & 1.00850 &  0.00850 & 0.00004 &  0.99607 & -0.00393 & 0.00031\\
    \hline
 \hline
\end{tabular}
\caption{Estimation results for an  uncontaminated CARMA$(3,1)$ process with parameter $\vartheta_0=(\vt_1,\vt_2,\vt_3,\vt_4,\vt_5)$ driven by a NIG Lévy process with $r=5$.}\label{Table 4B}
\end{table}

In the first instance, we compare the QMLE and the indirect estimator for uncontaminated \linebreak CARMA$(3,1)$ processes in Table~\ref{Table 4A} and Table~\ref{Table 4B}.
In Table~\ref{Table 4A} the driving Lévy process is  a Brownian motion where in Table~\ref{Table 4B} it is  a NIG-Lévy process. The results are very similar.
The QMLE has in general a lower variance than the indirect estimator. For $n=200$ and $n=1000$ it seems as well that the QMLE has a lower absolute bias. But for $n=5000$ this changes
and the indirect estimator has a lower absolute bias.
However,  both estimator perform excellent. For the Brownian motion driven model the QML optimization failed for $n=200$, $1000$ and $5000$ in 5, 4, and 5 cases, respectively, where for the NIG driven model it failed in 6, 5, and 2 cases, respectively. The indirect estimator never failed. The error occurs when the estimated value of $\vt_0$ is not an element of $\Theta$ anymore. The results in the table are averaged over experiments in which the algorithm did deliver a result, the failed attempts were discarded.
 We obtained similar results as in Table~\ref{Table 4A} and Table~\ref{Table 4B} for different parameter values.

\begin{table}[h] \small
\centering
\begin{tabular}{|l|c|c|c||c|c|c|c|}
\hline
 \; & \multicolumn{6}{c|}{$\xi=5, \gamma =0.1$} \\
 \hline
   \; & \multicolumn{3}{c||}{QMLE} & \multicolumn{3}{c|}{Indirect}  \\
 \hline
  \; & Mean & Bias & Var & Mean & Bias & Var \\
  \hline
    $\vt_1=-1$ &   -0.4288 & 0.5712 & 0.0055 & -1.0031 & -0.0031 & 0.0131\\
  \hline
  $\vt_2=-2$ &   -2.9691 & -0.9691 & 0.0088 & -2.0444 & -0.0444 & 0.0606\\
  \hline
  $\vt_3=-2$ & -2.4261 & -0.4261 & 0.6193  & -1.9969 & 0.0031 & 0.0325\\
  \hline
  $\vt_4=0$ & 1.9267 & 1.9267 & 0.0173 & -0.0157 & -0.0157 & 0.0070 \\
  \hline
  $\vt_5=1$ &  1.8987 & 0.8987 & 0.0028 & 0.9147 & -0.0853 & 0.0243 \\
 \hline
 \hline
 \; & \multicolumn{6}{c|}{$\xi=10, \gamma =0.1$} \\
 \hline
   \; & \multicolumn{3}{c||}{QMLE} & \multicolumn{3}{c|}{Indirect}  \\
 \hline
  \; & Mean & Bias & Var & Mean & Bias & Var \\
  \hline
    $\vt_1=-1$ & -0.0812 & 0.9188 & 0.0012 & -1.0031 & -0.0031 & 0.0131 \\
  \hline
  $\vt_2=-2$ & -3.6798 & -1.6798 & 0.4044 & -2.0446 & -0.0446 & 0.0608  \\
  \hline
  $\vt_3=-2$ &   -3.8853 & -1.8853 & 21.1039 & -1.9966 & 0.0034 & 0.0325 \\
  \hline
    $\vt_4=0$ &  3.9956 & 3.9956 & 0.1610 & -0.0157 & -0.0157 & 0.0070\\
  \hline
  $\vt_5=1$ &  2.3854 & 1.3854 & 0.1283 &  0.9144 & -0.0856 & 0.0243 \\
  \hline
 \hline
 \; & \multicolumn{6}{c|}{$\xi=5, \gamma =\frac16$} \\
 \hline
   \; & \multicolumn{3}{c||}{QMLE} & \multicolumn{3}{c|}{Indirect}  \\
 \hline
  \; & Mean & Bias & Var & Mean & Bias & Var \\
  \hline
    $\vt_1=-1$ & -0.1784 & 0.8216 & 0.0136 & -0.9476 & 0.0524 & 0.0426  \\
  \hline
  $\vt_2=-2$ & -4.8587 & -2.8587 & 15.9100 &  -2.1688 & -0.1688 & 0.1375\\
  \hline
  $\vt_3=-2$ & -12.0764 & -10.0764 & 268.3894  & -1.9481 & 0.0519 & 0.0469 \\
  \hline
    $\vt_4=0$ &  2.8032 & 2.8032 & 0.2815 & -0.0491 & -0.0491 & 0.0101 \\
  \hline
  $\vt_5=1$ & 2.1475 & 1.1475 & 0.2928 & 0.6996 & -0.3004 & 0.0705 \\
  \hline
 \hline
 \; & \multicolumn{6}{c|}{$\xi=5, \gamma =0.25$} \\
 \hline
   \; & \multicolumn{3}{c||}{QMLE} & \multicolumn{3}{c|}{Indirect}  \\
 \hline
  \; & Mean & Bias & Var & Mean & Bias & Var \\
  \hline
    $\vt_1=-1$ &   -0.1427 & 0.8573 & 0.0525  & -0.6035 & 0.3965 & 7.5738 \\
  \hline
  $\vt_2=-2$ &  -6.1993 & -4.1993 & 52.7257 & -3.8476 & -1.8476 & 25.0781\\
  \hline
  $\vt_3=-2$ & -13.7370 & -11.7370 & 767.4184 & -6.0640 & -4.0640 & 417.3265 \\
  \hline
    $\vt_4=0$ &  3.1533 & 3.1533 & 0.2231  & 2.1462 & 2.1462 & 11.3420\\
  \hline
  $\vt_5=1$ &   1.5824 & 0.5824 & 1.8174 & 0.7653 & -0.2347 & 26.1110\\
  \hline
\end{tabular}
\caption{Estimation results for a CARMA$(3,1)$ process with parameter $\vartheta_0=(\vt_1,\vt_2,\vt_3,\vt_4,\vt_5)$ driven by a Brownian motion with $n=1000$ and $r=5$.}\label{Table 3A}
\end{table}

Further, for a Brownian motion driven CARMA(3,1) process we estimate $\vt_0$ for each of the following contamination configurations in \autoref{Table 3A} (see as well Table~\ref{Table 5A} in the Supporting Information for different values of $n$): $\xi = 5$ and $\gamma = 0.1$, $\xi=10$ and $\gamma=0.1$, $\xi=5$ and $\gamma=1/6$, and $\xi=5$ and $\gamma=0.25$. In this situation, the breakdown point has an upper bound of $1/6$ since we have $r=5$. Hence, $\gamma=0.25$ lies above the breakdown point and we expect to encounter problems in the estimation procedure, while for $\gamma \leq 1/6$ these problems should not occur. This is indeed the case.

For the first two experiments, where $\gamma=0.1$, we immediately recognize the maximum likelihood estimate is severely biased and far from the true parameter value. Especially the inclusion of a zero component in the true parameter seems to pose a major problem since this component is affected by the most bias. On the other hand, the indirect estimator is still very close to the true parameter value in all components including the zero component. Increasing $\xi$ to $10$ while keeping $\gamma = 0.1$ results in a very similar performance of the indirect estimator.
 The increase of $\gamma$ from $0.1$ to $1/6$ also affects the performance of the indirect estimator. For all components of $\vt_0$ the absolute bias and the variance of the indirect estimator increase. However, the loss in quality of the indirect estimator is manageable and the calculated estimates still resemble the true parameter. This means that even at the breakdown point of $1/6$, the performance of the indirect estimator is satisfying, although of course not as good as for lower contamination probabilities.

The situation is vastly different in the experiment with $\gamma=0.25$ where $\gamma$ is above the breakdown point.  Here, we see not surprisingly that the indirect estimator, too, gives estimates which are severely biased and quite far away from the true parameters.  We also observe that the numerical procedure used to obtain the parameter estimates  quite often fails to deliver a result.  The ratio of successful to unsuccessful experiments is roughly equal to 1:2, i.e., the algorithm failed about twice as often as it succeeded. In this sense, we can say that the estimator has broken down: for a given outlier--contaminated sample, it either does not return an admissible estimate at all, or, if it does, the estimate is far away from the true parameter. The latter statement is also evident from the fact that the variances of the indirect estimates are far smaller in this case than in other experiments which intuitively means that the algorithm typically returns very similar bad estimates if it returns a result at all.

\section{Conclusion}    \label{sec:Conclusion}
In this paper we presented an indirect estimation procedure for the parameters of a discretely observed CARMA process by estimating the parameters
of its auxiliary AR$(r)$ representation using a GM-estimator. Since there does not exist an explicit form of the map between the AR parameters and the CARMA parameters, an
additional simulation step to get back from the AR parameters to the CARMA parameters was necessary. Sufficient conditions were given such that
the indirect estimator is consistent and asymptotically normally distributed, on the one hand, in a general context, but on the other hand, as well for the special case
 where $\widehat\pi_n=\widehat\pi^{\text{GM}}_n(\vt_0)$ and $\widehat\pi_n^{\text{S}}(\vt)=\widehat\pi^{\text{LS}}_n(\vt)$. Moreover, the indirect estimator
  satisfies different robustness properties as
 weakly resistant, $\pi_{d_n}$-robustness and it has a bounded influence functional.

Summarizing the simulation studies, the indirect estimator performs convincingly for various orders $p$ and $q$ of the CARMA process, for different driving L\'{e}vy processes and for a variety of outlier configurations. The QMLE failed in some simulations but the indirect estimator could be used always. In contrast
to the QMLE, the indirect estimator is robust against outliers where the QMLE is severely biased. For uncontaminated CARMA processes the indirect estimator
is less biased for large $n$ where for small $n$ it is opposite. But in this situation both estimators work quite well. Obviously the bias in the indirect estimation procedure can be decreased by using in the
estimation and in the simulation part the same type of estimator because then the bias from the estimation part
and the simulation part are cancelling out (cf. \citet{GourierouxRenaultTouzi,GouPhillipsYug}). But proving the asymptotic normality of the indirect estimator
using as well in the simulation part the GM estimator is involved and topic of some future research.

%


\section{Proofs} \label{sec:Proofs}

\subsection{Proofs of Section~\ref{sec:Preliminaries}}

\begin{proof}[Proof of \Cref{auxiliaryprop}]
First, we need to show that for any $r \in \N$ the covariance matrix of \linebreak  $(Y_h(\vt), \ldots, Y_{(r+1)h}(\vt))$ is non--singular. To see this, note that the autocovariance function of \linebreak
$(Y_{mh}(\vt))_{m \in \Z}$ is
$\gamma_\vt(mh) = c_\vt^T \e^{A_\vt hm} \Sigma_\vt c_\vt$, $m \in \N_0$ (see \eqref{ACF}).
Since  $\Sigma_\vt$ is non--singular \linebreak (cf. \citet[Corollary 3.9]{schlemmstelzer})  and $c_\vt\not=0_p$  we have that $\gamma_{\vt}(0) > 0$. Moreover, the eigenvalues of $A_\vt$ have strictly negative real parts by \ref{as_D3} and therefore, $\gamma_{\vt}(mh) \to 0$ as $m \to \infty$ holds. By \citet[Proposition 5.1.1]{brockwell}, it  follows that the covariance matrix of $(Y_{h}(\vt), \ldots, Y_{(r+1)h}(\vt))$ is non--singular for every $r \in \N$.
Thus, a conclusion of \citet[\S 8.1]{brockwell} is that there exist unique  $\pi_{ 1}(\vt), \ldots, \pi_{r}(\vt), \sigma^2(\vt)$ which solve the set of $r+1$ Yule--Walker equations, namely
\begin{subequations} \label{pi}
\begin{eqnarray} \label{2.6a}
    \hspace*{-1.5cm}\pi^*(\vt):=
    \begin{pmatrix}
        \pi_1(\vt)\\
        \vdots\\
        \pi_r(\vt)
    \end{pmatrix}
    \hspace*{-0.3cm}&=&\hspace*{-0.3cm}\begin{pmatrix}
        \gamma_\vt(0) & \gamma_\vt(h) & \cdots & \gamma_\vt((r-1)h)\\
        \gamma_\vt(h) & \gamma_\vt(0) & \cdots & \gamma_\vt((r-2)h)\\
        \vdots & \vdots  & & \vdots \\
        \gamma_\vt((r-1)h) &\gamma_\vt((r-2)h) & \cdots & \gamma_\vt(0)
    \end{pmatrix}^{-1}
    \begin{pmatrix}
        \gamma_\vt(h)\\
        \vdots\\
        \gamma_\vt(rh)
    \end{pmatrix} \nonumber \\
    &=:&\Gamma^{(r-1)}(\vt)^{-1}\gamma^{(r-1)}(\vt),\\
     \sigma^2(\vt)&=&\gamma_\vt(0)-\pi^*(\vt)^T\gamma^{(r-1)}(\vt). 
\end{eqnarray}
\end{subequations}
\end{proof}

\begin{proof}[Proof of \Cref{defbindfct}]
We make use of the fact that the discretely observed stationary CARMA$(p,q(\vt))$ process $(Y_{mh}(\vt))_{m \in \Z}$ admits a representation as a stationary ARMA$(p,p-1)$ process with weak white noise as is given in \eqref{chap8:ARMAP}.
Then we can decompose the map $\pi: \Theta \to \Pi$ into three separate maps for which we define the following spaces:
\begin{align*}
\MM := \{ &(\phi_1, \ldots, \phi_p, \theta_1, \ldots, \theta_{p-1}, \sigma) \in \R^{2p}: \text{  The coefficients define a weak ARMA($p$, $p-1$)} \\
&  \text{model as in \eqref{ARMA}  for which $\phi(z)$ and $\theta(z)$ have no common zeros}\} \subseteq \R^{2p}, \\
 \GG := \{ &\gamma = (\gamma_0, \ldots, \gamma_r) \in \R^{r+1}: \text{  The coefficients define the autocovariances up to order}  \\
&\text{ $r$ of a stationary stochastic process where $\Gamma^{(r-1)}$ is non-singular} \} \subseteq \R^{r+1},\\
    \Pi:=\{&(\pi_1,\ldots,\pi_r,\sigma)\in\R^r\times(0,\infty): \text { $(\pi_1,\ldots,\pi_r)$ are the coefficients of a stationary }\\
    & \text{AR$(r)$ process and $\sigma^2$ is the variance of the noise}\}\subseteq \R^{r+1},
\end{align*}
where $\Gamma^{(r-1)}$ is defined as $\Gamma^{(r-1)}(\vt)$ in \eqref{2.6a}.
Denote by $\pi_1: \Theta \to \MM$ the map which maps the parameters of a CARMA process to the coefficients of the weak ARMA($p, p-1$) representation of its sampled version as in \eqref{chap8:ARMAP}. Denote by $\pi_2: \MM \to \GG$ the map which maps the parameters of a weak ARMA($p,p-1$) process to its autocovariances of lags $0, \ldots, r$. Lastly, denote by $\pi_3: \GG \to \Pi$ the map which maps a vector of autocovariances $(\gamma_0,\ldots,\gamma_r)$ to the parameters of the auxiliary AR($r$) model. Then we have that $\pi = \pi_3 \circ \pi_2 \circ \pi_1$.
We will show that $\pi_i$ is injective for $i=1,2,3$ and receive from this the injectivity of $\pi$. The three-times continuous-differentiability of the map $\pi$ follows from the representation \eqref{pi} and the  three-times continuous-differentiability
of the autocovariance function $\gamma_\vt$.

\textbf{Step 1:} $\pi_1$ is injective.\\
Due to \autoref{as_D} and \citet[Theorem 3.13]{schlemmstelzer} the family of sampled processes $\{(Y_{mh}(\vt)_{m\in\Z}:\vt\in\Theta)\}$ is identifiable from
their spectral densities and hence, for any $\vt\not=\vt'\in\Theta$ the parameters of the weak ARMA process
in \eqref{chap8:ARMAP} differ.

\textbf{Step 2:} $\pi_2$ is injective if $r\geq 2p-1$.\\
The reason is that by the method of \citet[p. 93]{brockwell}, the autocovariances  of ARMA($p,p-1$) processes are completely determined as solutions of difference equations with $p$ boundary conditions which depend on the coefficient vector $(\phi_1, \ldots, \phi_p, \theta_1, \ldots, \theta_{p-1}, \sigma)$.  If $r \geq 2p-1$, 
the number of equations $r$ is greater than or equal to the number of variables $2p-1$ which results in the injectivity of $\pi_2$
(see also \citet[Section 4.1]{delunagenton}). To be more precise, let $\mathbold{\theta}=(\phi_1, \ldots, \phi_p, \theta_1, \ldots, \theta_{p-1}, \sigma)\in\MM$ and $\mathbold{\wt\theta}=(\wt\phi_1, \ldots, \wt\phi_p, \wt\theta_1, \ldots, \wt\theta_{p-1}, \wt\sigma)\in \MM$.

\underline{Case 1.} $(\phi_1, \ldots, \phi_p)\not=(\wt\phi_1, \ldots, \wt\phi_p)$. Define $\Gamma^{(p)}(\boldmath{\theta})\in\R^{(p+1)\times(p+1)}$ similarly
to $\Gamma^{(p-1)}(\vt)$ in \eqref{2.6a}. Due to \citet[(3.3.9)]{brockwell}
\begin{eqnarray*}
    ( -\phi_p \quad\ldots \quad -\phi_1\quad 1)\Gamma^{(p)}(\boldmath{\theta})&=&(0\quad 0\quad \cdots \quad 0),\\
    ( -\wt \phi_p \quad\ldots \quad -\wt\phi_1\quad 1)\Gamma^{(p)}(\boldmath{\wt\theta})&=&(0\quad 0\quad \cdots \quad 0).
\end{eqnarray*}
But since the vectors $( -\phi_p \quad\ldots \quad -\phi_1\quad 1)$ and $( -\wt \phi_p \quad\ldots \quad -\wt\phi_1\quad 1)$ are linear independent this is only possible if
$\Gamma^{(p)}(\boldmath{\theta})\not=\Gamma^{(p)}(\boldmath{\wt\theta})$ which implies
$\pi_2(\boldmath{\theta})\not=\pi_2(\boldmath{\wt\theta})$.

\underline{Case 2.} $(\phi_1, \ldots, \phi_p)=(\wt\phi_1, \ldots, \wt\phi_p)$.  Assume that
$\pi_2(\boldmath{\theta})=\pi_2(\boldmath{\wt\theta})$. But then due to
\cite[(3.3.9)]{brockwell}, $(\gamma_{\boldmath{\theta}}(k))_{k\in\N_0}=(\gamma_{\boldmath{\wt\theta}}(k))_{k\in\N_0}$
and hence, $\boldmath{\theta}=\boldmath{\wt\theta}$.

\textbf{Step 3:} $\pi_3$ is injective.\\
We can also rewrite the linear equations \eqref{pi} as a linear system with $(r+1)$ equations and the $(r+1)$ unknown variables $\gamma_0,\ldots,\gamma_r$
which gives the injectivity of $\pi_3$.
\end{proof}

\subsection{Proofs of Section~\ref{sec:indirect estimation}}

\begin{proof}[Proof of \Cref{chap8:IndEstThm}] $\mbox{}$\\
(a) \,
We first start by proving the consistency.
With the definition of $\mathscr{Q}_{\text{Ind}}$ we obtain
\begin{eqnarray*}
\lefteqn{\sup_{\vt \in \Theta} | \LL_{\text{Ind}}(\vt, \mathcal{Y}^n) - \mathscr{Q}_{\text{Ind}}(\vt)|} \\
&=&\sup_{\vt \in \Theta} | [\widehat{\pi}_n - \wh{\pi}_{sn}^{\text{S}}(\vt)]^T \Omega [\widehat{\pi}_n - \wh{\pi}_{sn}^{\text{S}}(\vt)] -
    [\pi(\vt) - \pi(\vt_0)]^T \Omega [\pi(\vt) - \pi(\vt_0)]| \\
&\leq &| \widehat{\pi}_n^T \Omega \widehat{\pi}_n - \pi(\vt_0)^T \Omega \pi(\vt_0) | + \sup_{\vt \in \Theta} | \wh{\pi}_{sn}^{\text{S}}(\vt)^T \Omega
\widehat{\pi}_n - \pi(\vt)^T \Omega \pi(\vt_0) | \\
&&+
\sup_{\vt \in \Theta} |  \widehat{\pi}_n^T \Omega \wh{\pi}_{sn}^{\text{S}}(\vt)  - \pi(\vt_0)^T \Omega \pi(\vt)  |
+ \sup_{\vt \in \Theta} | \wh{\pi}_{sn}^{\text{S}}(\vt)^T \Omega \wh{\pi}_{sn}^{\text{S}}(\vt)  - \pi(\vt)^T \Omega \pi(\vt)  |.
\end{eqnarray*}
The four summands on the right--hand side  converge in probability to zero as $n \to \infty$. For the first one, this is a consequence of \ref{chap8:UnifConv1}. For the remaining three terms, the arguments are similar, so that we treat only the second one exemplary.  We have
\begin{eqnarray*}
\lefteqn{\sup_{\vt \in \Theta} | \wh{\pi}_{sn}^{\text{S}}(\vt)^T \Omega \widehat{\pi}_n - \pi(\vt)^T \Omega \pi(\vt_0) |} \\
&\leq& \sup_{\vt \in \Theta} | \wh{\pi}_{sn}^{\text{S}}(\vt)^T \Omega \widehat{\pi}_n - \pi(\vt)^T \Omega \wh \pi_n| + \sup_{\vt \in \Theta} |\pi(\vt)^T \Omega \wh \pi_n- \pi(\vt)^T \Omega \pi(\vt_0)|\\
&\leq &\|\Omega\|\sup_{\vt \in \Theta} \| \wh{\pi}_{sn}^{\text{S}}(\vt) - \pi(\vt) \| \| \widehat{\pi}_n  \| + \|\Omega\|\sup_{\vt \in \Theta} \|  \pi(\vt) \| \| \wh \pi_n- \pi(\vt_0) \|
\stoch 0.
\end{eqnarray*}
Here, we used  the fact that $ \sup_{\vt \in \Theta} \| \pi(\vt) \|$ is finite due to the continuity of the map $\pi$ and the compactness of $\Theta$ as well as both \ref{chap8:UnifConv1} and \ref{chap8:UnifConv}.
Therefore, the function $\LL_{\text{Ind}}(\vt, \mathcal{Y}^n)$
converges uniformly in $\vt$ in probability to the limiting function  $\mathscr{Q}_{\text{Ind}}(\vt)$.
Per construction, $\wh{\vt}_n^{\text{Ind}}$ minimizes $\LL_{\text{Ind}}(\vt, \mathcal{Y}^n)$ and $\mathscr{Q}_{\text{Ind}}(\vt)$ has a unique minimum at $\vt = \vt_0$. Therefore, weak consistency of $\wh{\vt}_n^{\text{Ind}}$ follows by arguing  as in the proof of \citet[Theorem 2.4]{schlemmstelzer}; although in their proof convergence in probability is replaced by almost sure convergence, this doesn't matter  because we can use the subsequence criterion which says that a   sequence converges in probability
iff any subsequence has a further subsequence which converges almost surely.

The proof of strong consistency goes
similarly by replacing convergence in probability by almost sure convergence.

(b) \,
For the asymptotic normality, note that
$$\sqrt{n} (\widehat{\pi}_n - \wh{\pi}_{sn}^{\text{S}}(\vt_0)) = \sqrt{n} (\widehat{\pi}_n - \pi(\vt_0)) + \sqrt{n} ( \pi(\vt_0) - \wh{\pi}_{sn}^{\text{S}}(\vt_0)).$$
Since both estimators are independent from each other, we obtain with \ref{ThmIndEstAssumption1} and \ref{ThmIndEstAssumption2} that
\begin{equation}\label{chap8:EqInProof}
\sqrt{n} (\widehat{\pi}_n - \wh{\pi}_{sn}^{\text{S}}(\vt_0)) \stackrel{\DD}{\longrightarrow} \NN \left(0, \Xi_{\text{D}}(\vt_0) + \frac1s \Xi_{S}(\vt_0) \right).
\end{equation}
Moreover,
$$0 =  \nabla_\vt \LL_{\text{Ind}} ( \vt, \mathcal{Y}^n ) \big|_{\vt= \wh{\vt}_n^{\text{Ind}}} =  2[\nabla_\vt \wh{\pi}_{sn}^{\text{S}}( \wh{\vt}_n^{\text{Ind}})]^T \Omega [\wh{\pi}_{sn}^{\text{S}}( \wh{\vt}_n^{\text{Ind}})-\widehat{\pi}_n].$$
We now use a Taylor expansion of order $1$ around the true value $\vt_0$ to obtain
\begin{eqnarray*}
0 &=& \sqrt{n}  \nabla_\vt \LL_{\text{Ind}} ( \wh{\vt}^n_{\text{Ind}}, \mathcal{Y}^n ) \\
  &=& \sqrt{n}  \nabla_\vt \LL_{\text{Ind}} ( \vt_0, \mathcal{Y}^n ) + \sqrt{n} \nabla^2_\vt \LL_{\text{Ind}} ( \overline{\vt}_n , \mathcal{Y}^n )(\wh{\vt}_n^{\text{Ind}} - \vt_0) \\
  &=&2[\nabla_\vt \wh{\pi}_{sn}^{\text{S}}( \vt_0)]^T \Omega \sqrt{n} [\wh{\pi}_{sn}^{\text{S}}( \vt_0 )-\widehat{\pi}_n ]
       + 2[\nabla_\vt^2 \wh{\pi}_{sn}^{\text{S}}(\overline{\vt}_n)]^T \Omega [\wh{\pi}_{sn}^{\text{S}}(\overline{\vt}_n )-\widehat{\pi}_n ] \sqrt{n} (\wh{\vt}_n^{\text{Ind}} - \vt_0)\\
      &&\quad + 2[\nabla_\vt \wh{\pi}_{sn}^{\text{S}}( \overline{\vt}_n)]^T \Omega[\nabla_\vt \wh{\pi}_{sn}^{\text{S}}( \overline{\vt}_n)] \sqrt{n} (\wh{\vt}_n^{\text{Ind}} - \vt_0).
\end{eqnarray*}
Here, $\overline{\vt}_n$ is such that $\| \overline{\vt}_n - \vt_0 \| \leq \| \wh{\vt}_n^{\text{Ind}} - \vt_0 \|$ and hence,  $\overline{\vt}_n \stoch \vt_0$ as $n \to \infty$.
Moreover,
\begin{eqnarray} \label{3.3}
    [\nabla_\vt^2 \wh{\pi}_{sn}^{\text{S}}(\overline{\vt}^n)]^T \Omega [\wh{\pi}_{sn}^{\text{S}}(\overline{\vt}_n )-\widehat{\pi}_n ]+
[\nabla_\vt \wh{\pi}_{sn}^{\text{S}}( \overline{\vt}_n)]^T \Omega[\nabla_\vt \wh{\pi}_{sn}^{\text{S}}( \overline{\vt}_n)]\stoch [\nabla_\vt \pi(\vt_0)]^T \Omega [\nabla_\vt \pi(\vt_0)]
\end{eqnarray}
due to \ref{chap8:UnifConv1}, \ref{chap8:UnifConv}, \ref{ThmIndEstAssumption3} and  the continuity of $\pi(\vt)$. Furthermore, the right-hand side is non-singular since $\nabla_\vt \pi(\vt_0)$ has full column rank and $\Omega$ is non-singular.
Finally, we write
\begin{align*}
&\sqrt{n} (\wh{\vt}_n^{\text{Ind}} - \vt_0)\\
 &= \left([\nabla_\vt^2 \wh{\pi}_{sn}^{\text{S}}(\overline{\vt}^n)]^T \Omega [\wh{\pi}_{sn}^{\text{S}}(\overline{\vt}_n )-\widehat{\pi}_n ]+
[\nabla_\vt \wh{\pi}_{sn}^{\text{S}}( \overline{\vt}_n)]^T \Omega[\nabla_\vt \wh{\pi}_{sn}^{\text{S}}( \overline{\vt}_n)]\right)^{-1} [\nabla_\vt \wh{\pi}_{sn}^{\text{S}}( \vt_0)]^T \Omega \sqrt{n} ( \wh{\pi}_{sn}^{\text{S}}( \vt_0 )-\widehat{\pi}_n)
\end{align*}
and use \eqref{chap8:EqInProof}, \eqref{3.3} and \ref{ThmIndEstAssumption3} to  obtain as $n \to \infty$,
\begin{equation*}\label{chap8:FinalEq}
\sqrt{n} (\wh{\vt}_n^{\text{Ind}} - \vt_0) \stackrel{\DD}{\longrightarrow} \left( [\nabla_\vt \pi(\vt_0)]^T \Omega [\nabla_\vt \pi(\vt_0)] \right)^{-1}  [\nabla_\vt \pi(\vt_0)]^T \Omega \cdot \NN \left(0, \Xi_{\text{D}}(\vt_0) + \frac1s \Xi_{S}(\vt_0) \right).
\end{equation*}
This completes the proof.
\end{proof}

\subsection{Proofs of Section~\ref{sect:GM}}

\begin{proof}[Proof of \Cref{uniquenessprop}]
Using similar arguments as in \citet[Lemma 2.1]{maronnayohai} \linebreak ($\lim_{x\to 0}\chi(x)<0$, $\lim_{x\to\infty}\chi(x)=\infty$, the continuity and boundedness of $\chi$
and the Intermediate Value Theorem) we can show that for each fixed $(\pi_1, \ldots, \pi_r) \in \R^r$ there exists a unique solution $\sigma$ of the equation
$$\E \left[ \chi \left( \left( \frac{Y_{(r+1)h}(\vt) - \pi_1 Y_{rh}(\vt) - \ldots - \pi_r Y_h(\vt)}{\sigma} \right)^2 \right) \right] = 0.$$
By assumption \eqref{4.5}, the function $\chi$ is chosen in such a way that for $(\pi_{1}(\vt), \ldots, \pi_{r}(\vt))$ this unique solution is $\sigma(\vt)$. Therefore, we have that $\pi(\vt)$ is a solution of \eqref{pseudo:sigma}. Next, we show that $\pi(\vt)$ is a solution of \eqref{pseudo:pi} as well. Since the function $\phi(y,u)$ is odd in $u$ by Assumption \ref{as_H1}, it holds that
\begin{eqnarray}\label{AuxParamUnique}
\E \left[ \phi \left( \begin{pmatrix} Y_{h}(\vt) \\ \vdots \\ Y_{rh}(\vt) \end{pmatrix},  \frac{ U_{r+1}(\vt)}{\sigma(\vt)} \right) \begin{pmatrix} Y_{h}(\vt) \\ \vdots \\ Y_{rh}(\vt) \end{pmatrix} \right]
&=& \E \left[ -\phi \left( \begin{pmatrix} Y_{h}(\vt) \\ \vdots \\ Y_{rh}(\vt) \end{pmatrix},  -\frac{ U_{r+1}(\vt)}{\sigma(\vt)} \right) \begin{pmatrix} Y_{h}(\vt) \\ \vdots \\ Y_{rh}(\vt) \end{pmatrix} \right] \notag \\
&&\hspace*{-2cm}= -\E \left[ \phi \left( \begin{pmatrix} Y_{h}(\vt) \\ \vdots \\ Y_{rh}(\vt) \end{pmatrix},  \frac{ U_{r+1}(\vt)}{\sigma(\vt)} \right) \begin{pmatrix} Y_{h}(\vt) \\ \vdots \\ Y_{rh}(\vt) \end{pmatrix} \right],
\end{eqnarray}
where the last equality follows from \eqref{sym}. From this equation we can conclude that
$$\E \left[ \phi \left( \begin{pmatrix} Y_{h}(\vt) \\ \vdots \\ Y_{rh}(\vt) \end{pmatrix},  \frac{ U_{r+1}(\vt)}{\sigma(\vt)} \right) \begin{pmatrix} Y_{h}(\vt) \\ \vdots \\ Y_{rh}(\vt) \end{pmatrix} \right] = 0,$$
and  therefore, $\pi(\vt)$ is a solution of  equation \eqref{pseudo:pi}.

Next, we show similarly to \citet[Theorem 2.2(a)]{maronnayohai} for regression models that $\pi(\vt)$  is
the unique solution. Assume that another solution $\pi'=(\pi'_1, \ldots, \pi'_r, \sigma')$ of \eqref{pseudo}
exists. But then $(\pi'_1, \ldots, \pi'_r)\not=(\pi_1, \ldots, \pi_r)$. Note that the arguments in the derivation of \eqref{AuxParamUnique} still hold if we replace $\sigma(\vt)$ in the denominator of the second argument of $\phi$ by $\sigma'$. Thus, we obtain that
$$\E \left[ \phi \left( \begin{pmatrix} Y_{h}(\vt) \\ \vdots \\ Y_{rh}(\vt) \end{pmatrix},  \frac{ U_{r+1}(\vt)}{\sigma'} \right) \begin{pmatrix} Y_{h}(\vt) \\ \vdots \\ Y_{rh}(\vt) \end{pmatrix} \right] = 0,$$
and therefore
\begin{align}
&\hspace*{-1cm}\E \left[ \left[ \phi \left( \begin{pmatrix} Y_{h}(\vt) \\ \vdots \\ Y_{rh}(\vt) \end{pmatrix},  \frac{ Y_{(r+1)h}(\vt) - \pi'_1 Y_{rh}(\vt) -  \ldots - \pi'_r Y_{h}(\vt)}{\sigma'} \right)\begin{pmatrix} Y_{h}(\vt) \\ \vdots \\ Y_{rh}(\vt) \end{pmatrix} \right. \right.\notag\\
& \left. \left. - \phi \left( \begin{pmatrix} Y_{h}(\vt) \\ \vdots \\ Y_{rh}(\vt) \end{pmatrix},  \frac{ U_{r+1}(\vt)}{\sigma'} \right) \right] \begin{pmatrix} Y_{h}(\vt) \\ \vdots \\ Y_{rh}(\vt) \end{pmatrix} \right] = 0. \label{AuxParamUnique2}
\end{align}
Since $\P( (Y_{h}(\vt), \ldots, Y_{rh}(\vt)) = (0, \ldots, 0)) = 0$ and $\P( Y_{(r+1)h}(\vt) - \pi'_1 Y_{rh}(\vt) -  \ldots - \pi'_r Y_{h}(\vt)$ \linebreak $ = U_{r+1}(\vt))=0$ due to \ref{as_G4}
for $\gamma=0$
and $u\mapsto \phi(y,u)$ is strictly increasing on the interval $(-u_0, u_0)$, for every $y \in \R^r$ we have that
\begin{equation}\label{AuxParamUnique3}
\left| \frac{ Y_{(r+1)h}(\vt) - \pi'_1 Y_{rh}(\vt) -  \ldots - \pi'_r Y_{h}(\vt)}{\sigma'} \right| \geq u_0 \quad \; \Pas
\end{equation}
because otherwise \eqref{AuxParamUnique2} cannot hold.
Now, $\pi'$ is by assumption also a solution of \eqref{pseudo:sigma} and hence, we have due to \eqref{AuxParamUnique3} and \ref{as_H6}
\begin{align*}
0 = \E \left[ \chi \left( \left( \frac{ Y_{(r+1)h}(\vt) - \pi'_1 Y_{rh}(\vt) -  \ldots - \pi'_r Y_{h}(\vt)}{\sigma'} \right)^2 \right) \right] \geq \chi(u^2_0) > 0
\end{align*}
which is a contradiction.
\end{proof}

\begin{proof}[Proof of \Cref{AsNorm}]
By the Cramer--Wold device, the statement of the lemma is equivalent to the assertion that $\frac{1}{\sqrt{n-r}} x^T \sum_{k=1}^{n-r} \Psi_k(\vt^\gamma)$ converges to a univariate normal distribution with mean $0$ and variance  $x^T \II_{\text{GM}}(\vt^\gamma) x$ for every $x \in \R^{r+1}$. According to \citet[Theorem 1.7]{ibragimov}, this holds if we can show that
\begin{eqnarray} \label{eq1}
    \E | x^T \Psi_k(\vt^\gamma) |^{2+\delta} < \infty
\end{eqnarray}
and that
$(x^T \Psi_k(\vt^\gamma))_{k\in \N}$ is strongly mixing with mixing coefficients $\alpha_{x^T \Psi(\vt^\gamma)}(m)$ satisfying
\begin{eqnarray}  \label{eq2}
    \sum_{m=1}^\infty \alpha_{x^T \Psi(\vt^\gamma)}^{\delta/(2+\delta)}(m) < \infty \quad \text{ for some }\delta > 0.
\end{eqnarray}
The same theorem then also states that $x^T \II_{\text{GM}}(\vt^\gamma) x < \infty$ from which we then deduce that for $i, j \in \{ 1, \ldots, r+1 \}$ the entry $[\II_{\text{GM}}(\vt^\gamma)]_{ij}$ is finite and therefore, $\II_{\text{GM}}(\vt^\gamma)$ is well--defined.

We start to show the   existence of the $(2+\delta)$--th moment of $x^T \Psi_k(\vt^\gamma)$ in \eqref{eq1}. Therefore, note that
\begin{align}\label{chap8:boundedness}
\E  |x^T \Psi_k(\vt^\gamma) |^{2+\delta}
\leq C \|x\|^{2+\delta} \sum_{i=1}^{r+1} \E  \|\Psi_{k,i}(\vt^\gamma)\|^{2+\delta}  < \infty,
\end{align}
where the last inequality holds since $\Psi_{k,i}(\vt^\gamma)$ is bounded by \ref{as_H2} and \ref{as_H6}.

Finally, the process $(Y_{mh}^\gamma(\vt))_{m \in \Z}$ is strongly mixing and the mixing coefficients satisfy the above condition \eqref{eq2} for the following reason.
Either we have in the case of replacement outliers that  $Y_{mh}^\gamma(\vt) = G(V_m, Z_m, Y_{mh}(\vt))$ for some measurable function $G$
and the three processes
 $(V_m)$, $(Z_m)$ and $(Y_{mh}(\vt))$ are independent, or in the case of additive outliers we have $Y_{mh}^\gamma(\vt) = G(V_m, W_m, Y_{mh}(\vt))$ for some measurable function $G$
and the three processes
 $(V_m)$, $(W_m)$ and $(Y_{mh}(\vt))$ are independent. Hence, by \citet[Theorem 6.6(II)]{bradley}, Assumption \ref{as_G2} and \citet[Proposition~3.34]{MarquardtStelzer2007} we receive
$$\alpha_{Y^\gamma(\vt)}(m) \leq \alpha_V(m) + \alpha_Z(m) + \alpha_{Y(\vt)}(m) \leq C\rho^m$$
respectively,
$\alpha_{Y^\gamma(\vt)}(m) \leq \alpha_V(m) + \alpha_W(m) + \alpha_{Y(\vt)}(m) \leq C\rho^m$
for some $C > 0$ and $\rho \in (0,1)$. Furthermore, $\Psi_k(\vt^\gamma)$ depends only on the finitely many values  $Y_{kh}^\gamma(\vt), \ldots, Y_{(k+r)h}^\gamma(\vt)$ and by \citet[Remark 1.8(b)]{bradley} this ensures that $\alpha_{\Psi(\vt^\gamma)}(m) \leq \alpha_{Y^{\gamma}(\vt)}(m+r) \leq C \rho^m$. Thus, the strong mixing coefficients $\alpha_{x^T \Psi(\vt^\gamma)}(m)$ of $x^T \Psi(\vt^\gamma)$ satisfy the summability condition \eqref{eq2} and the lemma is proven.
\end{proof}

\begin{proof}[Proof of \Cref{Lemma 4.8}]
Note, first that for $i, j=1, \ldots, r$,
\begin{eqnarray*}
\lefteqn{\sup_{ \pi \in K }\left| \frac{\partial}{\partial \pi_i} \phi \left( \begin{pmatrix}Y_h^\gamma(\vt) \\ \vdots \\ Y_{rh}^\gamma(\vt) \end{pmatrix}, \frac{Y_{(r+1)h}^\gamma(\vt) - \pi_1 Y_{rh}^\gamma(\vt) - \ldots - \pi_{r} Y_{h}^\gamma(\vt)}{ \sigma } \right) Y_{jh}^\gamma(\vt) \right|}  \\
&=&\sup_{ \pi \in K}   \left\| \left(\frac{\partial}{\partial u}\phi \right) \left( \begin{pmatrix}Y_h^\gamma(\vt) \\ \vdots \\ Y_{rh}^\gamma(\vt) \end{pmatrix}, \frac{Y_{(r+1)h}^\gamma(\vt) - \pi_1 Y_{rh}^\gamma(\vt) - \ldots - \pi_{r} Y_{h}^\gamma(\vt)}{ \sigma } \right)Y_{jh}^\gamma(\vt) \right\|  \left|\frac{Y_{(r+1-i)h}^\gamma(\vt)}{\sigma}\right| \\
&\leq& \sup_{ u \in \R } C  \left\| \left(\frac{\partial}{\partial u}\phi \right) \left( \begin{pmatrix}Y_h^\gamma(\vt) \\ \vdots \\ Y_{rh}^\gamma(\vt) \end{pmatrix}, u \right) \begin{pmatrix}Y_h^\gamma(\vt) \\ \vdots \\ Y_{rh}^\gamma(\vt) \end{pmatrix} \right\|  \left\| \begin{pmatrix}Y_h^\gamma(\vt) \\ \vdots \\ Y_{rh}^\gamma(\vt) \end{pmatrix} \right\|
\leq C \left\| \begin{pmatrix}Y_h^\gamma(\vt) \\ \vdots \\ Y_{rh}^\gamma(\vt) \end{pmatrix} \right\|
\end{eqnarray*}
due to Assumption \ref{as_H4} and the boundedness of $1/\sigma$ on the compact set $K$.  By Assumption \ref{as_G1} and \ref{as_D2} the expectation on the right-hand side is finite. Similarly,
\begin{eqnarray*}
\lefteqn{\hspace*{-3cm}\sup_{ \pi \in K } \left\| \frac{\partial}{\partial \sigma} \phi \left( \begin{pmatrix}Y_h^\gamma(\vt) \\ \vdots \\ Y_{rh}^\gamma(\vt) \end{pmatrix}, \frac{Y_{(r+1)h}^\gamma(\vt) - \pi_1 Y_{rh}^\gamma(\vt) - \ldots - \pi_{r} Y_{h}^\gamma(\vt)}{ \sigma } \right) \begin{pmatrix}Y_h^\gamma(\vt) \\ \vdots \\ Y_{rh}^\gamma(\vt) \end{pmatrix} \right\| }\\
&&\hspace*{-3cm}\leq C \sup_{u \in \R} \left\| u \left(\frac{\partial}{\partial u}\phi \right) \left( \begin{pmatrix}Y_h^\gamma(\vt) \\ \vdots \\ Y_{rh}^\gamma(\vt) \end{pmatrix}, u \right) \begin{pmatrix}Y_h^\gamma(\vt) \\ \vdots \\ Y_{rh}^\gamma(\vt) \end{pmatrix} \right\|.
\end{eqnarray*}
The expectation on the right-hand side is finite due to Assumption \ref{as_H5}. Similar arguments, using Assumption \ref{as_H6}, also show that $\left| \frac{\partial}{\partial \pi_i} \chi \left( \left( \frac{Y_{(r+1)h}^\gamma(\vt) - \pi_1 Y_{rh}^\gamma(\vt) - \ldots - \pi_{r} Y_{h}^\gamma(\vt)}{ \sigma } \right)^2 \right) \right|$  for $i= 1, \ldots, r$ and \linebreak $ \left| \frac{\partial}{\partial \sigma} \chi \left( \left( \frac{Y_{(r+1)h}^\gamma(\vt) - \pi_1 Y_{rh}^\gamma(\vt) - \ldots - \pi_{r} Y_{h}^\gamma(\vt)}{ \sigma } \right)^2 \right) \right| $ are uniformly dominated by integrable random variables. Therefore, by
\cite[Theorem 16.8(ii)]{Billingsley1999} (that is an application of dominated convergence) $\nabla_\pi \mathscr{Q}_{\text{GM}}(\pi,\vt^\gamma)$ exists on $K$ and the order of differentiation and
expectation can be changed.

Moreover, due  \ref{as_H4}, \ref{as_H6} and  \cite[Theorem 16.8(i)]{Billingsley1999} the map $\pi\mapsto \nabla_\pi \mathscr{Q}_{\text{GM}} (\pi,\vt^\gamma)$ is continuous.
Hence, if $\pi_n\stackrel{\P}{\to}\pi^{\text{GM}}(\vt^\gamma)\in K$  then
$ \nabla_\pi \mathscr{Q}_{\text{GM}}(\pi_n,\vt^\gamma) \stackrel{\P}{\to} \nabla_\pi \mathscr{Q}_{\text{GM}}(\pi^{\text{GM}}(\vt^\gamma),\vt^\gamma)$.
\end{proof}

\begin{proof}[Proof of \Cref{Lemma 4.7}]
We use the decomposition
\begin{equation*}\label{chap8:proofNorm2}
\sqrt{n-p}  \mathscr{Q}_{\text{GM}} ( \widehat{\pi}_n^{\text{GM}}(\vt^\gamma),\vt^\gamma) =   \frac{1}{\sqrt{n-r}} \sum_{k=1}^{n-r} [ \mathscr{Q}_{\text{GM}} ( \widehat{\pi}_n^{\text{GM}}(\vt^\gamma),\vt^\gamma)+\Psi_k(\vt^\gamma)]- \frac{1}{\sqrt{n-r}} \sum_{k=1}^{n-r} \Psi_k(\vt^\gamma).
\end{equation*}
The first term is  of order $o_P(1)$ due to \cite[Lemma 3.5]{bustos} (cf. \citet[Lemma A.5]{Kimmig} in our setting). The second term
converges to $ \NN(0, \II_{\text{GM}}(\vt^\gamma))$ due to \Cref{AsNorm}. Hence, we receive the statement.
\end{proof}

\begin{proof}[Proof of \Cref{chap8:ThmNorm}]
Due to \eqref{pseudo}  we have  $\mathscr{Q}_{\text{GM}}(\pi(\vt^\gamma),\vt^\gamma) = 0$. Next,  a first-order Taylor expansion  around
$ \widehat{\pi}_n^{\text{GM}}(\vt^\gamma)$ gives
\begin{align*} 
0 &= \sqrt{n-r} \mathscr{Q}_{\text{GM}}( \pi(\vt^\gamma),\vt^\gamma) \nonumber \\
  &= \sqrt{n-r} \mathscr{Q}_{\text{GM}}( \widehat{\pi}_n^{\text{GM}}(\vt^\gamma),\vt^\gamma) + \sqrt{n-r} \nabla_\pi \mathscr{Q}_{\text{GM}} ( \overline{\pi}_n^{\text{GM}}(\vt^\gamma),\vt^\gamma) (\pi^{\text{GM}}(\vt^\gamma) - \widehat{\pi}_n^{\text{GM}}(\vt^\gamma) ),
\end{align*}
where $\| \pi^{\text{GM}}(\vt^\gamma) - \overline{\pi}_n^{\text{GM}}(\vt^\gamma) \| \leq \| \pi^{\text{GM}}(\vt^\gamma) - \widehat{\pi}_n^{\text{GM}}(\vt^\gamma) \|$.
The statement follows then from a combination of \Cref{Lemma 4.8} and \Cref{Lemma 4.7}.
\end{proof}

\subsection{Proofs of Section~\ref{sec:indirect estimator CARMA}}

For the ease of notation we write in the following for the Lévy process $(L_t^S)_{t\in\R}$ shortly $(L_t)_{t\in\R}$
and hence, assume that $\E|L_1|^{2N^*}$ for some $2N^*>\max(N(\Theta),4+\delta)$; similarly $(Y_t^S)_{t\in\R}$
is $(Y_t)_{t\in\R}$.

\begin{lemma} \label{Lemma 5.1}
    Define for any $\vt\in\Theta$ the function $f_\vt(u)=c_\vt^{T}\e^{A_\vt u}e_p 1_{\left[0,\infty\right)}(u)$
    and
    \begin{eqnarray*}
        G_\vt(u)=\left(f_\vt(u),\nabla_\vt f_\vt(u),\nabla_\vt^2 f_\vt(u)\right).
    \end{eqnarray*}
    Then $\Pas$ we have
    \begin{eqnarray*}
        \left(Y_{mh}(\vt),\nabla_\vt Y_{mh}(\vt), \nabla_\vt^2 Y_{mh}(\vt)\right)_{m\in\Z}
        =\left(\int_{-\infty}^{mh} G_\vt(mh-u)\,d L_u\right)_{m\in\Z}
    \end{eqnarray*}
    which is strongly mixing and ergodic.
\end{lemma}
The proof is moved to Section~\ref{Appendix:Lemma 5.1} in the Supporting Information.

\begin{proof}[Proof of \Cref{Konv Kov}]
(a) \, First, we prove the pointwise convergence of the sample autocovariance function and second, that $\wh\gamma_{\vt,n}(l,j)$ is locally Hölder-continuous which results in a
stochastic equicontinuity condition.  Then we are able to apply
\citet[Theorem 10.2]{Pollard}  which gives the uniform convergence.\\
\textbf{Step 1. Pointwise convergence.} From \Cref{Lemma 5.1} we already know that $(Y_{mh}(\vt))_{m\in\Z}$ is a stationary and ergodic sequence with $\E|Y_{mh}(\vt)|^2<\infty$
due to $\E|L_1|^2<\infty$.
Then Birkoff's Ergodic Theorem gives as $n\to\infty$,
\begin{eqnarray*}
    \wh\gamma_{\vt,n}(l,j)\stoch \gamma_{\vt}(l-j).
\end{eqnarray*}
\textbf{Step 2. $\wh\gamma_{\vt,n}(l,j)$ is locally Hölder-continuous.}
    Let $\gamma\in \left[0,1-N(\Theta)/(2N^*)\right)$ and
    \begin{eqnarray*}
        U_k:=\sup_{0<\|\vt_1-\vt_2\|<1 \atop \vt_1,\vt_2\in\Theta}\frac{|Y_{kh}(\vt_1)-Y_{kh}(\vt_2)|}{\|\vt_1-\vt_2\|^\gamma}.
    \end{eqnarray*}
    Since $((Y_{mh}(\vt))_{\vt\in\Theta})_{m\in\Z}$ is a stationary sequence, $U_k\stackrel{d}{=} U_1$ and due to \Cref{Lemma A.3} in the Supporting Information, $\E U_1^{2N^*}<\infty$.
    In particular, for any $\vt_1,\vt_2\in\Theta$ with $\|\vt_1-\vt_2\|<1$ the upper bound
    \begin{eqnarray*}
        |Y_{kh}(\vt_1)-Y_{kh}(\vt_2)|\leq U_k\|\vt_1-\vt_2\|^\gamma
    \end{eqnarray*}
    and hence,
    \begin{eqnarray*}
        \lefteqn{|Y_{(k+l)h}(\vt_1)Y_{(k+j)h}(\vt_1)-Y_{(k+l)h}(\vt_2)Y_{(k+j)h}(\vt_2)|}\\
        &&\leq \underbrace{\left(\sup_{\vt\in\Theta}|Y_{(k+l)h}(\vt)|+\sup_{\vt\in\Theta}|Y_{(k+j)h}(\vt)|\right)(U_{k+l}+U_{k+j})}_{=:U_{k+l,k+j}^*}\|\vt_1-\vt_2\|^\gamma
    \end{eqnarray*}
    hold.
    Finally,
    \begin{eqnarray} \label{A.6}
        |\widehat\gamma_{\vt_1,n}(l,j)-\widehat\gamma_{\vt_2,n}(l,j)|\leq \frac{1}{n-r}\sum_{k=1}^{n-r} U_{k+l,k+j}^*\|\vt_1-\vt_2\|^\gamma \quad \text{ for }\|\vt_1-\vt_2\|<1
    \end{eqnarray}
    with
    \begin{eqnarray*}
        \E (U_{k+l,k+j}^*)=  \E (U_{1+l,1+j}^*)\leq 4\left(\E\left(\sup_{\vt\in\Theta}Y_h(\vt)^2\right)\E U_{1}^2\right)^{1/2}<\infty
    \end{eqnarray*}
    where we used \Cref{Lemma A.3} in the Supporting Information to get the finite expectation.\\
\textbf{Step 3.}
    Let $\epsilon,\eta>0$. Take $0<\delta<\min\{1,\eta\epsilon/\E(U_{1+l,1+j}^*)\}^{1/\gamma}$. Then \eqref{A.6} and Markov's inequality give
    \begin{eqnarray*}
        \P\left(\sup_{0<\|\vt_1-\vt_2\|<\delta \atop \vt_1,\vt_2\in\Theta} |\widehat\gamma_{\vt_1,n}(l,j)-\widehat\gamma_{\vt_2,n}(l,j)|>\eta\right)
            \leq \E(U_{1+l,1+j}^*)\frac{\delta^\gamma}{\eta}<\epsilon.
    \end{eqnarray*}
    A conclusion of this stochastic equicontinuity condition, the pointwise convergence in Step 1  and \citet[Theorem 10.2]{Pollard}
    is the uniform convergence.

    The proof of (b,c) goes in the same vein as the proof of (a).
\end{proof}

\begin{proof}[Proof of \Cref{proposition 5.2}]
Define
\begin{eqnarray*}
    \wh\gamma_n^{(r-1)}(\vt)=\begin{pmatrix}\wh\gamma_{\vt,n}(r,r-1) \\ \vdots \\\wh\gamma_{\vt,n}(r,0) \end{pmatrix}
    \quad
    \mbox{ and } \quad
    \wh \Gamma_n^{(r-1)}(\vt)=\begin{pmatrix}\wh\gamma_{\vt,n}(r-1,r-1) & \cdots &\wh\gamma_{\vt,n}(0,r-1)\\
        \vdots & & \vdots\\
        \wh\gamma_{\vt,n}(r-1,0) & \cdots &\wh\gamma_{\vt,n}(0,0)
     \end{pmatrix}.
\end{eqnarray*}
Then
\begin{eqnarray}    \label{5.1}
\begin{split}
  \wh\pi_n^*(\vt) :=&\begin{pmatrix} \wh{\pi}_{n,1}^{\text{LS}}(\vt)\\ \vdots \\\wh{\pi}_{n,r}^{\text{LS}}(\vt) \end{pmatrix}
    =  [ \wh \Gamma_n^{(r-1)}(\vt)]^{-1}\wh\gamma_n^{(r-1)}(\vt),\\
    \sigma^2_{\text{LS},n}(\vt)=&\wh\gamma_{\vt,n}(r,r)- [\wh\pi_n^*(\vt)]^{T}\wh\gamma_n^{(r-1)}(\vt). 
\end{split}
\end{eqnarray}
A conclusion of \Cref{Konv Kov}(a) and the definition of $\Gamma^{(r-1)}(\vt)$ and $\gamma^{(r-1)}(\vt)$ in \eqref{pi}  is that
\begin{eqnarray} \label{5.3}
    \sup_{\vt\in\Theta}\|\wh \Gamma_n^{(r-1)}(\vt)-\Gamma^{(r-1)}(\vt)\|\stoch 0 \quad \text{ and } \quad
    \sup_{\vt\in\Theta}\|\wh\gamma_n^{(r-1)}(\vt)-\gamma^{(r-1)}(\vt)\|\stoch 0.
\end{eqnarray}
Due to the continuity and the positive definiteness of $\Gamma^{(r-1)}(\vt)$ (cf. proof of \Cref{auxiliaryprop}), and the compactness of
$\Theta$ we receive $\sup_{\vt\in\Theta}\|\Gamma^{(r-1)}(\vt)^{-1}\|<\infty$. Hence, statement (a) is a consequence
of \eqref{5.1}-\eqref{5.3} and \eqref{pi}.

(b) Note that
\begin{eqnarray} \label{5.4}
     \textstyle  \frac{\partial}{\partial\vt_i}\wh{\pi}_n^{*}(\vt)\hspace*{-0.3cm} &=&\hspace*{-0.3cm} \textstyle -[ \wh \Gamma_n^{(r-1)}(\vt)]^{-1}\left[\frac{\partial}{\partial\vt_i}\wh \Gamma_n^{(r-1)}(\vt)\right][ \wh \Gamma_n^{(r-1)}(\vt)]^{-1}\wh\gamma_n^{(r-1)}(\vt)+[ \wh \Gamma_n^{(r-1)}(\vt)]^{-1}\left[\frac{\partial}{\partial\vt_i}\wh\gamma_n^{(r-1)}(\vt)\right],\notag\\
     \textstyle  \frac{\partial}{\partial\vt_i}{\pi}^{*}(\vt)\hspace*{-0.3cm} &=&\hspace*{-0.3cm} \textstyle -[ \Gamma^{(r-1)}(\vt)]^{-1}\left[\frac{\partial}{\partial\vt_i}\Gamma^{(r-1)}(\vt)\right][\Gamma^{(r-1)}(\vt)]^{-1}\gamma^{(r-1)}(\vt)+[ \Gamma^{(r-1)}(\vt)]^{-1}\left[\frac{\partial}{\partial\vt_i}\gamma^{(r-1)}(\vt)\right]. \notag\\
\end{eqnarray}
Applying \Cref{Konv Kov}(b) we receive that
\begin{eqnarray} \label{5.5}
\begin{split}
    \sup_{\vt\in\Theta}\left\|\frac{\partial}{\partial\vt_i}\wh \Gamma_n^{(r-1)}(\vt)-\frac{\partial}{\partial\vt_i}\Gamma^{(r-1)}(\vt)\right\|\stoch 0 \text{ and }
    \sup_{\vt\in\Theta}\left\|\frac{\partial}{\partial\vt_i}\wh\gamma_n^{(r-1)}(\vt)-\frac{\partial}{\partial\vt_i}\gamma^{(r-1)}(\vt)\right\|\stoch 0.
\end{split}
\end{eqnarray}
Then the same arguments as in (a) and \eqref{5.3}-\eqref{5.5} lead to statement (b).

(c) \, The proof goes in analog lines as in (a) and (b).
\end{proof}

\begin{proof}[Proof of \Cref{Corollary 5.4}]
  (a) We use the upper bound
    \begin{eqnarray*}
        \|\wh{\pi}_n^{\text{LS}}(\ov\vt_n)-\pi(\vt_0)\|\leq \sup_{\vt\in\Theta}\|\wh{\pi}_n^{\text{LS}}(\vt)-\pi(\vt)\|
            +\|\pi(\ov\vt_n)-\pi(\vt_0)\|.
    \end{eqnarray*}
    The first term converges in probability to 0 due \Cref{proposition 5.2}(a) and the second term because $\pi(\vt)$ is continuous
     (see \Cref{defbindfct}) and $\ov\vt_n\stoch \vt_0$.
     The proof of (b,c) goes on the same way.
\end{proof}

\begin{proof}[Proof of \Cref{chap8:LemmaNormRing}]
Due to \Cref{proposition 5.2}(a) we already know that
 the LS-estimator $\wh{\pi}_n^{\text{LS}}(\vt)$ is consistent.
The asymptotic normality of $\wh{\pi}_n^{\text{LS}}(\vt)$ follows in principle from \autoref{chap8:ThmNorm}
 by interpreting the least squares estimator as a particular GM-estimator with $\phi(y,u) = u$ and $\chi(x) = x-1$.
 An assumption of \autoref{chap8:ThmNorm} is that the Jacobian $\JJ_{\text{GM}}(\vt)=\nabla_\pi \mathscr{Q}_{\text{GM}}(\pi(\vt),\vt)$  is non--singular.
 For the LS-estimator this can be verified by direct calculations  because
\begin{eqnarray*}
\nabla_\pi \mathscr{Q}_{\text{LS}}(\pi,\vt)=\JJ_{\text{LS}}(\vt)
= -\frac{1}{\sigma(\vt)} \begin{pmatrix} \gamma_\vt(0) & \gamma_\vt(h) & \ldots & \gamma_\vt((r-1)h) & 0 \\
                                                                 \gamma_\vt(h) & \gamma_\vt(0) & \ldots & \gamma_\vt((r-2)h) &0 \\
                                                                   \vdots & \vdots &\ddots &\vdots & \vdots  \\
                                                                 \gamma_\vt((r-1)h) & \ldots & \ldots  & \gamma_\vt(0) & 0 \\
                                                                  0 & 0 &\ldots &0 & 2
                                                                  \end{pmatrix}.
\end{eqnarray*}
Hence, $\JJ_{\text{LS}}(\vt)$ is non--singular if and only if the upper left $r \times r$ block is non-singular. However, the upper left block is up to a positive factor the covariance matrix of the random vector $(Y_h(\vt), \ldots, Y_{rh}(\vt))$ which is  non--singular (cf. proof of \Cref{auxiliaryprop}).

 Still, we need to be careful because the  function $\phi$ and $\chi$ do not satisfy Assumptions \ref{as_H2}, \ref{as_H4} and \ref{as_H6} with respect to boundedness. However, a close inspection of the proof of \autoref{chap8:ThmNorm} reveals that the boundedness is only used at two points. First, in \Cref{AsNorm} where we deduce the finiteness of the expectation in \eqref{chap8:boundedness}. However, for the LS-estimator
\begin{align*}
\Psi_{k,i}(\vt) &= \left[Y_{(k+r)h}(\vt) - \pi_{1}(\vt_0)Y_{(k+r-1)h}(\vt) - \ldots - \pi_{r}(\vt_0)Y_{kh}(\vt) \right] Y_{(k+i-1)h}(\vt)
\intertext{for $i=1, \ldots, r$  and}
  \Psi_{k,r+1}(\vt) &= \left( \frac{Y_{(k+r)h}(\vt) - \pi_{1}(\vt)Y_{(k+r-1)h}(\vt) - \ldots - \pi_{r}(\vt)Y_{kh}(\vt)}{\sigma(\vt)} \right)^2 -1.
\end{align*}
Therefore, inequality \eqref{chap8:boundedness} follows since
the L\'{e}vy process $(L_t)_{\in\R}$ has finite $(4+\delta)$--th moment  which then transfers to $(Y_t(\vt))_{t \in \R}$ by \cite[Proposition 3.30]{MarquardtStelzer2007} and subsequently the $(2+\delta/2)$-moment of $\Psi_{k,i}(\vt)$.

Second, the boundedness assumptions are used in the proof of \Cref{Lemma 4.8} to deduce the existence of
$\nabla_\pi \mathscr{Q}_{\text{LS}}(\pi,\vt)$ and its continuity. But by the above calculations $\nabla_\pi \mathscr{Q}_{\text{LS}}(\pi,\vt)$ exists obviously and is continuous.
\end{proof}

\begin{proof}[Proof of \Cref{Theorem:indirect estimator}]
\ref{chap8:UnifConv1} and \ref{ThmIndEstAssumption2} follow from \Cref{chap8:ThmCons} and \Cref{chap8:ThmNorm}.
\ref{chap8:UnifConv} is proven in \Cref{proposition 5.2}. \ref{ThmIndEstAssumption1} is a consequence of \Cref{chap8:LemmaNormRing}.
Finally, \ref{ThmIndEstAssumption3} is derived in \Cref{Corollary 5.4}.
\end{proof}

\subsubsection*{Acknowledgement}
The authors take pleasure to thank Thiago do R$\hat{\text{e}}$go Sousa for helping with the simulation study. They thank two anonymous referees
for useful comments  improving the paper.
Financial
support by the Deutsche Forschungsgemeinschaft through the research
grant FA 809/2-2 is gratefully acknowledged.

\subsubsection*{Data Availability Statement}
The data that support the findings of this study are available from the corresponding author upon reasonable request.


\newpage

\pagebreak
\clearpage

\setcounter{page}{1}

\setcounter{section}{8}

\begin{center}
    {\normalfont\bfseries{\huge \textsf{Supporting Information}}}
\end{center}

\vspace*{1cm}


\section{Robustness properties of the indirect estimator}\label{sect:RobProp}

In this section we study the robustness properties of the indirect estimator for the CARMA parameters of $(Y_{mh})_{m\in\Z}$ where we assume that
$\widehat\pi_n=\widehat\pi^{GM}_n(\vt_0)$ is the GM-estimator satisfying
Assumptions \ref{as_D}, \ref{Assumption B},  \ref{as_H}, \ref{as_G4} for $(Y_{mh})_{m\in\Z}$ and
that $\pi(\vt_0)$   is the unique solution of \eqref{pseudo} for $(Y_{mh})_{m\in\Z}$
(a sufficient criterium for this is given in \Cref{uniquenessprop}).
Moreover, we require similarly to \ref{as_H2}:

{\it
\begin{enumerate}
\item[(E2')] $(y,u) \mapsto \phi(y,u)$ is bounded and there exists $c > 0$ such that
\begin{eqnarray*}
    \| \phi(y_1,u)y_1 - \phi(y_2,u)y_2\| &\leq& c \|y_1-y_2 \|/\min(\|y_1\|,\|y_2\|),\\
    \| \phi(y,u_1)y - \phi(y,u_2)y\| &\leq& c |u_1-u_2 |/\min(|u_1|,|u_2|).
\end{eqnarray*}
\end{enumerate}}
Under some mild assumptions on $\psi(u)$ and $w(y)$ both the Mallows estimator and the Hampel--Krasker--Welsch estimator satisfy these conditions (cf. \citet[p.1305]{boentefraimanyohai}).
For the simulation part we take some estimator
$\widehat\pi_n^S(\vt)$, e.g., the LS-estimator,
such that as $n\to\infty$,
$$\sup_{\vt \in \Theta} \| \wh{\pi}_n^{\text{S}}(\vt) - \pi(\vt) \| \stoch 0$$
holds.

\subsection{Resistance and qualitative robustness}\label{sec:qualitative:1}

To this end, let $y$ be a (infinite-length) realization of the discretely sampled CARMA process $(Y_{mh})_{m\in \Z}$. Formally, we can write that $y=(y_{mh})_{m\in\N} \in \R^{\infty}$, where $\R^{\infty}$ denotes the infinite cartesian product of $\R$ with itself. On this space, equipped with the Borel $\sigma$-field $\BB^\infty$ we denote the set of all probability measures by $\PP(\R^{\infty})$. In the following, we denote for $y \in \R^\infty$ as above by  $y^{n} = (y_h, y_{2h}, \ldots, y_{nh})$ the vector of the first $n$ coordinates. Finally, $\P_{Y^{(h)}}$ denotes the probability measure of the discrete-sampled CARMA process $(Y_{mh})_{m \in \Z}$.

\begin{definition}
Let $y \in \R^\infty$ and let $( \widehat{\vt}_n )_{n \in \N}$ be a sequence of estimators. Denote by $\widehat{\vt}_n(z^n)$ the value of $\widehat{\vt}_n$ when it is calculated using the deterministic realization $z^n \in \R^n$.
\begin{itemize}
\item[(a)] $( \widehat{\vt}_n )_{n \in \N}$ is called {\em resistant at $y$} if for every $\epsilon >0$ there exists a $\delta > 0$ such that
\begin{equation}\label{chap8:defResistant}
\sup \left\{ \| \widehat{\vt}_n( z^{n}) - \widehat{\vt}_n( w^{n}) \|: z^{n}, w^{ n} \in B_{\delta}(y^{n}) \right\} \leq \epsilon \quad \forall n \in \N,
\end{equation}
where $B_{\delta}(x)$ denotes an open ball with center $x$ and radius $\delta$ with respect to the metric
\begin{equation*}\label{chap8:defDN}
d_n( z^{n}, w^{n} ) = \inf \left\{ \epsilon: \frac{\# \{i \in \{1, \ldots, n\}: | z^{n}_i - w^{n}_i | \geq \epsilon\}}{n} \leq \epsilon \right\}.
\end{equation*}
\item[(b)] $( \widehat{\vt}_n )_{n \in \N}$ is called {\em asymptotically resistant at $y$} if for any $\epsilon > 0$ there exists a $\delta > 0$ and $N_0(\epsilon,y) \in \N$ such that \eqref{chap8:defResistant} holds for $n \geq N_0(\epsilon, y)$.
\item[(c)] For $\Q \in \PP(\R^\infty)$ we say that $( \widehat{\vt}_n )_{n \in \N}$ is {\em strongly resistant at $\Q$} if
$$\Q\left( \left\{ y \in \R^{\infty}: ( \widehat{\vt}_n )_{n \in \N} \text{  is resistant at  } y \right\} \right) = 1.$$
\item[(d)]  $( \widehat{\vt}_n )_{n \in \N}$ is called {\em asymptotically strongly resistant at $\Q$} if
$$\Q\left( \left\{ y \in \R^{\infty}: ( \widehat{\vt}_n )_{n \in \N}  \text{ is {\em asymptotically resistant at   $y$}} \right\} \right) = 1.$$
\item[(e)]  $( \widehat{\vt}_n )_{n \in \N}$ is called {\em weakly resistant at $\Q$} if for any $\epsilon>0$ there exists a $\delta>0$
 such that
$$\Q\left( \left\{ y \in \R^{\infty}: \sup \left\{ \| \widehat{\vt}_n( z^{n}) - \widehat{\vt}_n( w^{n}) \|: z^{n}, w^{ n} \in B_{\delta}(y^{n}) \right\} \leq \epsilon \right\} \right) \geq 1-\epsilon \quad \forall\,n\in\N.$$
\item[(f)] $( \widehat{\vt}_n )_{n \in \N}$ is called {\em asymptotically weakly resistant at $\Q$} if for any $\epsilon>0$ there exist a $\delta>0$
and $N(\epsilon)\in\N$ such that
$$\Q\left( \left\{ y \in \R^{\infty}: \sup \left\{ \| \widehat{\vt}_n( z^{n}) - \widehat{\vt}_n( w^{n}) \|: z^{n}, w^{ n} \in B_{\delta}(y^{n}) \right\} \leq \epsilon \right\} \right) \geq 1-\epsilon \quad \forall\,n\geq N(\epsilon).$$
\end{itemize}
\end{definition}
With this definition at hand, we want to study the question whether our indirect estimator for the parameters of a CARMA processes is resistant. We will make use of the fact that the indirect estimator consists of two
 independent parts, the GM-estimator for the parameters of the auxiliary AR representation, which deals with possible outliers in the observations, and the outlier--free estimator of the AR parameters based on simulated data.

\begin{theorem}\label{QualR1}
 The GM-estimator $(\widehat{\pi}_n^{\text{GM}}(\vt_0))_{n \in \N}$ is  strongly resistant at  $\P_{Y^{(h)}}$.
\end{theorem}
\begin{proof}
First of all,  $(\widehat{\pi}_n^{\text{GM}}(\vt_0))_{n \in \N}$ is asymptotically strongly resistant at $\P_{Y^{(h)}}$. This  follows from \citet[Theorem 5.1]{boentefraimanyohai}. The theorem requires that $\phi$ and $\chi$ fulfill \autoref{as_H}, (E2') and that the limiting equation has a unique solution, which we assumed. Moreover, $(Y_{mh})_{m
 \in \Z}$ is stationary and ergodic due to \citet[Proposition 3.34]{MarquardtStelzer2007} and fulfills \ref{as_G4}.

Next, by \citet[Lemma 5]{cox}  $\wh\pi^{\text{GM}}_n(\vt_0)$ is a continuous function of $\mathcal{Y}^{n}$ for every $n \in \N$. This in combination
 with the asymptotically strongly resistance of $(\widehat{\pi}_n^{\text{GM}}(\vt_0))_{n \in \N}$  at $\P_{Y^{(h)}}$  implies the strong resistance due to \citet[Proposition 4.2]{boentefraimanyohai}.
\end{proof}

\begin{theorem}\label{chap8:ASRThm}
The indirect estimator $(\widehat{\vt}_n^{\text{Ind}})_{n \in \N}$ is weakly  resistant  and asymptotically weakly  resistant at  $\P_{Y^{(h)}}$.
\end{theorem}
\begin{proof}
Let $\epsilon>0$. Since $\QQQ_{\text{Ind}} $ has a unique minimum in $\vt_0$
\begin{eqnarray*}
    \eta:=3^{-1}\inf_{\|\vt-\vt_0\|>\epsilon/4}\QQQ_{\text{Ind}}(\vt)>0.
\end{eqnarray*}
Moreover, the map $x \mapsto x^T \Omega x$ is continuous and hence, uniformly continuous on the compact set $\Pi'=\pi(\Theta)$.
Thus, there exists an $\epsilon'>0$ such that
\begin{eqnarray*} \label{6.2}
    \sup_{\pi,\pi'\in\Pi' \atop \|\pi-\pi'\|\leq \epsilon'}| \pi^T \Omega \pi - {\pi'}^T \Omega \pi'|\leq \frac{\eta}{8}.
\end{eqnarray*}
Define
\begin{eqnarray*}
    \Omega_0(\epsilon,\delta):=\bigcap_{n\in\N}\left\{ y \in \R^{\infty}: \sup \left\{ \| \widehat{\pi}_n( z^{n}) - \widehat{\pi}_n( w^{n}) \|: z^{n}, w^{ n} \in B_{\delta}(y^{n}) \right\} \leq \epsilon/2 \right\}.
\end{eqnarray*}
Due to the strong resistance of the GM-estimator $\wh \pi_n=\wh\pi_n^{\text{GM}}(\vt_0)$  at $\P_{Y^{(h)}}$  given in \Cref{QualR1} and \cite[Proposition 4.1]{boentefraimanyohai}
there exists an $\delta>0$ such that for \linebreak $\wt\epsilon:=\min\{\epsilon,\epsilon',\eta/(8\|\Omega\|\sup_{\pi\in\Pi'}\|\pi\|)\}$
\begin{eqnarray*}
    \P( \Omega_0(\wt\epsilon,\delta))\geq 1-\frac{\wt\epsilon}{2}\geq 1-\frac{\epsilon}{2}.
\end{eqnarray*}
Let $y\in \Omega_0(\wt\epsilon,\delta)$ and $z^{n} \in B_{\delta}(y^{n})$. Then for any $n\in\N$,
\begin{eqnarray}
\lefteqn{\sup_{\vt \in \Theta} | \LL_{\text{Ind}} ( \vt, y^n ) - \LL_{\text{Ind}} ( \vt, z^n )|} \notag \\
&=& \sup_{\vt \in \Theta} | - 2[ \wh \pi_n(y^n) - \wh \pi_n(z^n)]^T \Omega \wh \pi_n^S(\vt)  + \wh \pi_n(y^n)^T \Omega \wh \pi_n(y^n) - \wh \pi_n(z^n)^T \Omega \wh \pi_n(z^n) | \notag \\
&\leq& \sup_{\pi \in \Pi'} 2\| \pi \| \| \Omega \| \| \wh \pi_n(y^n) - \wh \pi_n(z^n) \| + | \wh \pi_n(y^n)^T \Omega \wh \pi_n(y^n) - \wh \pi_n(z^n)^T \Omega \wh \pi_n(z^n) | \leq \frac{\eta}{4}. \label{ResistIneq}
\end{eqnarray}
On the other hand, from \autoref{chap8:IndEstThm}(a)
we know that $\sup_{\vt\in\Theta}|\LL_{\text{Ind}} ( \vt, \mathcal{Y}^n )-\QQQ_{\text{Ind}}(\vt)|\stoch 0$.
Hence, there exists an $N(\epsilon)\in\N$ such that
\begin{eqnarray*}
    \P\left(\sup_{\vt\in\Theta}\|\LL_{\text{Ind}} ( \vt, \mathcal{Y}^n )-\QQQ_{\text{Ind}}(\vt) |\leq \eta\right)>1-\frac{\epsilon}{2} \quad \forall\, n\geq N(\epsilon).
\end{eqnarray*}
Define $\Omega_n:=\{y\in\R^{\infty}:\sup_{\vt\in\Theta}|\LL_{\text{Ind}} ( \vt, y^n )-\QQQ_{\text{Ind}}(\vt) |\leq \eta\}$ and let $y\in\Omega_n\cap  \Omega_0(\wt\epsilon,\delta)$ and $n\geq N(\epsilon)$.
Then
\begin{eqnarray} \label{ResistIneq3}
    |\LL_{\text{Ind}} ( \vt_0, y^n )|=|\LL_{\text{Ind}} ( \vt_0, y^n )-\QQQ_{\text{Ind}}(\vt_0)|\leq \eta,
\end{eqnarray}
and with the definition of $\eta$ we receive
\begin{equation}\label{ResistIneq2}
\inf_{|\vt - \vt_0| \geq \frac{\epsilon}{4} } | \LL_{\text{Ind}} ( \vt, y^n ) |\geq
    \inf_{|\vt - \vt_0| \geq \frac{\epsilon}{4} } | \QQQ_{\text{Ind}} ( \vt) |-\sup_{\vt\in\Theta}|\LL_{\text{Ind}} ( \vt, y^n )-\QQQ_{\text{Ind}}(\vt) |
        \geq 3\eta-\eta=2\eta.
\end{equation}
Since $\wh \vt_n^{\text{Ind}}(y^n)$ minimizes $\LL_{\text{Ind}} ( \vt, y^n )$ we can deduce from \eqref{ResistIneq3} and
\eqref{ResistIneq2} that
\begin{equation}\label{ResistIneq4}
\| \wh \vt_n^{\text{Ind}}(y^n) - \vt_0 \| < \frac{\epsilon}{4}.
\end{equation}
Let $z^{n} \in B_{\delta}(y^{n})$.  Due to \eqref{ResistIneq} and \eqref{ResistIneq3} we obtain
\begin{eqnarray} \label{ResistIneq6}
    | \LL_{\text{Ind}} ( \vt_0, z^n ) |\leq \sup_{\vt\in\Theta} | \LL_{\text{Ind}} ( \vt, z^n ) - \LL_{\text{Ind}} ( \vt, y^n )| + |\LL_{\text{Ind}} ( \vt_0, y^n ) | \leq \frac{\eta}{4}+\eta = \frac{5\eta}{4}.
\end{eqnarray}
Likewise, \eqref{ResistIneq} and \eqref{ResistIneq2} give us that
\begin{align} \label{ResistIneq7}
\inf_{|\vt - \vt_0| \geq \frac{\epsilon}{4} } | \LL_{\text{Ind}} ( \vt, z^n ) |
&\geq \inf_{|\vt - \vt_0| \geq \frac{\epsilon}{4} } | \LL_{\text{Ind}} ( \vt, y^n ) | - \sup_{\vt\in\Theta} | \LL_{\text{Ind}} ( \vt, z^n ) - \LL_{\text{Ind}} ( \vt, y^n ) | \geq \frac{7\eta}{4}.
\end{align}
Since $\wh \vt_n^{\text{Ind}}(z^n)$ minimizes $\LL_{\text{Ind}} ( \vt, z^n )$  we can conclude from \eqref{ResistIneq6} and \eqref{ResistIneq7} that
\begin{equation} \label{ResistIneq5}
\| \wh \vt_n^{\text{Ind}}(z^n) - \vt_0 \| < \frac{\epsilon}{4}.
\end{equation}
Finally, \eqref{ResistIneq4} and \eqref{ResistIneq5} result in
$$\|\wh \vt_n^{\text{Ind}}(z^n) - \wh \vt_n^{\text{Ind}}(y^n) \| \leq \| \wh \vt_n^{\text{Ind}}(z^n) - \vt_0 \| + \| \wh \vt_n^{\text{Ind}}(y^n) - \vt_0 \| < \frac{\epsilon}{2}.$$
To summarize,  for $n\geq N(\epsilon)$ we have $\P_{Y^{(h)}}( \Omega_0(\wt\epsilon,\delta) \cap \Omega_n) \geq 1-\epsilon$, and
for  $y\in  \Omega_0(\wt\epsilon,\delta) \cap \Omega_n$ and $z^n,w^n \in B_{\delta}(y^{n})$ we have
$$\| \wh{\vt}_n^{\text{Ind}} ( z^{n} ) - \wh{\vt}_n^{\text{Ind}} ( w^{n} ) \| \leq \| \wh{\vt}_n^{\text{Ind}} ( z^{n} ) - \wh{\vt}_n^{\text{Ind}} ( y^{n} ) \| +
\| \wh{\vt}_n^{\text{Ind}} ( y^{n} ) - \wh{\vt}_n^{\text{Ind}} ( w^{n} ) \|< \epsilon.$$
This gives the asymptotically weakly resistance of  $(\widehat{\vt}_n^{\text{Ind}})_{n \in \N}$  at $\P_{Y^{(h)}}$.

 By definition, $\wh{\vt}_n^{\text{Ind}}$ depends on $\mathcal{Y}^n$  through a continuous function applied to $\wh\pi^{\text{GM}}_n(\vt_0)$  and therefore, $\widehat{\vt}_n^{\text{Ind}}$ is a continuous function in $\mathcal{Y}^n$. This and the asymptotically weakly resistance of $(\widehat{\vt}_n^{\text{Ind}})_{n \in \N}$ at $\P_{Y^{(h)}}$
  imply the weakly resistance at $\P_{Y^{(h)}}$ by \citet[Proposition 4.2]{boentefraimanyohai}.
\end{proof}

\begin{remark}
If the stronger version $\sup_{\vt \in \Theta} \| \wh{\pi}_n^{\text{S}}(\vt) - \pi(\vt) \| \to 0$ $\Pas$ holds then it is possible to show on a similar
way that the indirect estimator $(\widehat{\vt}_n^{\text{Ind}})_{n \in \N}$ is even  strongly resistant.
\end{remark}

As already mentioned, one could also define qualitative robustness of a sequence of estimators by demanding that the distribution of the estimator does not change too much when the data is changed slightly. To make this notion explicit, we first define a pseudometric for measures on  metric spaces.
\begin{definition}
For a metric space $(M, d)$ with Borel sets $\BB(M)$, the Prokhorov distance $\pi_d$ between two measures $\mu, \nu$ on $\BB(M)$ with respect to $d$ is defined as
$$\pi_{d} (\mu, \nu) := \inf \{ \epsilon > 0: \mu(A) \leq \nu( \{ x \in M: d( x , A) < \epsilon \} ) + \epsilon \;\; \forall\, A \in \BB(M) \}.$$
\end{definition}
This pseudometric is a key component of the definition of qualitative robustness.
\begin{definition}
Let ${d_{\Theta}}$ be a metric on $\Theta$ and
let $\rho_n$ be a pseudometric on $\PP(\R^n)$, $n \in \N$. For $\P \in \PP( \R^\infty)$  denote by $\P_n$ the $n$--th order marginal of $\P$.
Finally,  $\P_{\widehat{\vt}_n} \in \PP(\Theta)$  is the distribution of the estimator $\widehat{\vt}_n$ under $\P_n$.
Then the sequence of estimators $( \widehat{\vt}_n )_{n \in \N}$ is called {\em $\rho_n$--robust at $\P$} if for every $\epsilon > 0$ there exists a $\delta > 0$ such that for every $\Q_n \in \PP(\R^n)$ with $\rho_n ( \P_n, \Q_n) < \delta$:
$$\pi_{d_{\Theta}} ( \P_{\widehat{\vt}_n}, \Q_{\widehat{\vt}_n} ) \leq \epsilon.$$
\end{definition}

As shown in \citet[Theorem 3.1]{boentefraimanyohai}, this is a direct generalization of the definition of $\pi$--robustness given by \citet{hampel} for i.i.d. processes.

\begin{theorem}
On $\R^n$ we define the metric
\begin{eqnarray*}
    d_n(x^n,y^n)=\inf\left\{\epsilon:\,\#\{i:|x_i-y_i|\geq \epsilon\}/n\leq \epsilon\right\}
\end{eqnarray*}
and use the Prokhorov  distance   with respect to $d_n$,
\begin{eqnarray*}
\pi_{d_n}( \P_n, \Q_n)
= \inf \{ \epsilon > 0: \P_n(A) \leq \Q_n( \{ x^n \in \R^n: d_n( x^n , A) < \epsilon \} ) + \epsilon \,\, \forall\, A \, \in \BB(\R^n) \}
\end{eqnarray*}
on $\BB(\R^n)$.
Then the sequence of estimators $(\widehat{\vt}_n^{\text{Ind}})_{n \in \N}$ is $\pi_{d_n}$--robust at $\P_{Y^{(h)}}$.
\end{theorem}
Regarding the metric $d_n$ two points are close if all coordinates except a small fraction are close.
\begin{proof}
Due to \Cref{chap8:ASRThm} the estimator $(\wh{\vt}_n^{\text{Ind}})_{n \in \N}$ is weakly resistant at $\P_{Y^{(h)}}$.
A conclusion of \citet[Theorem 4.2(i)]{boentefraimanyohai} is then the $\pi_{d_n}$--robustness of $(\wh{\vt}_n^{\text{Ind}})_{n \in \N}$ at $\P_{Y^{(h)}}$.
\end{proof}

In summary, we can say that our indirect estimator is weakly resistant at $\P_{Y^{(h)}}$ as well as $\pi_{d_n}$--robust. This is in contrast to, e.g., M--estimators, which are not qualitatively robust even in the case of linear regression (cf. \citet[p.8]{maronnayohai}).

\subsection{The influence functional}\label{InfFunc:1}
In addition to the assumptions given at the beginning of the Supporting Information we assume throughout this section
that there exists a unique solution $\pi^{\text{GM}}(\vt_0^\gamma)$   of \eqref{pseudo} for $(Y_{mh}^\gamma)_{m\in\Z}=(Y_{mh}^\gamma(\vt_0))_{m\in \Z}$ for any $0\leq\gamma\leq 1$, and $\JJ_{\text{GM}}(\vt_0)$ is non-singular.
We denote the probability measure associated to the distribution of $(Y_{mh}^\gamma)_{m\in \Z}$ by $\P^{\gamma}_{Y^{(h)}}$ for $0 \leq \gamma \leq 1$. Note that $\gamma = 0$ corresponds to the case where there are no outliers, i.e., we can observe the nominal process without error and then write $\P^0_{Y^{(h)}} = \P_{Y^{(h)}}$. Similarly $\P_Z$ is the distribution of $(Z_m)_{m\in\Z}$ and $\P_W$ is the distribution of $(W_m)_{m\in\Z}$. We write
$\{ \P^{\gamma}_{Y^{(h)}} \} := \{ \P^{\gamma}_{Y^{(h)}}, 0 \leq \gamma \leq 1\} \subseteq \PP(\R^\infty)$
and introduce the statistical functional
\begin{align*}
T_{\text{GM}}: \{ \P^{\gamma}_{Y^{(h)}} \} \to \Pi  \quad\text{ as } \quad
               \P^{\gamma}_{Y^{(h)}} \mapsto \pi^{\text{GM}}(\vt^\gamma_0).
\end{align*}
Then, the definition of the influence functional for the GM-estimator is
\begin{equation}\label{chap8:IFAR}
IF_{\text{GM}}(\P_Z, \{ \P^{\gamma}_{Y^{(h)}} \} ) := \lim_{\gamma \to 0} \frac{ T_{\text{GM}} ( \P^{\gamma}_{Y^{(h)}}) - T_{\text{GM}}(\P_{Y^{(h)}})}{\gamma} =\lim_{\gamma \to 0} \frac{\pi^{\text{GM}}(\vt_0^\gamma) - \pi(\vt_0)}{\gamma}
\end{equation}
whenever this limit is well--defined.
Note that the influence functional depends on the whole ``arc'' of contaminated measures $\{ \P^{\gamma}_{Y^{(h)}}, 0 \leq \gamma \leq 1\}$. This is the most important difference to the definition used by \citet{kunsch}, because in that paper the approximation $\P^\gamma_{Y^{(h)}} = (1-\gamma) \P_{Y^{(h)}} + \gamma \nu$ for some fixed $\nu \in \PP(\R^\infty)$ is used (\citet[Eq. (1.11)]{kunsch}).
 The influence functional measures the effect of an infinitesimal contamination of the true process by the process $(Z_m)$ on the asymptotic estimate defined via the functional $T_{\text{GM}}$.

In a similar vein, we can define the influence functional for the estimation of the parameter $\vt_0$ of our CARMA process. Analogous to $T_{\text{GM}}$, we first define a suitable statistical functional
\begin{align*}
T_{\text{Ind}}: \{ \P^{\gamma}_{Y^{(h)}} \} &\to \Theta  \quad\text{ as } \quad
               \P^{\gamma}_{Y^{(h)}} \mapsto \vt_0^{\text{Ind}}(\gamma) := \argmin_{\vt \in \Theta} [\pi(\vt)-\pi^{\text{GM}}(\vt_0^\gamma)]^T \Omega [\pi(\vt)-\pi^{\text{GM}}(\vt_0^\gamma)].
\end{align*}
This is the analog of $\vt_0=\argmin_{\vt \in \Theta} [\pi(\vt)-\pi(\vt_0)]^T \Omega [\pi(\vt)-\pi(\vt_0)]$ in the uncontaminated case (cf. \eqref{mathscrind}).
With this and $\vt_0^{\text{Ind}}(0)=\vt_0$ due to $\pi^{\text{GM}}(\vt_0)=\pi(\vt_0)$ the definition of the influence functional of the indirect estimator  is
$$IF_{\text{Ind}}(\P_Z, \{ \P^{\gamma}_{Y^{(h)}} \} ) = \lim_{\gamma \to 0} \frac{ T_{\text{Ind}} ( \P^{\gamma}_{Y^{(h)}}) - T_{\text{Ind}} (\P_{Y^{(h)}})}{\gamma} = \lim_{\gamma \to 0} \frac{\vt_0^{\text{Ind}}(\gamma) - \vt_0}{\gamma}.$$
We are interested in properties of this functional, in particular, if it is bounded. Boundedness of the influence functional implies that the estimate arising from the contaminated process cannot move too far away from the one in the uncontaminated case if the rate of contamination is infinitesimal.  This property is well--known for the influence functional for GM-estimators of AR processes. Since these estimators are an integral building block of the indirect estimator, one can hope that it carries over to our scenario and indeed it does, since the two functionals are proportional.
\begin{proposition} \label{theorem6.8}
Assume that  $\nabla_\vt\pi(\vt_0)$ has full column rank $N(\Theta)$.  Then  if $IF_{\text{GM}}$ exists then $IF_{\text{Ind}}$ exists and
$$IF_{\text{Ind}}(\P_Z, \{ \P^{\gamma}_{Y^{(h)}} \} ) = \HH(T_{\text{Ind}} (\P_{Y^{(h)}})) IF_{\text{GM}}(\P_Z, \{ \P^{\gamma}_{Y^{(h)}} \} ),$$
where
$$\HH(T_{\text{Ind}} (\P_{Y^{(h)}})) = \HH(\vt_0) = [ \nabla_\vt \pi(\vt_0)^T \Omega \nabla_\vt \pi(\vt_0)]^{-1} \nabla_\vt \pi(\vt_0)^T  \Omega.$$
\end{proposition}
\begin{proof}
This follows from \citet[Theorem 1]{delunagenton2} as special case.
\end{proof}
From this theorem we see that the question of the boundedness of the influence functional for the indirect estimator of a
discretely sampled CARMA process reduces to the question of the boundedness of the influence functional for the GM-estimatior of the auxiliary AR representation of the sampled CARMA
process.

\begin{theorem}\label{chap8:ThmBoundedIF}
Let the additive outlier model be given and let \autoref{as_G}  hold.
 Then there exists a constant $K > 0$ such that
\begin{equation*}
    \| IF_{\text{GM}}(\P_{W}, \{ \P^{\gamma}_{Y^{(h)}} \} ) \| \leq 2 (r+1) K \left\| \JJ_{\text{GM}}(\vt_0)^{-1} \right\|.
\end{equation*}
\end{theorem}
\begin{proof}
The plan is to apply \citet[Theorem 4.3]{martinyohai}. Therefore, we have to check that  \cite[Eq. (4.6)]{martinyohai} hold. Sufficient conditions for this equation are given in \citet[Theorem 4.2]{martinyohai} which are obviously satisfied in our case due to \ref{as_H2}, \ref{as_H6} and due to the fact that
$\pi^{\text{GM}}(\vt^\gamma_0)$ depends only on the distribution of the finite random vector $(Y_h^\gamma,\ldots,Y_{(r+1)h}^\gamma)$.
For the same reasons and with our assumption that $\JJ_{\text{GM}}(\vt_0)$ is non-singular the other conditions in \cite[Theorem 4.3]{martinyohai} are satisfied as well.
\end{proof}
As well due to  \citet[Theorem 4.3]{martinyohai} it is possible to give an explicit (but  not a very handy) representation of
$IF_{\text{GM}}(\P_{W}, \{ \P^{\gamma}_{Y^{(h)}} \} )$ and hence, though \Cref{theorem6.8} for $IF_{\text{Ind}}(\P_W, \{ \P^{\gamma}_{Y^{(h)}} \} ) $.

\section{Proof of \Cref{Lemma 5.1}} \label{Appendix:Lemma 5.1}

Before we state the proof we require some auxiliary results.

\begin{lemma} \label{Lemma A.1}
For any $i,j,l\in\{1,\ldots,N(\Theta)\}$ the following conditions hold:
\begin{itemize}
    \item[(a)] $\sup_{\vt\in\Theta}|f_\vt(u)|\leq C\e^{-\rho u}$ and $\int_0^\infty \sup_{\vt\in\Theta}f_\vt(u)^2\,d u<\infty$.
    \item[(b)] $\sup_{\vt\in\Theta}\left|\frac{\partial f_\vt(u)}{\partial {\vt_j}}\right|\leq Cu\e^{-\rho u}$ and $\int_0^\infty \sup_{\vt\in\Theta}\left(\frac{\partial f_\vt(u)}{\partial {\vt_j}}\right)^2\,d u<\infty$.
    \item[(c)] $\sup_{\vt\in\Theta}\left|\frac{\partial^2 f_\vt(u)}{\partial {\vt_j}\partial {\vt_i}}\right|\leq Cu^2\e^{-\rho u}$ and $\int_0^\infty \sup_{\vt\in\Theta}\left(\frac{\partial^2 f_\vt(u)}{\partial {\vt_j}\partial {\vt_i}}\right)^2\,d u<\infty$.
    \item[(d)] $\sup_{\vt\in\Theta}\left|\frac{\partial^3 f_\vt(u)}{\partial {\vt_j}\partial {\vt_i}\partial {\vt_l}}\right|\leq Cu^3\e^{-\rho u}$ and $\int_0^\infty \sup_{\vt\in\Theta}\left(\frac{\partial^3 f_\vt(u)}{\partial {\vt_j}\partial {\vt_i}\partial {\vt_l}}\right)^2\,d u<\infty$.
\end{itemize}
\end{lemma}
\begin{proof}
(a) \, Due to \Cref{Remark 2.1}(iii) we have that $\sup_{\vt\in\Theta}\|\e^{A_\vt u}\|\leq C\e^{-\rho u}$ and hence,
$\sup_{\vt\in\Theta} |f_\vt(u)|\leq \sup_{\vt\in\Theta}\|c_\vt\|\sup_{\vt\in\Theta}\|\e^{A_\vt u}\|\leq C\e^{-\rho u}$ using the continuity of $c_\vt$ on the compact set $\Theta$.
Finally, $\int_0^\infty \sup_{\vt\in\Theta} f_\vt(u)^2\,du<\infty$.

(b) \, A consequence of \citet{Wilcox1967}  is that
\begin{eqnarray*}
    \frac{\partial \e^{A_\vt u}}{\partial \vt_j}=\int_0^{u}\e^{A_\vt (u -s)}\left(\frac{\partial}{\partial \vt_j}A_\vt\right) \e^{A_\vt s}\,ds,
\end{eqnarray*}
and hence, $\sup_{\vt\in\Theta}\|\frac{\partial \e^{A_\vt u}}{\partial {\vt_j}}\|\leq Cu\e^{-\rho u}$. From this and
\begin{eqnarray*}
    \frac{\partial f_\vt(u)}{\partial \vt_j}=\left(\frac{\partial c_\vt}{\partial \vt_j}\right)\e^{A_\vt u}e_p
        +c_\vt\left(\frac{\partial \e^{A_\vt u}}{\partial \vt_j}\right)e_p
\end{eqnarray*}
we receive
\begin{eqnarray*}
    \int_0^\infty \sup_{\vt\in\Theta}\left(\frac{\partial f_\vt(u)}{\partial {\vt_j}}\right)^2\,d u
        \leq C \int_0^\infty u^2\e^{-2\rho u}\,du<\infty.
\end{eqnarray*}
(c,d) can be proven similarly to (a) and (b).
\end{proof}

\begin{lemma} \label{Lemma A.2}
Let $(L_t)_{t\in\R}$ be a Lévy process with $\E|L_1|^{2N^*}<\infty$ for some $N^*\in\N$ and $\E(L_1)=0$. Furthermore, let $\phi:\R\to\R$
be a measurable map with $\phi\in \bigcap_{j=1}^{N^*}L^{2j}(\R)$. Then $\E|\int_{-\infty}^{\infty}\phi(u)\,dL_u|^{2N^*}<\infty$
and there exist finite constants $c_{j_1,\ldots,j_k}$ such that
\begin{eqnarray*}
    \E\left(\int_{-\infty}^{\infty}\phi(u)\,dL_u\right)^{2N^*}=\sum_{k=1}^{N^*}\sum_{j_1+\ldots+j_k=N^* \atop j_1,\ldots,j_k\in\N_0}c_{j_1,\ldots,j_k}\left(\int_{-\infty}^\infty\phi^{2j_1}(u)\,du\right)\cdots\left(\int_{-\infty}^\infty\phi^{2j_k}(u)\,du\right).
\end{eqnarray*}
\end{lemma}
\begin{proof}
For $N^*=2$ the result was already derived in \citet[Lemma 3.2]{Cohen:Lindner}. The proof for general $N^*$ uses the same ideas.
Let $\nu$ be the Lévy measure of $(L_t)_{t\in\R}$ and $V$ its Gaussian parameter.
Define $$\psi(s)=-\frac{1}{2}V^2s^2\int_{-\infty}^{\infty}\phi^2(u)\,du+\int_{-\infty}^{\infty}\int_{-\infty}^{\infty}[\e^{is\phi(u)x}-1-is\phi(u)x]\nu(dx)\,ds,$$
and write $\psi^{(i)}(s)$ for the $i$-th derivative of $\psi(s)$.
Due to the Lévy-Khintchine formula
we get
\begin{eqnarray*}
    \xi(s):=\E\left(\exp\left(is\int_{-\infty}^{\infty}\phi(u)\,dL_u\right)\right)
    =\exp\left(\psi(s)\right).
\end{eqnarray*}
Then $\E\left(\int_{-\infty}^{\infty}\phi(u)\,dL_u\right)^{2N^*}$ is obtained as $(2N^*)$-th derivative of $\xi(s)$ at $s=0$ times $(-1)^{N^*}$.
Straightforward calculations yield that the $(2N^*)$-th derivative $\xi^{(2N^ *)}(s)$ of $\xi(s)$ has the form
\begin{eqnarray*}
    \xi^{(2N^ *)}(s)=\left(\sum_{k=1}^{2N^*}\sum_{i_1+\ldots+i_k=2N^* \atop i_1,\ldots,i_k\in\N_0}\wt c_{i_1,\ldots,i_k}\psi^{(i_1)}(s)\cdots \psi^{(i_k)}(s)\right)\exp(\psi(s)).
\end{eqnarray*}
Plugging in for $s=0$ and taking into account that $\psi(0)=1$ and $\psi^{(1)}(0)=0$ gives
\begin{eqnarray*}
     \E\left(\int_{-\infty}^{\infty}\phi(u)\,dL_u\right)^{2N^*}=(-1)^{N^*}\xi^{(2N^ *)}(0)=\sum_{k=1}^{N^*}\sum_{j_1+\ldots+j_k=N^* \atop j_1,\ldots,j_k\in\N_0} c_{j_1,\ldots,j_k}\psi^{(2j_1)}(0)\cdots \psi^{(2j_k)}(0).
\end{eqnarray*}
Finally, with
\begin{eqnarray*}
    \psi^{(2)}(0)&=&\left[-V^2-\int_{-\infty}^\infty x^2\,\nu(dx)\right]\left[\int_{-\infty}^\infty\phi^2(u)\,du\right],\\
    \psi^{(2i)}(0)&=&(-1)^{i_1}\left[\int_{-\infty}^{\infty}x^{2i}\,\nu(dx)\right]\left[\int_{-\infty}^\infty\phi^{2i}(u)\,du\right], \quad i\geq 2,
\end{eqnarray*}
we obtain the result.
\end{proof}

\begin{lemma} \label{Lemma A.3}
Let $N^*\in\N$ be such that $2N^*>N(\Theta)$ and $\E|L_1|^{2N^*}<\infty$.
For any $i,j,l\in\{1,\ldots,N(\Theta)\}$  the following maps are $\Pas$  Hölder-continuous of order
$\gamma\in\left[0,1-N(\Theta)/(2N^*)\right)$:
\begin{itemize}
    \item[(a)] $\vt\mapsto \int_0^\infty f_\vt(u)\,dL_u=:Z(\vt)$,
    \item[(b)] $\vt\mapsto \int_0^\infty \frac{\partial}{\partial\vt_j}f_\vt(u)\,dL_u=:Z^{(j)}(\vt)$,
    \item[(c)] $\vt\mapsto \int_0^\infty \frac{\partial}{\partial\vt_j\partial\vt_i}f_\vt(u)\,dL_u=:Z^{(j,i)}(\vt)$.
\end{itemize}
Moreover, $\E(\sup_{\vt\in\Theta}|Z(\vt)|^{2N^*})<\infty$ and for $U:=\sup_{0<\|\vt_1-\vt_2\|<1\atop \vt_1,\vt_2\in\Theta}\frac{|Z(\vt_1)-Z(\vt_2)|}{\|\vt_1-\vt_2\|^\gamma}$ we have $\E U^{2N^*}<\infty$.
The same is true for $Z^{(j)}(\vt)$ and $Z^{(j,i)}(\vt)$.
\end{lemma}
\begin{proof} $\mbox{}$\\
(a) \, Let $\vt_1,\vt_2\in\Theta$ and define $\phi(u):=f_{\vt_1}(u)-f_{\vt_2}(u)$.
Due to a Taylor expansion we obtain
\begin{eqnarray*}
    \phi(u)=f_{\vt_1}(u)-f_{\vt_2}(u)=\nabla_\vt f_{\wt\vt(u)}(u)(\vt_1-\vt_2)
\end{eqnarray*}
for some $\wt\vt(u)\in\Theta$ with $\|\wt\vt(u)-\vt_2\|\leq \|\vt_1-\vt_2\|$. Hence, we receive by \Cref{Lemma A.1}
\begin{eqnarray*}
    |\phi(u)|\leq \sup_{\vt\in\Theta}\|\nabla_\vt f_\vt(u)\|\|\vt_1-\vt_2\|\leq Cu\e^{-\rho u}\|\vt_1-\vt_2\|.
\end{eqnarray*}
Plugging this into  \Cref{Lemma A.2} gives
\begin{eqnarray*}
    \lefteqn{ \E(Z(\vt_1)-Z(\vt_2))^{2N^*}}\\
        &\leq& C\sum_{k=1}^{N^*}\sum_{j_1+\ldots+j_k=N^* \atop j_1,\ldots,j_k\in\N_0}|c_{j_1,\ldots,j_k}|\left(\int_{-\infty}^\infty \|\vt_1-\vt_2\|^{2j_1}u^{2j_1}\e^{-2j_1\rho}\,du\right)\cdots\left(\int_{-\infty}^\infty \|\vt_1-\vt_2\|^{2j_k}u^{2j_k}\e^{-2j_k\rho}\,du \right)\\
        & \leq &C\|\vt_1-\vt_2\|^{2N^*}.
\end{eqnarray*}
Then, an application of Kolmogorov-Chentsov Theorem (cf. \citet[Theorem 10.1]{Schilling:Partzsch})  yield  the Hölder continuity and $\E U^{2N^*}<\infty$.
Since $\Theta$ is compact, some straightforward calculations yield $\E(\sup_{\vt\in\Theta}|Z(\vt)|^{2N^*})<\infty$. 

The proofs of (b)-(c) are similarly to the proof of (a) and thus, skipped.
\end{proof}

\begin{lemma} \label{Lemma A.4}
Let $[a,b]\subseteq \R$ be a bounded interval, $(L_t)_{t\in\R}$ be a Lévy process with finite second moments and let $g:[a,b]\times\R\to\R$
be differentiable in the first component with derivative $\frac{\partial g(\vt,u)}{\partial \vt}$. Moreover, assume
\begin{itemize}
    \item[(a)] ${\displaystyle \frac{\partial g(\vt,u)}{\partial \vt}}$ is bounded $\mathcal{B}([a,b])\otimes \mathcal{B}([-u_1,u_2])$-measurable for all $u_1,u_2>0$,
    \item[(b)] ${\displaystyle \lim_{M\to\infty}\sup\limits_{\vt\in[a,b]} \int_{|u|>M}\left|\frac{\partial g(\vt,u)}{\partial \vt}\right|\,du=0}$ and ${\displaystyle \lim\limits_{|u|\to\infty}\sup_{\vt\in[a,b]}\left|\frac{\partial g(\vt,u)}{\partial \vt}\right|=0}$,
    \item[(c)] ${\displaystyle \vt\mapsto \int_{-\infty}^{\infty}  g(\vt,u) \,d L_u}$ is $\Pas$ continuous,
    \item[(d)] ${\displaystyle \vt\mapsto \int_{-\infty}^{\infty} \frac{\partial g(\vt,u)}{\partial \vt}\,d L_u}$ is $\Pas$ continuous.
\end{itemize}
Then, outside a $\P$-zero set, $Z(\vt,\omega):=\int_{-\infty}^{\infty}  g(\vt,u) \,d L_u(\omega)$ is continuous differentiable over the interval $(a,b)$ and
\begin{eqnarray*}
       \frac{\partial }{\partial \vt} \int_{-\infty}^{\infty}  g(\vt,u) \,d L_u(\omega)=\int_{-\infty}^{\infty} \frac{\partial g(\vt,u)}{\partial \vt}\,d L_u(\omega).
\end{eqnarray*}
\end{lemma}
\begin{proof}
An application of Fubini's Theorem for Lévy-integrals (see \cite[Theorem 2.4]{Brockwell:Schlemm})  gives that
for $\wt \vt \in[a,b]$
\begin{eqnarray*}
    \int_a^{\wt \vt} \int_{-\infty}^{\infty} \frac{\partial g(\vt,u)}{\partial \vt}\,d L_u\,d\vt=\int_{-\infty}^{\infty} \int_a^{\wt \vt} \frac{\partial g(\vt,u)}{\partial \vt} \,d\vt\,d L_u  \quad \Pas
\end{eqnarray*}
The remaining of the proof follows the same line as \citet[Theorem 2.2]{Hutton:Nelson}.
\end{proof}

\begin{proof}[Proof of \Cref{Lemma 5.1}]
A combination of \Cref{Lemma A.1}-\Cref{Lemma A.4} result in the MA representation.
A conclusion of the MA representation and
\citet[Theorem 3.5]{fuchsstelzer} is that the process is mixing and therefore, in particular, ergodic.
\end{proof}

\newpage

\section{Simulation results} \label{sec:simulationSI}

\begin{table}[!hp]
 \small
\centering
\begin{tabular}{|l||c|c|c||c|c|c||c|c|c|}
\hline
 \; & \multicolumn{9}{c|}{Indirect estimation} \\
\hline
\; & \multicolumn{9}{c|}{$\xi=0,\gamma=0$ (uncontaminated)} \\
\hline
   \; & \multicolumn{3}{c||}{$r=5$} & \multicolumn{3}{c||}{$r=6$}  &  \multicolumn{3}{c|}{$r=7$}\\
 \hline
  \; & Mean & Bias & Var & Mean & Bias & Var  & Mean & Bias & Var\\
  \hline
    $\vt_1=-1$ &  -1.02333 & -0.02333 & 0.00515  & -1.02802 & -0.02802 & 0.00501 & -1.02079 & -0.02079 & 0.00681\\
 \hline
    $\vt_2=-2$ & -1.98112 & 0.01888 & 0.00905 & -1.97851 & 0.02149 & 0.00865  & -2.00511 & -0.00511 & 0.02403\\
 \hline
    $\vt_3=-2$ &  -2.00188 & -0.00188 & 0.01286  & -2.00098 & -0.00098 & 0.01263  & -1.99128 & 0.00872 & 0.01429\\
 \hline
    $\vt_4=0$ &   0.01253 & 0.01253 & 0.00436 & 0.01781 & 0.01781 & 0.00470  & 0.01541 & 0.01541 & 0.00552\\
 \hline
    $\vt_5=1$  & 1.00248 & 0.00248 & 0.00356 & 0.99551 & -0.00449 & 0.00198 & 0.98748 & -0.01252 & 0.00830\\
\hline
\hline
 \; & \multicolumn{9}{c|}{$\xi=5,\gamma=0.1$} \\
\hline
   \; & \multicolumn{3}{c||}{$r=5$} & \multicolumn{3}{c||}{$r=6$}  &  \multicolumn{3}{c|}{$r=7$}\\
 \hline
  \; & Mean & Bias & Var & Mean & Bias & Var  & Mean & Bias & Var\\
  \hline
    $\vt_1=-1$ & -1.0031 & -0.0031 & 0.0131 & -1.1226 & -0.1226 & 0.4248 & -1.0670 & -0.0670 & 0.1802\\
 \hline
    $\vt_2=-2$ & -2.0444 & -0.0444 & 0.0606 & -2.2005 & -0.2005 & 1.0920 & -2.1946 & -0.1946 & 0.6932 \\
 \hline
    $\vt_3=-2$ & -1.9969 & 0.0031 & 0.0325 & -2.0514 & -0.0514 & 0.2113 & -2.0483 & -0.0483 & 0.1112\\
 \hline
    $\vt_4=0$ & -0.0157 & -0.0157 & 0.0070 & -0.0279 & -0.0279 & 0.0130 & -0.0104 & -0.0104 & 0.0109 \\
 \hline
    $\vt_5=1$  & 0.9147 & -0.0853 & 0.0243 & 0.8876 & -0.1124 & 0.0536 & 0.8370 & -0.1630 & 0.1106  \\
\hline
\hline
 \; & \multicolumn{9}{c|}{$\xi=10,\gamma=0.1$} \\
\hline
   \; & \multicolumn{3}{c||}{$r=5$} & \multicolumn{3}{c||}{$r=6$}  &  \multicolumn{3}{c|}{$r=7$}\\
 \hline
  \; & Mean & Bias & Var & Mean & Bias & Var  & Mean & Bias & Var\\
  \hline
    $\vt_1=-1$ & -1.0031 & -0.0031 & 0.0131 & -1.0204 & -0.0204 & 0.0137 & -1.0319 & -0.0319 & 0.0445 \\
 \hline
    $\vt_2=-2$ & -2.0446 & -0.0446 & 0.0608 & -2.0581 & -0.0581 & 0.0980 & -2.1881 & -0.1881 & 0.7340\\
 \hline
    $\vt_3=-2$ & -1.9966 & 0.0034 & 0.0325 & -1.9994 & 0.0006 & 0.0341 & -2.0324 & -0.0324 & 0.0532\\
 \hline
    $\vt_4=0$ &  -0.0157 & -0.0157 & 0.0070 & -0.0193 & -0.0193 & 0.0104 & -0.0082 & -0.0082 & 0.0098\\
 \hline
    $\vt_5=1$  &  0.9144 & -0.0856 & 0.0243 & 0.9007 & -0.0993 & 0.0466 & 0.8583 & -0.1417 & 0.0747\\
\hline
\hline
 \; & \multicolumn{9}{c|}{$\xi=5,\gamma=1/6$} \\
\hline
   \; & \multicolumn{3}{c||}{$r=5$} & \multicolumn{3}{c||}{$r=6$}  &  \multicolumn{3}{c|}{$r=7$}\\
 \hline
  \; & Mean & Bias & Var & Mean & Bias & Var  & Mean & Bias & Var\\
  \hline
    $\vt_1=-1$ & -0.9476 & 0.0524 & 0.0426 & -1.0935 & -0.0935 & 0.1174 & -0.9831 & 0.0169 & 0.0561 \\
 \hline
    $\vt_2=-2$ &  -2.1688 & -0.1688 & 0.1375 & -2.4429 & -0.4429 & 1.1446 & -2.4168 & -0.4168 & 0.5622  \\
 \hline
    $\vt_3=-2$ & -1.9481 & 0.0519 & 0.0469 & -1.9418 & 0.0582 & 0.1021 & -1.9096 & 0.0904 & 0.1023 \\
 \hline
    $\vt_4=0$ &  -0.0491 & -0.0491 & 0.0101 & -0.0743 & -0.0743 & 0.0231 & -0.0481 & -0.0481 & 0.0192\\
 \hline
    $\vt_5=1$  & 0.6996 & -0.3004 & 0.0705 & 0.5946 & -0.4054 & 0.1107 & 0.5980 & -0.4020 & 0.1281\\
\hline
\hline
 \; & \multicolumn{9}{c|}{$\xi=5,\gamma=0.25$} \\
\hline
   \; & \multicolumn{3}{c||}{$r=5$} & \multicolumn{3}{c||}{$r=6$}  &  \multicolumn{3}{c|}{$r=7$}\\
 \hline
  \; & Mean & Bias & Var & Mean & Bias & Var  & Mean & Bias & Var\\
  \hline
    $\vt_1=-1$ & -0.6035 & 0.3965 & 7.5738 & -0.2362 & 0.7638 & 0.1676 & -0.2224 & 0.7776 & 0.1255\\
 \hline
    $\vt_2=-2$ & -3.8476 & -1.8476 & 25.0781 & -3.1737 & -1.1737 & 0.7282 & -3.1487 & -1.1487 & 0.6338 \\
 \hline
    $\vt_3=-2$ &  -6.0640 & -4.0640 & 417.3265 & -3.2305 & -1.2305 & 5.8199 & -3.1563 & -1.1563 & 4.8466 \\
 \hline
    $\vt_4=0$ & 2.1462 & 2.1462 & 11.3420 & 1.7850 & 1.7850 & 2.7106 & 1.6978 & 1.6978 & 1.7545  \\
 \hline
    $\vt_5=1$  & 0.7653 & -0.2347 & 26.1110 & 1.4551 & 0.4551 & 0.8809 & 1.4777 & 0.4777 & 0.7939\\
\hline
\end{tabular}
\caption{Estimation results for a CARMA$(3,1)$ process with parameter $\vartheta_0=(\vt_1,\vt_2,\vt_3,\vt_4,\vt_5)$ driven by a Brownian motion with  $n=1000$.}\label{Table 7A}
\end{table}

\begin{table}[!hp] \small
\centering
\begin{tabular}{|l||c|c|c||c|c|c||c|c|c|}
\hline
 \; & \multicolumn{9}{c|}{Indirect estimation} \\
\hline
 \; & \multicolumn{9}{c|}{$\xi=0,\gamma=0$ (uncontaminated)} \\
\hline
   \; & \multicolumn{3}{c||}{$n=200$} & \multicolumn{3}{c||}{$n=500$}  &  \multicolumn{3}{c|}{$n=5000$}\\
 \hline
  \; & Mean & Bias & Var & Mean & Bias & Var  & Mean & Bias & Var\\
  \hline
    $\vt_1=-1$ & -1.32232 & -0.32232 & 0.94063  & -1.06069 & -0.06069 & 0.01299 & -1.00534 & -0.00534 & 0.00046\\
 \hline
    $\vt_2=-2$ & -2.28763 & -0.28763 & 1.29387 & -2.03660 & -0.03660 & 0.02302 & -1.99576 & 0.00424 & 0.00034 \\
 \hline
    $\vt_3=-2$ & -2.11439 & -0.11439 & 0.42097 & -1.94386 & 0.05614 & 0.01541 & -1.99930 & 0.00070 & 0.00103\\
 \hline
    $\vt_4=0$ & 0.00287 & 0.00287 & 0.02619 & -0.03115 & -0.03115 & 0.00969   & 0.00089 & 0.00089 & 0.00035 \\
 \hline
    $\vt_5=1$  & 0.88711 & -0.11289 & 0.08930 & 0.93227 & -0.06773 & 0.02011 & 0.99864 & -0.00136 & 0.00027  \\
\hline
 \; & \multicolumn{9}{c|}{$\xi=5,\gamma=0.1$} \\
\hline
   \; & \multicolumn{3}{c||}{$n=200$} & \multicolumn{3}{c||}{$n=500$}  &  \multicolumn{3}{c|}{$n=5000$}\\
 \hline
  \; & Mean & Bias & Var & Mean & Bias & Var  & Mean & Bias & Var\\
  \hline
   $\vt_1=-1$ &  -1.0290 & -0.0290 & 0.0589 & -1.0561 & -0.0561 & 0.0206 & -1.0011 & -0.0011 & 0.0013 \\
   \hline
    $\vt_2=-2$ & -2.2015 & -0.2015 & 0.4188 & -2.1196 & -0.1196 & 0.1148 & -2.0141 & -0.0141 & 0.0026\\
    \hline
    $\vt_3=-2$ & -2.0014 & -0.0014 & 0.1333 & -1.9090 & 0.0910 & 0.0292 & -1.9866 & 0.0134 & 0.0029 \\
    \hline
    $\vt_4=0$ & 0.0044 & 0.0044 & 0.0302 & -0.0641 & -0.0641 & 0.0077 & -0.0140 & -0.0140 & 0.0004 \\
    \hline
    $\vt_51$ &  0.7710 & -0.2290 & 0.1279 & 0.8764 & -0.1236 & 0.0463 & 0.9651 & -0.0349 & 0.0035\\
    \hline
 \; & \multicolumn{9}{c|}{$\xi=10,\gamma=0.1$} \\
\hline
   \; & \multicolumn{3}{c||}{$n=200$} & \multicolumn{3}{c||}{$n=500$}  &  \multicolumn{3}{c|}{$n=5000$}\\
 \hline
  \; & Mean & Bias & Var & Mean & Bias & Var  & Mean & Bias & Var\\
  \hline
   $\vt_1=-1$ &-1.0244 & -0.0244 & 0.0611 & -1.0559 & -0.0559 & 0.0206 & -1.0026 & -0.0026 & 0.0012  \\
   \hline
    $\vt_2=-2$ & -2.2144 & -0.2144 & 0.4986 & -2.1197 & -0.1197 & 0.1158 & -2.0113 & -0.0113 & 0.0022  \\
    \hline
    $\vt_3=-2$ & -1.9920 & 0.0080 & 0.1456 & -1.9089 & 0.0911 & 0.0292 & -1.9880 & 0.0120 & 0.0028\\
    \hline
    $\vt_4=0$ &  0.0026 & 0.0026 & 0.0296 & -0.0641 & -0.0641 & 0.0077 & -0.0141 & -0.0141 & 0.0004\\
    \hline
    $\vt_5=1$ & 0.7643 & -0.2357 & 0.1375 & 0.8760 & -0.1240 & 0.0467 & 0.9672 & -0.0328 & 0.0033\\
    \hline
 \hline
 \; & \multicolumn{9}{c|}{$\xi=5,\gamma=1/6$} \\
\hline
   \; & \multicolumn{3}{c||}{$n=200$} & \multicolumn{3}{c||}{$n=500$}  &  \multicolumn{3}{c|}{$n=5000$}\\
 \hline
  \; & Mean & Bias & Var & Mean & Bias & Var  & Mean & Bias & Var\\
  \hline
   $\vt_1=-1$ & 0.9337 & 0.0663 & 0.1275 & -1.0080 & -0.0080 & 0.0325 & -0.9521 & 0.0479 & 0.0085\\
   \hline
    $\vt_2=-2$ &-2.3455 & -0.3455 & 0.7441 & -2.2549 & -0.2549 & 0.1925 & -2.0912 & -0.0912 & 0.0263\\
    \hline
    $\vt_3=-2$ & -2.0464 & -0.0464 & 0.6561 & -1.8934 & 0.1066 & 0.0527 & -1.9667 & 0.0333 & 0.0114\\
    \hline
    $\vt_4=0$ &0.0009 & 0.0009 & 0.0850 & -0.0617 & -0.0617 & 0.0119 & -0.0370 & -0.0370 & 0.0026 \\
    \hline
    $\vt_5=1$ & 0.5380 & -0.4620 & 0.2457 & 0.6533 & -0.3467 & 0.1024 & 0.7963 & -0.2037 & 0.0301\\
    \hline
 \hline
 \; & \multicolumn{9}{c|}{$\xi=5,\gamma=0.25$} \\
\hline
   \; & \multicolumn{3}{c||}{$n=200$} & \multicolumn{3}{c||}{$n=500$}  &  \multicolumn{3}{c|}{$n=5000$}\\
 \hline
  \; & Mean & Bias & Var & Mean & Bias & Var  & Mean & Bias & Var\\
  \hline
   $\vt_1=-1$ &-0.4340 & 0.5660 & 1.5176 & -0.1225 & 0.8775 & 0.0497 & -0.0498 & 0.9502 & 0.0262  \\
   \hline
    $\vt_2=-2$ &-3.0477 & -1.0477 & 0.8088 & -2.9516 & -0.9516 & 0.1455 & -2.8767 & -0.8767 & 0.0451    \\
    \hline
    $\vt_3=2$ & -4.0795 & -2.0795 & 55.4861 & -2.5088 & -0.5088 & 0.6350 & -2.3780 & -0.3780 & 0.6102 \\
    \hline
    $\vt_4=0$ &  1.9005 & 1.9005 & 18.3389 & 1.3724 & 1.3724 & 0.3679 & 1.2918 & 1.2918 & 0.3086\\
    \hline
    $\vt_5=1$ &  1.0938 & 0.0938 & 3.3676 & 1.7279 & 0.7279 & 0.1833 & 1.8691 & 0.8691 & 0.0508\\
    \hline
 \hline
\end{tabular}
\caption{Estimation results for a CARMA$(3,1)$ process with parameter $\vartheta_0=(\vt_1,\vt_2,\vt_3,\vt_4,\vt_5)$
driven by a Brownian motion with $r=5$.}\label{Table 5A}
\end{table}

\begin{table}[h] \small
\centering
\begin{tabular}{|l|c|c|c||c|c|c|c|}
\hline
 \; & \multicolumn{6}{c|}{$\xi=0, \gamma =0$ (uncontaminated)} \\
  \hline
   \; & \multicolumn{3}{c||}{QMLE} & \multicolumn{3}{c|}{Indirect}  \\
 \hline
  \; & Mean & Bias & Var & Mean & Bias & Var \\
  \hline
   $\vt_1=-0.1$ & -0.10186 & -0.00186 & 0.00064 & -0.12204 & -0.02204 & 0.00365 \\
   \hline
    $\vt_2=-1.2$ &  -1.20025 & -0.00025 & 0.00025 & -1.23417 & -0.03417 & 0.00846\\
    \hline
    $\vt_3=-2.1$ & -2.09793 & 0.00207 & 0.00019  & -2.04382 & 0.05618 & 0.02458 \\
    \hline
    $\vt_4=0$ &  0.00080 & 0.00080 & 0.00001 & -0.00368 & -0.00368 & 0.00212 \\
    \hline
    $\vt_5= 1$ & 1.00766 & 0.00766 & 0.00051 & 0.98731 & -0.01269 & 0.01179\\
    \hline
\end{tabular}
\caption{Estimation results for a CARMA$(3,1)$ process with parameter $\vartheta_0=(\vt_1,\vt_2,\vt_3,\vt_4,\vt_5)$ driven by a Brownian motion with $n=1000$ and $r=5$.}\label{Table 3B}
\end{table}

\end{document}